\documentclass[a4paper]{amsart}
\usepackage{amsrefs}
\usepackage{amssymb,mathrsfs}
\usepackage[all]{xy}
\newdir{ >}{{}*!/-5pt/@{>}} 
\SelectTips{cm}{} 

\title{On the motivic Segal conjecture}
\author{Thomas Gregersen}
\address{Department of Technology Systems, University of Oslo, Norway}
\email{thomas.gregersen@its.uio.no}
\author{John Rognes}
\address{Department of Mathematics, University of Oslo, Norway}
\email{rognes@math.uio.no}

\date{August 8th 2023}

\newtheorem{theorem}{Theorem}[section]
\newtheorem{proposition}[theorem]{Proposition}
\newtheorem{lemma}[theorem]{Lemma}
\newtheorem{corollary}[theorem]{Corollary}
\theoremstyle{definition}
\newtheorem{definition}[theorem]{Definition}
\newtheorem{notation}[theorem]{Notation}
\theoremstyle{remark}
\newtheorem{example}[theorem]{Example}
\newtheorem{remark}[theorem]{Remark}

\DeclareMathOperator{\cl}{cl}
\DeclareMathOperator{\gr}{gr}
\DeclareMathOperator{\del}{del}
\DeclareMathOperator{\loc}{loc}
\DeclareMathOperator{\res}{res}
\DeclareMathOperator{\sh}{sh}
\DeclareMathOperator{\wt}{wt}

\DeclareMathOperator*{\colim}{colim}
\DeclareMathOperator*{\hocolim}{hocolim}

\DeclareMathOperator{\Ext}{Ext}
\DeclareMathOperator*{\holim}{holim}
\DeclareMathOperator{\Hom}{Hom}
\DeclareMathOperator*{\limone}{lim^1}
\DeclareMathOperator{\Spec}{Spec}
\DeclareMathOperator{\Tor}{Tor}
\DeclareMathOperator{\id}{id}
\newcommand{\bA}{\mathbb{A}}
\newcommand{\bC}{\mathbb{C}}
\newcommand{\bF}{\mathbb{F}}
\newcommand{\bG}{\mathbb{G}}
\newcommand{\bP}{\mathbb{P}}
\newcommand{\bQ}{\mathbb{Q}}
\newcommand{\bR}{\mathbb{R}}
\newcommand{\bS}{\mathbb{S}}
\newcommand{\bZ}{\mathbb{Z}}
\newcommand{\hatotimes}{\mathbin{\widehat{\otimes}}}
\newcommand{\cotensor}{\mathbin{\square}}
\newcommand{\longto}{\longrightarrow}
\newcommand{\longfrom}{\longleftarrow}
\newcommand{\onto}{\twoheadrightarrow}
\newcommand{\qqandqq}{\qquad \text{and} \qquad}
\newcommand{\sA}{\mathscr{A}}
\newcommand{\sO}{\mathscr{O}}
\renewcommand{\:}{\colon}
\newcommand{\<}{\langle}
\renewcommand{\>}{\rangle}
\newcommand{\ds}{\displaystyle} 

\numberwithin{equation}{section}

\begin{document}

\begin{abstract}
We establish motivic versions of the theorems of Lin and Gunawardena,
thereby confirming the motivic Segal conjecture for
the algebraic group~$\mu_\ell$ of $\ell$-th roots of unity,
where $\ell$ is any prime.  To achieve
this we develop motivic Singer constructions associated to the symmetric
group~$S_\ell$ and to~$\mu_\ell$, and introduce a delayed limit Adams
spectral sequence.
\end{abstract}

\maketitle

\tableofcontents

\section{Introduction}

Let $\gamma^1 \downarrow \bR P^\infty \simeq BC_2$ be the tautological
line bundle over infinite-dimensional real projective space, let $\bR
P_{-m}^\infty = Th(-m \gamma^1)$ be the Thom spectrum of the negative
of $m$ times $\gamma^1$, and let $\bR P_{-\infty}^\infty = \holim_m \bR
P_{-m}^\infty$.  Mahowald conjectured that there is a $2$-adic equivalence
$\bR P_{-\infty}^\infty \simeq S^{-1}$, see Adams~\cite{Ada74b}*{p.~5}.
More generally, Segal conjectured for finite groups~$G$ that there is an
$I(G)$-adic equivalence $(S_G)^G \simeq (S_G)^{hG}$ from the fixed points
to the homotopy fixed points of the $G$-equivariant sphere spectrum.
Here $I(G)$ denotes the augmentation ideal in the Burnside ring of~$G$.

Mahowald's conjecture, which is equivalent to Segal's Burnside ring
conjecture for $C_2$, was proved by Lin in~\cite{Lin80}*{Thm.~1.2}.
For odd primes~$\ell$, Gunawardena~\cite{Gun81} proved Segal's
conjecture for $C_\ell$, obtaining an $\ell$-adic equivalence
$L_{-\infty}^\infty \simeq S^{-1}$.  Here $L_{-\infty}^\infty$
denotes a homotopy limit of Thom spectra over the infinite-dimensional
lens space $L^\infty \simeq BC_\ell$.  Segal's conjecture was later
affirmed for all finite groups by Carlsson~\cite{Car84}, building on
May-McClure~\cite{MM82}, Adams--Gunawardena--Miller~\cite{AGM85} and
Caruso--May--Priddy~\cite{CMP87}.

In this paper we promote the classical theorems of Lin and Gunawardena
to the motivic setting, obtaining $\pi_{*,*}$-isomorphisms $\bS
\simeq \Sigma^{1,0} L_{-\infty}^\infty$, after $(\ell,\eta)$-adic
completion, for all primes~$\ell$.  Here $\bS$ denotes the motivic
sphere spectrum, $\eta \in \pi_{1,1}(\bS)$ is the Hopf fibration,
and now $L_{-\infty}^\infty = \holim_m L_{-2m}^\infty$,
where $L_{-2m}^\infty$ is the Thom spectrum of a virtual algebraic vector
bundle over the geometric classifying space $L^\infty = B\mu_\ell$ of the
algebraic group~$\mu_\ell$ of $\ell$-th roots of unity.  More precisely,
$L_{-2m}^\infty = \hocolim_n L_{-2m}^{2n-2m-1}$, where $L_{-2m}^{2n-2m-1}
= Th(-m \gamma_n^* \downarrow L^{2n-1})$ and $\gamma_n^*$ is the dual of
the tautological algebraic line bundle over $L^{2n-1} = (\bA^n \setminus
\{0\})/\mu_\ell$.

\begin{theorem} \label{thm:main}
Let $S$ be a finite dimensional Noetherian scheme, essentially smooth
over a field or Dedekind domain containing $1/\ell$.  There is a
$\pi_{*,*}$-isomorphism
$$
e^\wedge_{\ell,\eta} \: \bS^\wedge_{\ell,\eta}
        \longto (\Sigma^{1,0} L_{-\infty}^\infty)^\wedge_{\ell,\eta}
$$
in the stable motivic homotopy category $SH(S)$.  If $S = \Spec k$
for $k$ a field, then $e^\wedge_{\ell,\eta}$ is a motivic equivalence.
\end{theorem}

In other words, we prove the motivic Segal conjecture in its
non-equivariant form, in the case of the algebraic group~$\mu_\ell$, for
any prime~$\ell$.  For $\ell=2$ this is the motivic version of Mahowald's
conjecture and Lin's theorem.  For $\ell$ odd it is the motivic version
of Gunawardena's theorem.

Already for $S = \Spec k$ in the algebraically closed case $k = \bC$,
the additional information about motivic weight has proved to be a
valuable new tool for calculational purposes, cf.~Isaksen~\cite{Isa19}
and Isaksen--Wang--Xu~\cite{IWX20}.  In the real case $k = \bR$, many
new phenomena arise, cf.~Hill~\cite{Hil11}, Dugger--Isaksen~\cite{DI17}
and~Belmont--Isaksen~\cite{BI22}.  Our results are valid even in the
arithmetically most substantial cases of (rings of $\ell$-integers in)
number fields.  In particular, we have made an effort to not have
to assume that the mod~$\ell$ motivic cohomology groups $H^{*,*}
= H^{*,*}(S; \bZ/\ell)$ are finite in each bidegree.  Our results
enable an analysis of $\pi_{*,*}(\bS)$ by comparison with the homotopy
spectral sequence associated to the tower $\{L_{-2m}^\infty\}_m$, i.e.,
the motivic Mahowald root invariants, refining~Mahowald~\cite{Mah85}
and Mahowald--Ravenel~\cite{MR93}.  Such applications have already
appeared in Quigley's papers~\cite{Qui19}, \cite{Qui21a}, \cite{Qui21b}.
We expect the interplay between the motivic cohomology of number fields
and the Mahowald root invariants to be very rich.

In Section~\ref{sec:motivic-Steenrod} we review from Voevodsky's
article~\cite{Voe03a} the Hopf algebroid structure of the motivic
dual Steenrod algebra $\sA_{*,*}$, and of its quotients $A(n)_{*,*} =
\sA_{*,*}/I(n)$.  In Section~\ref{sec:bicomodule-algebras} we generalize
the approach of Adams--Gunawardena--Miller from~\cite{AGM85}*{\S2},
and introduce the $A(n)_{*,*}$-$A(n-1)_{*,*}$ bicomodule algebras
$C(n)_{*,*} = \sA_{*,*}/J(n)$ and their localizations $B(n)_{*,*}$
away from~$\xi_1$.  In Section~\ref{sec:dual-bimodules} we
dualize these constructions, following Boardman~\cite{Boa82}*{\S3},
obtaining the motivic Steenrod algebra $\sA$, its finite subalgebras
$A(n)$, and the $A(n)$-$A(n-1)$-bimodules $C(n)$ and $B(n)$.
In Section~\ref{sec:small-Singer} we generalize the (small) Singer
construction of Singer~\cite{Sin81} and Li--Singer~\cite{LS82}, obtaining
an $\sA$-module $R_S(M) = \colim_n B(n) \otimes_{A(n-1)} M$ and a natural
homomorphism $\epsilon \: R_S(M) \to M$ for each $\sA$-module~$M$.

We prove in Theorem~\ref{thm:RSigmaH-SigmaHBSigmalloc} that $R_S(H^{*,*})
\cong \Sigma^{1,0} H^{*,*}(BS_\ell)_{\loc}$ is a shifted localization of
the motivic cohomology of the geometric classifying space of the symmetric
group $S_\ell$ on $\ell$ letters.  In Section~\ref{sec:large-Singer}
we recast Adams--Gunawardena--Miller~\cite{AGM85}*{\S5} and
construct a (large) Singer construction $R_\mu(M)$ and a natural
$\sA$-module homomorphism $\epsilon \: R_\mu(M) \to M$.  We show in
Corollary~\ref{cor:RmuH-SigmaHBmulloc} that $R_\mu(H^{*,*}) \cong
\Sigma^{1,0} H^{*,*}(B\mu_\ell)_{\loc}$ is a shifted localization
of the motivic cohomology of the infinite lens space $B\mu_\ell$.
In Section~\ref{sec:evaluations} we prove that the evaluation
homomorphisms~$\epsilon$ are $\Ext$-equivalences.  Here we deviate
from the $\Tor$-equivalence approach of~\cite{AGM85}*{\S2}, due to the
two-sided nature of Hopf algebroids.

In Section~\ref{sec:tower-Thom} we construct the tower
$\{L_{-2m}^\infty\}_m$ of motivic spectra, and the map $e \: \bS \to
\Sigma^{1,0} L_{-\infty}^\infty$ to the suspension of their homotopy
limit.  We show in~Proposition~\ref{prop:HcLinfty-vs-HBmuloc}
that the continuous cohomology $H_c^{*,*}(L_{-\infty}^\infty)
= \colim_m H^{*,*}(L_{-2m}^\infty)$ is isomorphic as an
$\sA$-module to $H^{*,*}(B\mu_\ell)_{\loc}$, and that $e$ induces
the $\Ext$-equivalence~$\epsilon$, via the identifications above.
The plan is now to compare the motivic mod~$\ell$ Adams spectral
sequence for $\bS$ with the tower of Adams spectral sequences
associated to the $L_{-2m}^\infty$.  This works fine in the presence of
sufficient finiteness to ensure that the algebraic limit of these Adams
spectral sequences is again a spectral sequence, as is the case in the
classical setting of Caruso--May--Priddy~\cite{CMP87}.  However, for
base schemes~$S$ such that $H^{*,*}$ is not finite in each bidegree,
this approach can fail.  Instead, we form a modified Adams spectral
sequence, called the delayed limit Adams spectral sequence, where any
$\limone$-classes arising from non-exactness are shifted up into the
next filtration degree.

In Section~\ref{sec:generalized-EM} we prepare for this
construction by introducing some terminology for motivic generalized
Eilenberg--MacLane spectra, and formulate a finiteness condition,
called bifinite type, which lets us identify the $E_1$- and
$E_2$-terms of motivic Adams spectral sequences in algebraic terms.
In Section~\ref{sec:delayed-Adams} we introduce the delayed limit Adams
spectral sequence in Definition~\ref{def:del-lim-adams-sp-seq},
and identify its $E_2$-term as $\Ext$ for a continuous
cohomology $\sA$-module in Proposition~\ref{prop:dellimE2}.
In Proposition~\ref{prop:del-lim-cond-conv} we show that the
delayed limit Adams spectral sequence converges conditionally, and
in Proposition~\ref{prop:comparison} we adapt a comparison theorem
from Boardman~\cite{Boa99} for morphisms of conditionally convergent
spectral sequences.  In Section~\ref{sec:motivic-lin-gunawardena}
the threads are brought together.  See Theorem~\ref{thm:mot-lin-gun}
for the proof of Theorem~\ref{thm:main}.

This article is based on the first author's PhD thesis~\cite{Gre13},
guided by the second author.

\section{The motivic Steenrod algebra and its dual}
\label{sec:motivic-Steenrod}

Let $S$ be a Noetherian (separated) scheme of finite (Krull)
dimension~$d$, essentially smooth over a field or a Dedekind domain,
and let $\ell$ be a prime that is invertible on~$S$.

Let $SH(S)$ be Voevodsky's motivic stable homotopy
category~\cite{Voe98}*{Def.~5.7}, \cite{Jar00} associated to smooth
schemes over~$S$.  It is triangulated, and has a compatible closed
symmetric monoidal structure given by the motivic sphere spectrum~$\bS =
\Sigma^\infty S_+$, the smash product pairing $- \wedge -$, the twist
isomorphism $\gamma$ and the function spectrum $F(-, -)$.
Let $H = H\bZ/\ell$ be the motivic Eilenberg--MacLane spectrum
representing motivic cohomology with coefficients in~$\bZ/\ell$.
It is a commutative ring spectrum, with unit map $\eta \: \bS \to H$
and product $\mu \: H \wedge H \to H$.  Moreover, $H$ is known to be
cellular~\cite{Hoy15}*{Prop.~8.1}, \cite{Spi18}*{Cor.~10.4}, i.e.,
an iterated homotopy colimit of stable motivic spheres.

Let $H_{*,*} = \pi_{*,*}(H) = H^{-*,-*}$ denote the motivic homology
and cohomology groups of the base scheme~$S$.  Then $H^{p,q} = 0$
unless $0 \le p \le \min\{q+d, 2q\}$, cf.~\cite{Gei04}*{Cor.~4.4},
\cite{Hoy15}*{Cor.~4.26}.  For $x \in \pi_{t,u}(X)$, where $X$ is any
motivic spectrum, we refer to~$t$ and~$u$ as the topological degree and
weight of~$x$, respectively.  We write $|x| = \deg(x) = t$, $\wt(x) = u$
and $\|x\| = (t,u)$.  The cup product induced by~$\mu$ gives $H^{*,*} =
H_{-*,-*}$ the structure of a bigraded commutative $\bZ/\ell$-algebra.
Only the parity of the topological degree plays a role in bigraded
commutativity.

Let $\sA = H^{*,*}(H) = \pi_{-*,-*} F(H, H)$ denote the motivic Steenrod
algebra, and let $\sA_{*,*} = H_{*,*}(H) = \pi_{*,*} (H \wedge H)$
denote its dual.  Then $\sA_{*,*}$ is free as a left $H_{*,*}$-module,
cf.~Lemma~\ref{lem:left-H-basis-A}, so the pair $(H_{*,*},
\sA_{*,*})$ admits the structure of a bigraded Hopf algebroid
\cite{Ada69}*{Lec.~3}, \cite{MR77}*{\S1}, \cite{Rav86}*{Def.~A1.1.1}.
Its structure maps are the following $\bZ/\ell$-algebra homomorphisms:
\begin{enumerate}
\item
the left unit $\eta_L \: H_{*,*} \to \sA_{*,*}$ induced by
$1 \wedge \eta \: H = H \wedge \bS \to H \wedge H$;
\item
the right unit $\eta_R \: H_{*,*} \to \sA_{*,*}$ induced by
$\eta \wedge 1 \: H = \bS \wedge H \to H \wedge H$;
\item
the product $\phi \: \sA_{*,*} \otimes \sA_{*,*} \to \sA_{*,*}$
induced by $(\mu \wedge \mu)(1 \wedge \gamma \wedge 1) \: H \wedge H
\wedge H \wedge H \to H \wedge H$;
\item
the counit $\epsilon \: \sA_{*,*} \to H_{*,*}$ induced
by $\mu \: H \wedge H \to H$;
\item
the coproduct $\psi \: \sA_{*,*} \to \sA_{*,*} \otimes_{H_{*,*}}
\sA_{*,*}$ induced by $1 \wedge \eta \wedge 1 \: H \wedge H
= H \wedge \bS \wedge H \to H \wedge H \wedge H \cong
(H \wedge H) \wedge_H (H \wedge H)$;
\item
the conjugation $\chi \: \sA_{*,*} \to \sA_{*,*}$ induced by $\gamma \:
H \wedge H \to H \wedge H$.
\end{enumerate}
We use the left and right units to view $\sA_{*,*}$ as an
$H_{*,*}$-$H_{*,*}$-bimodule, and $- \otimes_{H_{*,*}} -$ in~(5)
denotes the bimodule tensor product.

More explicitly,
$$
\sA_{*,*} = H_{*,*} [ \tau_0, \tau_1, \dots, \xi_1, \xi_2, \dots ] /
	(\tau_i^2 - T_i \mid i\ge0)
$$
is a bigraded commutative $H_{*,*}$-algebra generated by classes
$\tau_i$ in bidegree $\|\tau_i\| = (2\ell^i-1, \ell^i-1)$ and
$\xi_i$ in bidegree $\|\xi_i\| = (2\ell^i-2, \ell^i-1)$, where
$$
T_i = \begin{cases}
\tau \xi_{i+1} + \rho \tau_{i+1} + \rho \tau_0 \xi_{i+1}
	& \text{for $\ell=2$,} \\
0 & \text{for $\ell$ odd.}
\end{cases}
$$
Here the elements $\rho \in H^{1,1} = H_{-1,-1}$ and $\tau \in H^{0,1}
= H_{0,-1}$ are specified for $\ell=2$ in~\cite{Voe03a}*{Thm.~6.10}.
They shall be interpreted to be zero for $\ell$ odd.  In these terms,
\begin{enumerate}
\item
the algebra unit is $\eta_L$;
\item
$\eta_R = \chi \eta_L$ satisfies $\eta_R(\rho) = \rho$ and
$\eta_R(\tau) = \tau + \rho \tau_0$;
\item
the algebra product is $\phi$;
\item
the counit $\epsilon$ maps each $\tau_i$ and $\xi_i$ to $0$;
\item
the coproduct $\psi$ satisfies
$$
\psi(\tau_k) = \tau_k \otimes 1
	+ \sum_{i+j=k} \xi_i^{\ell^j} \otimes \tau_j
\qqandqq
\psi(\xi_k) = \sum_{i+j=k} \xi_i^{\ell^j} \otimes \xi_j \,,
$$
where $\xi_0 = 1$;
\item
the conjugation $\chi$ satisfies
$$
\tau_k + \sum_{i+j=k} \xi_i^{\ell^j} \chi(\tau_j) = 0
\qqandqq
\sum_{i+j=k} \xi_i^{\ell^j} \chi(\xi_j) = 0 \,,
$$
and $\chi^2 = 1$.
\end{enumerate}
See \cite{Voe03a}*{Thm.~12.6, Lem.~12.11, Rem.~12.12},
\cite{Voe10}*{Thm.~3.49}, \cite{Rio12}*{Thm.~5.2.13},
\cite{HKO17}*{Thm.~5.6} and~\cite{Spi18}*{Thm.~10.26} for proofs.

\begin{lemma} \label{lem:left-H-basis-A}
The monomials
$$
\tau^E \xi^R = \tau_0^{e_0} \tau_1^{e_1} \dotsm
	\xi_1^{r_1} \xi_2^{r_2} \dotsm \,,
$$
where $E = (e_0, e_1, \dots)$ and $R = (r_1, r_2, \dots)$ range through
the finite length integer sequences with $e_s \in \{0,1\}$ and $r_s\ge0$,
form a basis for
$$
\sA_{*,*} = \frac{
        H_{*,*}[\tau_0, \tau_1, \dots, \xi_1, \xi_2, \dots]
        }{ (\tau_i^2 - T_i \mid i\ge0) }
$$
as a free left $H_{*,*}$-module.
\end{lemma}

\begin{proof}
For $\ell$ odd this is clear.  The claim for $\ell=2$ follows from the
form of the relations $\tau_i^2 = T_i$, since $\xi_{i+1}$, $\tau_{i+1}$
and $\tau_0 \xi_{i+1}$ have higher weight than~$\tau_i^2$.
\end{proof}

\begin{lemma} \label{lem:right-H-basis-A}
The same monomials $\tau^E \xi^R$ as in
Lemma~\ref{lem:left-H-basis-A} form a basis for $\sA_{*,*}$
as a free right $H_{*,*}$-module.
\end{lemma}

\begin{proof}
For $t\ge0$ let
$$
F^t \sA_{*,*} = \< \tau^E \xi^R \mid \deg(\tau^E \xi^R) \ge t \>
	\subset \sA_{*,*}
$$
be the left $H_{*,*}$-submodule generated by the monomials from
Lemma~\ref{lem:left-H-basis-A} of topological degree $\ge t$.
These are also right $H_{*,*}$-submodules, since $\epsilon \eta_L =
\id = \epsilon \eta_R$ implies $\eta_L \equiv \eta_R \mod F^1 \sA_{*,*}
= \ker(\epsilon)$, and $F^t \sA_{*,*} \cdot F^1 \sA_{*,*} \subset
F^{t+1} \sA_{*,*} \subset F^t \sA_{*,*}$.  (The first inclusion uses
that $\tau_i^2 = T_i$ has topological degree less than or equal to that
of $\xi_{i+1}$, $\tau_{i+1}$ and~$\tau_0 \xi_{i+1}$.)  This defines a
decreasing filtration of $\sA_{*,*}$ by $H_{*,*}$-$H_{*,*}$-bimodules,
such that the left and right $H_{*,*}$-module actions agree on each
filtration quotient
$$
\gr^t \sA_{*,*} = \frac{ F^t \sA_{*,*} }{ F^{t+1} \sA_{*,*} } \,.
$$
The (cosets of the) degree~$= t$ monomials $\tau^E \xi^R$ from
Lemma~\ref{lem:left-H-basis-A} freely generate this quotient
as a left $H_{*,*}$-module, hence also as a right $H_{*,*}$-module.
It follows that the degree $\ge0$ monomials $\tau^E \xi^R$ freely
generate $\sA_{*,*}$ as a right $H_{*,*}$-module, since in any given
bidegree $F^t \sA_{*,*} = 0$ for all sufficiently large~$t$.
\end{proof}


The classical definitions of~\cite{Ste62}*{\S II.3, \S VI.4} generalize
to the motivic setting.

\begin{definition} \label{def:In-An}
For $n\ge-1$, let $I(n) \subset \sA_{*,*}$ be the ideal
$$
I(n) = (\tau_{n+1}, \tau_{n+2}, \dots,
	\xi_1^{\ell^n}, \xi_2^{\ell^{n-1}}, \dots,
	\xi_n^\ell, \xi_{n+1}, \xi_{n+2}, \dots)
$$
generated by $\tau_k$ for $k\ge n+1$ and by
$\xi_i^{\ell^j}$ for $i\ge1$, $j\ge0$ and $i+j\ge n+1$.
Note that $T_i \in I(n)$ for $i \ge n$.
Let
$$
A(n)_{*,*} = \sA_{*,*}/I(n)
= \frac{ H_{*,*}[ \tau_0, \dots, \tau_n, \xi_1, \xi_2, \dots, \xi_n] }
	{ (\tau_0^2 - T_0, \dots, \tau_{n-1}^2 - T_{n-1}, \tau_n^2,
	\xi_1^{\ell^n}, \xi_2^{\ell^{n-1}}, \dots, \xi_n^\ell) }
$$
be the quotient algebra.
\end{definition}

\begin{example}
\begin{align*}
I(-1) &= (\tau_0, \tau_1, \dots, \xi_1, \xi_2, \dots) \\
I(0) &= (\tau_1, \tau_2, \dots, \xi_1, \xi_2, \dots) \\
I(1) &= (\tau_2, \tau_3, \dots, \xi_1^\ell, \xi_2, \xi_3, \dots)
\end{align*}
so
\begin{align*}
A(-1)_{*,*} &= H_{*,*} \\
A(0)_{*,*} &= H_{*,*} [\tau_0] / (\tau_0^2) \\
A(1)_{*,*} &= H_{*,*} [\tau_0, \tau_1, \xi_1] /
	(\tau_0^2 - T_0, \tau_1^2, \xi_1^\ell) \,.
\end{align*}
\end{example}

\begin{lemma}
There is a unique Hopf algebroid structure on $(H_{*,*}, A(n)_{*,*})$
making the canonical projection $\pi_n \: \sA_{*,*} \to \sA_{*,*}/I(n)
= A(n)_{*,*}$ a Hopf algebroid homomorphism.
\end{lemma}

\begin{proof}
The Hopf algebroid structure maps of $(H_{*,*}, A(n)_{*,*})$ are
$\bZ/\ell$-algebra homomorphisms, determined as follows:
\begin{enumerate}
\item
The left unit $\eta_{L,n} \: H_{*,*} \to A(n)_{*,*}$ is the
composite $\pi_n \circ \eta_L$.
\item
The right unit $\eta_{R,n} \: H_{*,*} \to A(n)_{*,*}$ is the
composite $\pi_n \circ \eta_R$.
\item
The algebra product $\phi_n \: A(n)_{*,*} \otimes A(n)_{*,*} \to
A(n)_{*,*}$ is characterized by $\phi_n \circ (\pi_n \otimes \pi_n)
= \pi_n \circ \phi$, and exists because $I(n) \subset \sA_{*,*}$
is an ideal.
\item
The counit $\epsilon_n \: A(n)_{*,*} \to H_{*,*}$ is characterized by
$\epsilon_n \circ \pi_n = \epsilon$, and exists because
$\epsilon(x) = 0$ for each generator $x$ of $I(n)$.
\item
The coproduct $\psi_n \: A(n)_{*,*} \to A(n)_{*,*} \otimes_{H_{*,*}}
A(n)_{*,*}$ is characterized by $\psi_n \circ \pi_n = (\pi_n \otimes
\pi_n) \psi$, and exists because $(\pi_n \otimes \pi_n) \psi(x) = 0$
for each generator $x$ of $I(n)$.
\item
The conjugation $\chi_n \: A(n)_{*,*} \to A(n)_{*,*}$ is characterized
by $\chi_n \circ \pi_n = \pi_n \circ \chi$, and exists because
$\chi(x) \in I(n)$ for each generator $x$ of $I(n)$.
\end{enumerate}
In more detail, the explicit formulas for the coproduct show that
$\psi(\tau_k)$, for all $k \ge n+1$, and $\psi(\xi_i^{\ell^j}) =
\psi(\xi_i)^{\ell^j}$, for all $i\ge1$, $j\ge0$ with $i+j \ge n+1$,
are in the image of
$$
I(n) \otimes_{H_{*,*}} \sA_{*,*} \ \oplus \ \sA_{*,*} \otimes_{H_{*,*}} I(n)
	\longto \sA_{*,*} \otimes_{H_{*,*}} \sA_{*,*} \,,
$$
so that $\psi(I(n))$ is contained in this image.  Likewise, the
recursive formulas for the conjugation show that $\chi(\tau_k)$ and
$\chi(\xi_i^{\ell^j}) = \chi(\xi_i)^{\ell^j}$ are in $I(n)$, for the
same~$k$, $i$ and~$j$, so that $\chi(I(n)) = I(n)$.  The verification
that these structure maps make $(H_{*,*}, A(n)_{*,*})$ a Hopf algebroid,
with $(\id, \pi_n)$ a Hopf algebroid homomorphism, follows formally from
the fact that $(H_{*,*}, \sA_{*,*})$ is a Hopf algebroid.
\end{proof}

\begin{lemma} \label{lem:left-H-basis-An}
The monomials $\tau^E \xi^R$, where $E = (e_0, \dots, e_n)$ and $R =
(r_1, \dots, r_n)$ range through the integer sequences with $e_s \in
\{0,1\}$ and $0 \le r_s < \ell^{n+1-s}$, form a basis for $A(n)_{*,*}$
as a finitely generated free left $H_{*,*}$-module.
\end{lemma}

\begin{proof}
The ideal~$I(n)$ equals the free left $H_{*,*}$-submodule
of $\sA_{*,*}$ generated by the monomials $\tau^E \xi^R$ from
Lemma~\ref{lem:left-H-basis-A} for which $e_s = 1$ for some $s\ge n+1$
or $r_s \ge \ell^{n+1-s}$ for some $s\ge1$.  This implies the claim.
\end{proof}

\begin{lemma} \label{lem:right-H-basis-An}
The same monomials $\tau^E \xi^R$ as in
Lemma~\ref{lem:left-H-basis-An} form a basis for $A(n)_{*,*}$
as a free right $H_{*,*}$-module.
\end{lemma}

\begin{proof}
Replace $\sA_{*,*}$ and Lemma~\ref{lem:left-H-basis-A} in the proof
of Lemma~\ref{lem:right-H-basis-A}
by $A(n)_{*,*}$ and Lemma~\ref{lem:left-H-basis-An}.
\end{proof}

The inclusions $I(n) \subset I(n-1)$ induce a
tower of surjective Hopf algebroid homomorphisms
\begin{equation} \label{eq:dualAntower}
\sA_{*,*} \longto \dots \longto A(n)_{*,*} \longto
A(n-1)_{*,*} \longto \dots \longto H_{*,*} \,.
\end{equation}
The composites
$$
\lambda \: \sA_{*,*} \overset{\psi}\longto
\sA_{*,*} \otimes_{H_{*,*}} \sA_{*,*}
\overset{\pi_n \otimes \id}\longto
A(n)_{*,*} \otimes_{H_{*,*}} \sA_{*,*}
$$
and
$$
\rho \: \sA_{*,*} \overset{\psi}\longto
\sA_{*,*} \otimes_{H_{*,*}} \sA_{*,*}
\overset{\id \otimes \pi_{n-1}}\longto
\sA_{*,*} \otimes_{H_{*,*}} A(n-1)_{*,*}
$$
are $\bZ/\ell$-algebra homomorphisms giving $\sA_{*,*}$ the structure of
an $A(n)_{*,*}$-$A(n-1)_{*,*}$-bicomodule algebra, and the projection
$\pi_n \: \sA_{*,*} \to A(n)_{*,*}$ is a morphism in the category of
such bicomodule algebras.

\begin{definition} \label{def:Xn}
Let
$$
X(n)_{*,*} = H_{*,*} \{ \tau^E \xi^R \mid
	e_0 = \dots = e_n = 0, \ell^n \mid r_1, \dots, \ell \mid r_n \}
$$
be the free left $H_{*,*}$-module generated by the monomials $\tau^E
\xi^R$ with $E = (e_0, e_1, \dots)$ and $R = (r_1, r_2, \dots)$
satisfying $e_s = 0$ for $0 \le s \le n$ and $\ell^{n+1-s} \mid r_s$ for
$1 \le s \le n$.  Let $\alpha_n \: \sA_{*,*} \to X(n)_{*,*}$ be the left
$H_{*,*}$-module homomorphism mapping $\tau^E \xi^R$ to the same monomial
if $e_0 = \dots = e_n = 0$ and $\ell^n \mid r_1, \dots, \ell \mid r_n$,
and to $0$ otherwise.
\end{definition}

\begin{lemma} \label{lem:left-An-comod-iso-A}
The composite
$$
\sA_{*,*} \overset{\lambda}\longto A(n)_{*,*} \otimes_{H_{*,*}} \sA_{*,*}
	\overset{\id \otimes \alpha_n}\longto
	A(n)_{*,*} \otimes_{H_{*,*}} X(n)_{*,*}
$$
is a left $A(n)_{*,*}$-comodule isomorphism.
\end{lemma}

\begin{proof}
Both $\lambda$ and $\id \otimes \alpha_n$ respect the left
$A(n)_{*,*}$-coactions, so it suffices to show that their composite is
a left $H_{*,*}$-module isomorphism.  Each monomial in the basis from
Lemma~\ref{lem:left-H-basis-A} for $\sA_{*,*}$ factors uniquely as
$\tau^E \xi^R = \tau^{E'} \xi^{R'} \cdot \tau^{E''} \xi^{R''}$ with
$$
\begin{cases}
E' = (e_0, \dots, e_n, 0, \dots) \\
R' = (r_1, \dots, r_n, 0, \dots)
	& \text{where $r_s < \ell^{n+1-s}$ for $1 \le s \le n$,} \\
E'' = (0, \dots, 0, e_{n+1}, \dots) \\
R'' = (r_1, \dots, r_n, r_{n+1}, \dots)
	& \text{where $\ell^{n+1-s} \mid r_s$ for $1 \le s \le n$,}
\end{cases}
$$
and $E = E' + E''$, $R = R' + R''$.
Hence the restricted multiplication
\begin{align*}
A(n)_{*,*} \otimes_{H_{*,*}} X(n)_{*,*}
	&\overset{\phi}\longto \sA_{*,*} \\
\tau^{E'} \xi^{R'} \otimes \tau^{E''} \xi^{R''}
	&\longmapsto \tau^{E'} \xi^{R'} \cdot \tau^{E''} \xi^{R''}
\end{align*}
defines a left $H_{*,*}$-module isomorphism.  We show that the composite
$$
(\id \otimes \alpha_n) \lambda \phi \: A(n)_{*,*} \otimes_{H_{*,*}} X(n)_{*,*}
	\longto A(n)_{*,*} \otimes_{H_{*,*}} X(n)_{*,*}
$$
is bijective.  For $t\ge0$ let $F^t X(n)_{*,*}$ be the
free left $H_{*,*}$-submodule generated by the monomials from
Definition~\ref{def:Xn} that have topological degree $\ge t$.  These
define a decreasing filtration of $X(n)_{*,*}$, with associated graded
modules $\gr^t X(n)_{*,*} = F^t X(n)_{*,*} / F^{t+1} X(n)_{*,*}$.
Direct calculation of $\lambda = (\pi_n \otimes \id)\psi$ shows that
$$
\lambda(\tau^{E'} \xi^{R'}) \equiv \tau^{E'} \xi^{R'} \otimes 1
	\mod A(n)_{*,*} \otimes_{H_{*,*}} F^1 \sA_{*,*} \,,
$$
where $F^1 \sA_{*,*} = \ker(\epsilon)$ as before, and
$$
\lambda(\tau^{E''} \xi^{R''}) = 1 \otimes \tau^{E''} \xi^{R''} \,,
$$
since each $\tau_k$ for $k \ge n+1$ and each $\xi_i^{\ell^j}$ for $i+j
\ge n+1$ is left $A(n)_{*,*}$-comodule primitive.
It follows that for $\tau^{E'} \xi^{R'} \in A(n)_{*,*}$
and $\tau^{E''} \xi^{R''} \in F^t X(n)_{*,*}$ we have
$$
(\id \otimes \alpha_n) \lambda (\tau^{E'} \xi^{R'} \cdot \tau^{E''} \xi^{R''})
	\equiv \tau^{E'} \xi^{R'} \otimes \tau^{E''} \xi^{R''}
	\mod A(n)_{*,*} \otimes_{H_{*,*}} F^{t+1} X(n)_{*,*} \,.
$$
Hence $(\id \otimes \alpha_n) \lambda \phi$ maps
$A(n)_{*,*} \otimes_{H_{*,*}} F^t X(n)_{*,*}$ to itself, for each
$t\ge0$, and the induced homomorphism
$$
A(n)_{*,*} \otimes_{H_{*,*}} \gr^t X(n)_{*,*}
	\longto A(n)_{*,*} \otimes_{H_{*,*}} \gr^t X(n)_{*,*}
$$
is the identity.  The lemma follows, since $A(n)_{*,*} \otimes_{H_{*,*}}
F^t X(n)_{*,*}$ is eventually zero in any given bidegree.
\end{proof}

\section{Some bicomodule algebras}
\label{sec:bicomodule-algebras}

The classical definitions of~\cite{AGM85}*{\S2} also generalize to the
motivic setting.

\begin{definition} \label{def:Jn-Cn-Bn}
For $n\ge0$, let $J(n) \subset \sA_{*,*}$ be the ideal
$$
J(n) = (\tau_{n+1}, \tau_{n+2}, \dots,
	\xi_2^{\ell^{n-1}}, \xi_3^{\ell^{n-2}}, \dots,
	\xi_n^\ell, \xi_{n+1}, \xi_{n+2}, \dots)
$$
generated by $\tau_k$ for $k\ge n+1$ and by
$\xi_i^{\ell^j}$ for $i\ge2$, $j\ge0$ and $i+j\ge n+1$.
Note that $I(n) = J(n) + (\xi_1^{\ell^n})$.
Let
$$
C(n)_{*,*} = \sA_{*,*}/J(n)
= \frac{ H_{*,*}[ \tau_0, \dots, \tau_n, \xi_1, \xi_2, \dots, \xi_n] }
	{ (\tau_0^2 - T_0, \dots, \tau_n^2 - T_n,
	\xi_2^{\ell^{n-1}}, \dots, \xi_n^\ell) }
$$
be the quotient algebra.  Let
$$
B(n)_{*,*} = C(n)_{*,*} [1/\xi_1]
= \frac{ H_{*,*}[ \tau_0, \dots, \tau_n, \xi_1^{\pm1}, \xi_2, \dots, \xi_n] }
	{ (\tau_0^2 - T_0, \dots, \tau_n^2 - T_n,
	\xi_2^{\ell^{n-1}}, \dots, \xi_n^\ell) }
$$
be the localization of $C(n)_{*,*}$ away from $\xi_1$.
\end{definition}

\begin{example}
\begin{align*}
J(0) &= (\tau_1, \tau_2, \dots, \xi_2, \xi_3, \dots) \\
J(1) &= (\tau_2, \tau_3, \dots, \xi_2, \xi_3, \dots) \\
J(2) &= (\tau_3, \tau_4, \dots, \xi_2^\ell, \xi_3, \dots)
\end{align*}
so
\begin{align*}
C(0)_{*,*} &= H_{*,*} [ \tau_0, \xi_1 ] / (\tau_0^2 - T_0) \\
C(1)_{*,*} &= H_{*,*} [ \tau_0, \tau_1, \xi_1 ] / (\tau_0^2 - T_0, \tau_1^2) \\
C(2)_{*,*} &= H_{*,*} [ \tau_0, \tau_1, \tau_2, \xi_1, \xi_2 ] /
	(\tau_0^2 - T_0, \tau_1^2 - T_1, \tau_2^2, \xi_2^\ell)
\end{align*}
and
\begin{align*}
B(0)_{*,*} &= H_{*,*} [ \tau_0, \xi_1^{\pm1} ] / (\tau_0^2 - T_0) \\
B(1)_{*,*} &= H_{*,*} [ \tau_0, \tau_1, \xi_1^{\pm1} ] /
	(\tau_0^2 - T_0, \tau_1^2) \\
B(2)_{*,*} &= H_{*,*} [ \tau_0, \tau_1, \tau_2, \xi_1^{\pm1}, \xi_2 ] /
	(\tau_0^2 - T_0, \tau_1^2 - T_1, \tau_2^2, \xi_2^\ell) \,.
\end{align*}
\end{example}

\begin{lemma} \label{left-H-basis-Cn-Bn}
(a)
The monomials $\tau^E \xi^R$, where $E = (e_0, \dots, e_n)$ and $R =
(r_1, \dots, r_n)$ range through all sequences with $e_s \in \{0,1\}$
for $0 \le s \le n$, $r_1 \ge0$ and $0 \le r_s < \ell^{n+1-s}$ for $2 \le
s \le n$, form a basis for $C(n)_{*,*}$ as a free left $H_{*,*}$-module.

(b)
The monomials
$$
\tau^E \xi^R = \tau_0^{e_0} \dotsm \tau_n^{e_n}
	\xi_1^{r_1} \xi_2^{r_2} \dotsm \xi_n^{r_n}
$$
with $(E,R)$ as in~(a), except that $r_1$ can now be any integer,
form a basis for $B(n)_{*,*}$ as a free left $H_{*,*}$-module.
\end{lemma}

\begin{proof}
The ideal~$J(n)$ equals the free left $H_{*,*}$-submodule
of $\sA_{*,*}$ generated by the monomials $\tau^E \xi^R$ from
Lemma~\ref{lem:left-H-basis-A} for which $e_s = 1$ for some $s\ge
n+1$ or $r_s \ge \ell^{n+1-s}$ for some $s\ge2$.  This implies part~(a).
Part~(b) follows by inverting~$\xi_1$.
\end{proof}

\begin{lemma} \label{lem:right-H-basis-Cn-Bn}
(a)
The same monomials $\tau^E \xi^R$ as in
Lemma~\ref{left-H-basis-Cn-Bn}(a) form a basis for $C(n)_{*,*}$
as a free right $H_{*,*}$-module.

(b)
The same monomials $\tau^E \xi^R$ as in
Lemma~\ref{left-H-basis-Cn-Bn}(b) form a basis for $B(n)_{*,*}$
as a free right $H_{*,*}$-module.
\end{lemma}

\begin{proof}
For part~(a), replace $\sA_{*,*}$ and Lemma~\ref{lem:left-H-basis-A}
in the proof of Lemma~\ref{lem:right-H-basis-A} by $C(n)_{*,*}$ and
Lemma~\ref{left-H-basis-Cn-Bn}(a).

For part~(b), instead replace these by $B(n)_{*,*}$ and
Lemma~\ref{left-H-basis-Cn-Bn}(b), and allow the filtration index~$t$
in the proof of Lemma~\ref{lem:right-H-basis-A} to run over all integers,
noting that in any given bidegree $F^t B(n)_{*,*} = B(n)_{*,*}$ for all
sufficiently negative~$t$.
(Alternatively, part~(b) can be deduced from part~(a) by inverting
$\xi_1$, but the given proof also ensures that the left and right
$H_{*,*}$-actions on $\gr^t B(n)_{*,*}$ agree, which will be needed
in Lemma~\ref{lem:right-H-basis-Cn-Bn-cohom}(b).)
\end{proof}

\begin{example} \label{ex:left-right-H-bases-C0-B0}
(a)
The monomials
$$
\{ \tau_0^e \xi_1^r \mid e \in \{0,1\}, r\ge0 \}
$$
form a basis for $C(0)_{*,*}$, both as a left
$H_{*,*}$-module and as a right $H_{*,*}$-module.

(b)
The monomials
$$
\{ \tau_0^e \xi_1^r \mid e \in \{0,1\}, r \in \bZ \}
$$
form a basis for $B(0)_{*,*}$, both as a left $H_{*,*}$-module and as a
right $H_{*,*}$-module.  The homological bidegree of $\tau_0^e \xi_1^r$
is $(e + (2\ell-2)r, (\ell-1)r)$.
\end{example}

The inclusions $J(n) \subset I(n)$ and the localization homomorphisms
yield a commutative diagram of $\bZ/\ell$-algebras and algebra
homomorphisms
\begin{equation} \label{eq:ACB}
\xymatrix{
\sA_{*,*} \ar@{->>}[r]^-{\pi_n}
        & A(n)_{*,*} \ar@{->>}[r]
        & A(n-1)_{*,*} \ar@{->>}[r]
        & A(0)_{*,*} \ar@{->>}[r] & H_{*,*} \\
\sA_{*,*} \ar@{->>}[r]^-{\pi'_n} \ar@{=}[u] \ar@{ >->}[d]
        & C(n)_{*,*} \ar@{->>}[rr]^-{\gamma_n} \ar@{->>}[u] \ar@{ >->}[d]
        & 
        & C(0)_{*,*} \ar@{->>}[u] \ar@{ >->}[d] \\
\sA_{*,*}[1/\xi_1] \ar@{->>}[r]
        & B(n)_{*,*} \ar@{->>}[rr]^-{\beta_n}
        & 
        & B(0)_{*,*} & \,.
}
\end{equation}

\begin{lemma}
There is a unique $A(n)_{*,*}$-$A(n-1)_{*,*}$-bicomodule
algebra structure on $C(n)_{*,*}$ making the canonical
projection $\pi'_n \: \sA_{*,*} \to \sA_{*,*}/J(n) = C(n)_{*,*}$ an
$A(n)_{*,*}$-$A(n-1)_{*,*}$-bicomodule algebra homomorphism.
\end{lemma}

\begin{proof}
The bicomodule structure maps are $\bZ/\ell$-algebra homomorphisms,
determined as follows:
\begin{enumerate}
\item
The left coaction $\lambda_n \: C(n)_{*,*} \to A(n)_{*,*}
\otimes_{H_{*,*}} C(n)_{*,*}$ is characterized by $\lambda_n \circ \pi'_n
= (\id \otimes \pi'_n) \circ \lambda$, and exists because $(\pi_n \otimes
\pi'_n) \psi(x) = 0$ for each generator $x$ of $J(n)$.
\item
The right coaction $\rho_n \: C(n)_{*,*} \to C(n)_{*,*} \otimes_{H_{*,*}}
A(n-1)_{*,*}$ is characterized by $\rho_n \circ \pi'_n = (\pi'_n \otimes
\id) \circ \rho$, and exists because $(\pi'_n \otimes \pi_{n-1}) \psi(x)
= 0$ for each generator $x$ of $J(n)$.
\end{enumerate}
More explicitly, $\psi(\tau_k)$ and $\psi(\xi_i^{\ell^j})$
are in the image of both
$$
I(n) \otimes_{H_{*,*}} \sA_{*,*} \ \oplus \ \sA_{*,*} \otimes_{H_{*,*}}
	J(n) \longto \sA_{*,*} \otimes_{H_{*,*}} \sA_{*,*}
$$
and
$$
J(n) \otimes_{H_{*,*}} \sA_{*,*} \ \oplus \ \sA_{*,*} \otimes_{H_{*,*}}
	I(n-1) \longto \sA_{*,*} \otimes_{H_{*,*}} \sA_{*,*}
$$
for each $k \ge n+1$ and each $i\ge2$, $j\ge0$ and $i+j \ge n+1$,
respectively.
The verification that the algebra homomorphisms $\lambda_n$ and
$\rho_n$ define coactions, and that they commute, follows formally from
the fact that $\sA_{*,*}$ is an $A(n)_{*,*}$-$A(n-1)_{*,*}$-bicomodule.
\end{proof}

\begin{lemma} \label{lem:CCA-SES}
Let $\|\xi_1^{\ell^n}\| = ((2\ell-2)\ell^n, (\ell-1)\ell^n)$ denote the
bidegree of $\xi_1^{\ell^n}$.  There is a short exact sequence
$$
0 \to \Sigma^{\|\xi_1^{\ell^n}\|} C(n)_{*,*}
	\overset{\cdot \xi_1^{\ell^n}}\longto C(n)_{*,*}
	\longto A(n)_{*,*} \to 0
$$
of $A(n)_{*,*}$-$A(n-1)_{*,*}$-bicomodules, where $\cdot \xi_1^{\ell^n}$
denotes $x \mapsto x \cdot \xi_1^{\ell^n}$.
\end{lemma}

\begin{proof}
From the definition of $I(n)$ and $J(n)$ it is clear that multiplication
by $\xi_1^{\ell^n}$ acts injectively on $C(n)_{*,*}$ with cokernel
$A(n)_{*,*}$.  It remains to verify that $\cdot \xi_1^{\ell^n}$ is an
$A(n)_{*,*}$-$A(n-1)_{*,*}$-bicomodule homomorphism, i.e., that it commutes
with the left $A(n)_{*,*}$-coaction and the right $A(n-1)_{*,*}$-coaction.
This is equivalent to $\xi_1^{\ell^n}$ being left $A(n)_{*,*}$-comodule
primitive and right $A(n-1)_{*,*}$-comodule primitive, which follows
from the observations that
$$
\psi(\xi_1^{\ell^n}) \equiv 1 \otimes \xi_1^{\ell^n}
	\mod I(n) \otimes_{H_{*,*}} \sA_{*,*}
$$
and
$$
\psi(\xi_1^{\ell^n}) \equiv \xi_1^{\ell^n} \otimes 1
	\mod \sA_{*,*} \otimes_{H_{*,*}} I(n-1) \,.
$$
\end{proof}

\begin{definition} \label{def:Bn-An-An-1-bicomod}
We assign to
$$
B(n)_{*,*} = C(n)_{*,*}[1/\xi_1^{\ell^n}]
= \colim_j \bigl( \Sigma^{-j \|\xi_1^{\ell^n}\|}
	C(n)_{*,*} \bigr)
$$
the $A(n)_{*,*}$-$A(n-1)_{*,*}$-bicomodule structure given by the
colimit of the diagram
$$
C(n)_{*,*}
    \overset{\cdot \xi_1^{\ell^n}}\longto
	\Sigma^{-\|\xi_1^{\ell^n}\|} C(n)_{*,*}
    \overset{\cdot \xi_1^{\ell^n}}\longto
	\Sigma^{-2\|\xi_1^{\ell^n}\|} C(n)_{*,*}
    \overset{\cdot \xi_1^{\ell^n}}\longto
	\dots \,.
$$
\end{definition}

\begin{lemma}
$B(n)_{*,*}$ is an $A(n)_{*,*}$-$A(n-1)_{*,*}$-bicomodule algebra, and
the canonical morphism
$C(n)_{*,*} \to B(n)_{*,*}$ is an $A(n)_{*,*}$-$A(n-1)_{*,*}$-bicomodule
algebra homomorphism.
\end{lemma}

\begin{proof}
The left $A(n)_{*,*}$-coaction
$$
\lambda_n \: B(n)_{*,*} \longto A(n)_{*,*} \otimes_{H_{*,*}} B(n)_{*,*}
$$
is obtained from the left coaction
$$
\lambda_n \: C(n)_{*,*} \longto A(n)_{*,*} \otimes_{H_{*,*}} C(n)_{*,*}
$$
by inverting (a positive power of) $\xi_1$.  Since the latter coaction
is an algebra homomorphism, so is the former.  The case of right
$A(n-1)_{*,*}$-coactions is entirely similar.
\end{proof}

\begin{definition}
Let $\gamma_n \: C(n)_{*,*} \to C(0)_{*,*}$ and
$\beta_n \: B(n)_{*,*} \to B(0)_{*,*}$ be the $\bZ/\ell$-algebra
homomorphisms shown in~\eqref{eq:ACB}.
Let $\gamma'_n \: C(n)_{*,*} \to H_{*,*}[\xi_1^{\ell^n}]$
be the composite of $\gamma_n$
and the left $H_{*,*}$-module homomorphism
$$
C(0)_{*,*} = H_{*,*}[\tau_0, \xi_1] / (\tau_0^2-T_0)
	\longto H_{*,*}[\xi_1^{\ell^n}]
$$
given for $e \in \{0,1\}$ and $r\ge0$ by
$$
\tau_0^e \xi_1^r \longmapsto \begin{cases}
\xi_1^r & \text{if $e = 0$ and $\ell^n \mid r$,} \\
0 & \text{otherwise,}
\end{cases}
$$
and let $\beta'_n \: B(n)_{*,*} \to H_{*,*}[\xi_1^{\pm \ell^n}]$ be
its localization.
\end{definition}

Note that $\tau_0 \cdot \eta_R(\tau)$ in $C(0)_{*,*}$ maps by $\gamma'_0$
to $\rho \tau \xi_1$.  Hence $\gamma'_n$ is sometimes not right
$H_{*,*}$-linear.

\begin{proposition} \label{prop:left-An-comod-isos-Cn-Bn}
The composites
$$
C(n)_{*,*}
	\overset{\lambda_n}\longto
	A(n)_{*,*} \otimes_{H_{*,*}} C(n)_{*,*}
	\overset{\id \otimes \gamma'_n}\longto 
	A(n)_{*,*} \otimes_{H_{*,*}} H_{*,*}[\xi_1^{\ell^n}]
$$
and
$$
B(n)_{*,*}
	\overset{\lambda_n}\longto
	A(n)_{*,*} \otimes_{H_{*,*}} B(n)_{*,*}
	\overset{\id \otimes \beta'_n}\longto 
	A(n)_{*,*} \otimes_{H_{*,*}} H_{*,*}[\xi_1^{\pm \ell^n}]
$$
are left $A(n)_{*,*}$-comodule isomorphisms.
\end{proposition}

\begin{proof}
The $\bZ/\ell$-algebra homomorphism
$$
(\id \otimes \gamma_n) \lambda_n \: C(n)_{*,*}
	\longto A(n)_{*,*} \otimes_{H_{*,*}} C(0)_{*,*}
$$
is left $H_{*,*}$-linear and maps the remaining algebra generators by
$$
\begin{cases}
\tau_0 \longmapsto \tau_0 \otimes 1 + 1 \otimes \tau_0 \\
\tau_k \longmapsto \tau_k \otimes 1 + \xi_k \otimes \tau_0
& \text{for $1 \le k \le n$,} \\
\xi_1 \longmapsto \xi_1 \otimes 1 + 1 \otimes \xi_1 \\
\xi_k \longmapsto \xi_k \otimes 1 + \xi_{k-1}^\ell \otimes \xi_1
&\text{for $2 \le k \le n$.}
\end{cases}
$$
In particular, it and $(\id \otimes \gamma'_n) \lambda_n$ respect the
decreasing $(\xi_1^{\ell^n})$-adic filtrations defined (internally
to this proof) for $m\ge0$ by
\begin{align*}
F^m C(n)_{*,*} &= C(n)_{*,*} \cdot (\xi_1^{\ell^n})^m \\
F^m C(0)_{*,*} &= C(0)_{*,*} \cdot (\xi_1^{\ell^n})^m \\
F^m H_{*,*}[\xi_1^{\ell^n}]
	&= H_{*,*}[\xi_1^{\ell^n}] \cdot (\xi_1^{\ell^n})^m \,.
\end{align*}
The induced homomorphism
$$
\frac{F^m C(n)_{*,*}}{F^{m+1} C(n)_{*,*}}
	\longto A(n)_{*,*} \otimes_{H_{*,*}}
\frac{F^m H_{*,*}[\xi_1^{\ell^n}]}{F^{m+1} H_{*,*}[\xi_1^{\ell^n}]}
$$
of associated graded left $A(n)_{*,*}$-comodules is the isomorphism
given by
$$
\tau^E \xi^R \cdot (\xi_1^{\ell^n})^m
	\longmapsto \tau^E \xi^R \otimes (\xi_1^{\ell^n})^m \,,
$$
where $e_s \in \{0,1\}$ for $0 \le s \le n$ and $0 \le r_s < \ell^{n+1-s}$
for $1 \le s \le n$.  Each filtration is eventually zero in each
bidegree, so this implies that $(\id \otimes \gamma'_n) \lambda_n$
is an isomorphism.  Inverting $\xi_1^{\ell^n}$, it follows that
$$
(\id \otimes \beta'_n) \lambda_n \:
B(n)_{*,*} \overset{\cong}\longto
        A(n)_{*,*} \otimes_{H_{*,*}} H_{*,*}[\xi_1^{\pm \ell^n}]
$$
is also an isomorphism.
\end{proof}

\begin{proposition} \label{prop-right-An-1-comod-iso-Bn}
The composite
$$
B(n)_{*,*}
	\overset{\rho_n}\longto
	B(n)_{*,*} \otimes_{H_{*,*}} A(n-1)_{*,*}
	\overset{\beta_n \otimes \id}\longto 
	B(0)_{*,*} \otimes_{H_{*,*}} A(n-1)_{*,*}
$$
is a right $A(n-1)_{*,*}$-comodule algebra isomorphism.
\end{proposition}

\begin{proof}
The $\bZ/\ell$-algebra homomorphism $(\beta_n \otimes \id) \rho_n$
is left $H_{*,*}$-linear and maps the remaining algebra generators by
$$
\begin{cases}
\tau_0 \longmapsto \tau_0 \otimes 1 + 1 \otimes \tau_0 \\
\tau_k \longmapsto \xi_1^{\ell^{k-1}} \otimes \tau_{k-1} + 1 \otimes \tau_k 
&\text{for $1 \le k \le n$,} \\
\xi_1 \longmapsto \xi_1 \otimes 1 + 1 \otimes \xi_1 \\
\xi_k \longmapsto \xi_1^{\ell^{k-1}} \otimes \xi_{k-1} + 1 \otimes \xi_k
&\text{for $2 \le k \le n$,} \\
\xi_1^{-\ell^n} \longmapsto \xi_1^{-\ell^n} \otimes 1 \,.
\end{cases}
$$
Letting
$$
\begin{cases}
\check\tau_k = \tau_k \cdot \xi_1^{-\ell^{k-1}}
&\text{for $1 \le k \le n$,} \\
\check\xi_k = \xi_k \cdot \xi_1^{-\ell^{k-1}}
&\text{for $2 \le k \le n$,}
\end{cases}
$$
we can rewrite the presentation in Definition~\ref{def:Jn-Cn-Bn} as
$$
B(n)_{*,*} = \frac{ H_{*,*}[\tau_0, \check\tau_1, \dots, \check\tau_n,
	\xi_1^{\pm1}, \check\xi_2, \dots, \check\xi_n] }
{( \tau_0^2 - T_0, \check\tau_1^2 -  \check T_1, \dots,
	\check\tau_n^2 - \check T_n, \check\xi_2^{\ell^{n-1}}, \dots,
	\check\xi_n^\ell) } \,,
$$
where for $1 \le i \le n$ we use the notation
$$
\check T_i = \begin{cases}
\tau\check\xi_{i+1} + \rho\check\tau_{i+1} + \rho\tau_0\check \xi_{i+1}
	&\text{for $\ell=2$,} \\
0 & \text{for $\ell$ odd.}
\end{cases}
$$
Note that $\check\tau_k = \tau_k \cdot \xi_1^{\ell^n - \ell^{k-1}}
\cdot \xi_1^{-\ell^n}$ and $\check\xi_k = \xi_k \cdot \xi_1^{\ell^n -
\ell^{k-1}} \cdot \xi_1^{-\ell^n}$.  Hence $(\beta_n \otimes \id)
\rho_n$ satisfies
\begin{multline*}
\check\tau_k \longmapsto 1 \otimes \tau_{k-1}
  + \xi_1^{-\ell^{k-1}} \otimes \tau_k
  + \dots \\
  + \xi_1^{\ell^{k-1} - \ell^n} \otimes \tau_{k-1} \xi_1^{\ell^n - \ell^{k-1}}
  + \xi_1^{-\ell^n} \otimes \tau_k \xi_1^{\ell^n - \ell^{k-1}}
\end{multline*}
for $1 \le k \le n$, and
\begin{multline*}
\check\xi_k \longmapsto 1 \otimes \xi_{k-1}
  + \xi_1^{-\ell^{k-1}} \otimes \xi_k
  + \dots \\
  + \xi_1^{\ell^{k-1} - \ell^n} \otimes \xi_1^{\ell^n - \ell^{k-1}} \xi_{k-1}
  + \xi_1^{-\ell^n} \otimes \xi_1^{\ell^n - \ell^{k-1}} \xi_k
\end{multline*}
for $2 \le k \le n$.  The omitted summands involve binomial
coefficients, and each summand after the first has a negative power of
$\xi_1$ as its left hand tensor factor.  Hence $(\beta_n \otimes \id)
\rho_n$ respects the increasing filtrations defined (internally to
this proof) for $m \in \bZ$ by
\begin{align*}
F_m B(n)_{*,*} &= H_{*,*} \{
	\tau_0^{e_0} \check\tau_1^{e_1} \dotsm \check\tau_n^{e_n}
	\xi_1^{r_1} \check\xi_2^{r_2} \dotsm \check\xi_n^{r_n}
	\mid e_0 + 2r_1 \le m \} \\
F_m B(0)_{*,*} &= H_{*,*} \{ \tau_0^{e_0} \xi_1^{r_1}
	\mid e_0 + 2r_1 \le m \} \,,
\end{align*}
where $e_s \in \{0,1\}$ for $0 \le s \le n$, $r_1 \in \bZ$
and $0 \le r_s < \ell^{n+1-s}$ for $2 \le s \le n$ as in
Lemma~\ref{left-H-basis-Cn-Bn}(b).  The induced homomorphism
$$
\frac{F_m B(n)_{*,*}}{F_{m-1} B(n)_{*,*}}
	\longto \frac{F_m B(0)_{*,*}}{F_{m-1} B(0)_{*,*}}
		\otimes_{H_{*,*}} A(n-1)_{*,*}
$$
of associated graded right $A(n-1)_{*,*}$-comodules is the left
$H_{*,*}$-module isomorphism given by
$$
\tau_0^{e_0} \check\tau_1^{e_1} \dotsm \check\tau_n^{e_n}
	\xi_1^{r_1} \check\xi_2^{r_2} \dotsm \check\xi_n^{r_n}
\longmapsto
\tau_0^{e_0} \xi_1^{r_1} \otimes \tau_0^{e_1} \dotsm \tau_{n-1}^{e_n}
	\xi_1^{r_2} \dotsm \xi_{n-1}^{r_n}
$$
for $e_0 + 2 r_1 = m$.  In particular, $\check\tau_k \mapsto \tau_{k-1}$
and $\check\xi_k \mapsto \xi_{k-1}$.  Each filtration is exhaustive and
eventually zero in each bidegree, so this implies that $(\beta_n \otimes
\id) \rho_n$ is an isomorphism.
\end{proof}

\section{$\dots$ and their dual bimodules}
\label{sec:dual-bimodules}

We now dualize the results of the previous section, following \cite{Boa82}.

\begin{definition}[\cite{Boa82}*{Def.~3.2}]
Given a left $H_{*,*}$-module $M$ we define the \emph{dual}
left $H_{*,*}$-module to be
$$
M^\vee = \Hom_{H_{*,*}}(M, H_{*,*}) \,.
$$
The left action of $h \in H_{*,*}$ on $f \in M^\vee$ is given by
$$
(hf)(m) = h \, f(m) = (-1)^{|h| |f|} f(hm)
$$
for $m \in M$, where $|h|$ and $|f|$ are the topological
degrees of~$h$ and $f$, respectively.  If~$M$ is an
$H_{*,*}$-$H_{*,*}$-bimodule then $M^\vee$ is also a bimodule, with
right action defined by
$$
(fh)(m) = (-1)^{|h| |m|} f(mh) \,.
$$
\end{definition}

\begin{example}
The canonical isomorphism $H_{*,*}^\vee \cong H_{*,*} = H^{-*, -*}$,
taking $f$ to $f(1)$, is $H_{*,*}$-$H_{*,*}$-bilinear.
\end{example}

\begin{lemma}[\cite{Boa82}*{Lem.~3.3}] \label{lem:Boa82Lem33}
Let $M$ be an $H_{*,*}$-$H_{*,*}$-bimodule and let $N$ be a left
$H_{*,*}$-module.

(a)
There is a natural homomorphism
$\theta \: M^\vee \otimes_{H_{*,*}} N^\vee
	\longto (M \otimes_{H_{*,*}} N)^\vee$
of left $H_{*,*}$-modules (or of $H_{*,*}$-$H_{*,*}$-bimodules, if $N$
is a bimodule), given by
$$
\theta(f \otimes g)(m \otimes n) = (-1)^{|g| |m|} f(m \, g(n))
$$
for $f \in M^\vee$, $g \in N^\vee$, $m \in M$ and~$n \in N$.

(b)
If $L$ is another bimodule, the diagram
$$
\xymatrix{
L^\vee \otimes_{H_{*,*}} M^\vee \otimes_{H_{*,*}} N^\vee
	\ar[r]^-{\theta \otimes \id}
	\ar[d]_-{\id \otimes \theta}
& (L \otimes_{H_{*,*}} M)^\vee \otimes_{H_{*,*}} N^\vee
	\ar[d]^-{\theta} \\
L^\vee \otimes_{H_{*,*}} (M \otimes_{H_{*,*}} N)^\vee
	\ar[r]^-{\theta}
& (L \otimes_{H_{*,*}} M \otimes_{H_{*,*}} N)^\vee
}
$$
commutes.

(c)
Both composites
$M^\vee \cong M^\vee \otimes_{H_{*,*}} H_{*,*}^\vee
\overset{\theta}\longto
(M \otimes_{H_{*,*}} H_{*,*})^\vee = M^\vee$
and
$M^\vee \cong H_{*,*}^\vee \otimes_{H_{*,*}} M^\vee
\overset{\theta}\longto
(H_{*,*} \otimes_{H_{*,*}} M)^\vee = M^\vee$
are the identity homomorphism.
\end{lemma}

\begin{lemma}[\cite{Boa82}*{Lem.~3.4}]
(a)
Let $(H_{*,*}, \Gamma)$ be a Hopf algebroid.  The dual
$\Gamma^\vee$ is a bigraded $\bZ/\ell$-algebra,
containing $H_{*,*}^\vee$ as a subalgebra.

(b)
Let $M$ be a left $\Gamma$-comodule.  The dual $M^\vee$ is
a left $\Gamma^\vee$-module.

(c)
Let $(H_{*,*}, \Sigma)$ be a second Hopf algebroid, and
let $N$ be a $\Gamma$-$\Sigma$-bicomodule.  The dual $N^\vee$ is
a $\Gamma^\vee$-$\Sigma^\vee$-bimodule.
\end{lemma}

\begin{proof}
Let $\psi \: \Gamma \to \Gamma \otimes_{H_{*,*}} \Gamma$ be
the coproduct, and let $\lambda \: M \to \Gamma \otimes_{H_{*,*}} M$
be the left coaction.
Boardman uses Lemma~\ref{lem:Boa82Lem33} to define the multiplication
on $\Gamma^\vee$ as the composite
$$
\Gamma^\vee \otimes \Gamma^\vee
\longto
\Gamma^\vee \otimes_{H_{*,*}} \Gamma^\vee
\overset{\theta}\longto
(\Gamma \otimes_{H_{*,*}} \Gamma)^\vee
\overset{\psi^\vee}\longto
\Gamma^\vee
$$
and to define the left action on $M^\vee$ as the composite
$$
\Gamma^\vee \otimes M^\vee
\longto
\Gamma^\vee \otimes_{H_{*,*}} M^\vee
\overset{\theta}\longto
(\Gamma \otimes_{H_{*,*}} M)^\vee
\overset{\lambda^\vee}\longto
M^\vee \,.
$$
Likewise, we define the bimodule action on $N^\vee$ as the now
evident composite
$$
\Gamma^\vee \otimes N^\vee \otimes \Sigma^\vee
\longto
\Gamma^\vee \otimes_{H_{*,*}} N^\vee \otimes_{H_{*,*}} \Sigma^\vee
\longto
(\Gamma \otimes_{H_{*,*}} N \otimes_{H_{*,*}} \Sigma)^\vee
\longto
N^\vee \,.
$$
The dual $\epsilon^\vee \: H_{*,*}^\vee \to \Gamma^\vee$ of the Hopf
algebroid counit is split by $\eta_L^\vee$ (and by $\eta_R^\vee$),
and exhibits $H_{*,*}^\vee$ as a subalgebra of $\Gamma^\vee$.
\end{proof}

The dual $\bZ/\ell$-algebra $\Gamma^\vee$ is usually
non-commutative.  Switching to cohomological grading, we now
refer to the duals of (left or right) $H_{*,*}$-module actions
as (left or right) $H^{*,*}$-module actions.

\begin{notation}
The motivic Steenrod algebra $\sA = \sA_{*,*}^\vee$ is the dual of
the Hopf algebroid $(H_{*,*}, \sA_{*,*})$, cf.~\cite{Voe03a}*{\S13},
and contains $H^{*,*}$ as a subalgebra.  It is freely generated as
a left $H^{*,*}$-module by the Milnor basis $\{\rho(E,R)\}_{E,R}$,
defined to be dual to the monomial basis $\{\tau^E \xi^R\}_{E,R}$ of
Lemma~\ref{lem:left-H-basis-A}.  The cohomological bidegree of~$\rho(E,R)$
is equal to the homological bidegree of $\tau^E \xi^R$.  In particular,
the Steenrod operation $\beta^e P^r$ is dual to $\tau_0^e \xi_1^r$,
for $e \in \{0,1\}$ and $r\ge0$, cf.~\cite{Voe03a}*{Lem.~13.1,
Lem.~13.5}.  By~\cite{Voe03a}*{Lem.~11.1, Cor.~12.5} and the Adem
relations~\cite{Voe03a}*{Thm.~10.3}, \cite{Rio12}*{Thm.~4.5.1} the
operations $\beta, P^1, P^\ell, P^{\ell^2}, \dots$, together with
the elements of $H^{*,*}$, generate $\sA$ as a $\bZ/\ell$-algebra.
When $\ell=2$ we write $Sq^{2r}$ for $P^r$ in cohomological bidegree~$(2r,
r)$ and $Sq^{2r+1}$ for $\beta P^r$ in cohomological bidegree~$(2r+1,r)$.
\end{notation}

\begin{lemma} \label{lem:right-H-basis-A-cohom}
The operations $\rho(E,R)$, for $(E,R)$ as in
Lemma~\ref{lem:left-H-basis-A}, also form a basis for~$\sA$ as a right
$H^{*,*}$-module.
\end{lemma}

\begin{proof}
Recall the decreasing $H_{*,*}$-$H_{*,*}$-bimodule
filtration $F^t \sA_{*,*}$ of $\sA_{*,*}$ from the proof of
Lemma~\ref{lem:right-H-basis-A}.  For $t\ge0$ let
$$
F_{t-1} \sA = \< \rho(E,R) \mid \deg(\tau^E \xi^R) < t \> \subset \sA
$$
be the left $H^{*,*}$-submodule generated by the operations $\rho(E,R)$
of cohomological topological degree $< t$.  This is also a right
$H^{*,*}$-submodule, in view of the short exact sequence
$$
0 \to F_{t-1} \sA \longto \sA_{*,*}^\vee \longto (F^t \sA_{*,*})^\vee
	\to 0 \,.
$$
Hence $\{F_t \sA\}_t$ is an increasing filtration of $\sA$ by
$H^{*,*}$-$H^{*,*}$-bimodules, with filtration quotients
$$
\gr_t \sA = \frac{F_t \sA}{F_{t-1} \sA} \cong (\gr^t \sA_{*,*})^\vee \,.
$$
Since the left and right $H_{*,*}$-module actions agree on $\gr^t
\sA_{*,*}$, the dual left and right $H^{*,*}$-module actions on $\gr_t
\sA$ are also equal.  Hence the (cosets of the) operations $\rho(E,R)$
of degree~$= t$ freely generate $\gr_t \sA$ as a right $H^{*,*}$-module.
Since the filtration is exhaustive, the set of degree $\ge0$ operations
is a right $H^{*,*}$-module basis for~$\sA$.
\end{proof}

\begin{definition}
For $n\ge-1$ let the $\bZ/\ell$-algebra $A(n) = A(n)_{*,*}^\vee \subset
\sA$ be the dual of the Hopf algebroid $(H_{*,*}, A(n)_{*,*})$.
\end{definition}

\begin{lemma} \label{lem:left-H-basis-An-cohom}
The operations $\rho(E,R)$, for $(E,R)$ as in
Lemma~\ref{lem:left-H-basis-An}, form a basis for $A(n)$ as a finitely
generated free left $H^{*,*}$-module.  In particular, there is an
exhaustive sequence of $\bZ/\ell$-algebra homomorphisms
$$
H^{*,*} \subset \dots \subset A(n-1) \subset A(n) \subset \dots
	\subset \sA \,.
$$
\end{lemma}

\begin{proof}
This follows from (the proof of) Lemma~\ref{lem:left-H-basis-An}, since
$I(n) \subset \sA_{*,*}$ is a monomial ideal.  The sequence is dual to the
tower~\eqref{eq:dualAntower}.
\end{proof}

\begin{lemma}
The operations $\rho(E,R)$, for $(E,R)$ as in
Lemma~\ref{lem:left-H-basis-An}, also form a basis for $A(n)$ as a free
right $H^{*,*}$-module.
\end{lemma}

\begin{proof}
Replace $\sA_{*,*}$ and Lemma~\ref{lem:right-H-basis-A} in the proof
of Lemma~\ref{lem:right-H-basis-A-cohom}
by $A(n)_{*,*}$ and Lemma~\ref{lem:right-H-basis-An}.
\end{proof}

\begin{example}
(a)
$A(0) = H^{*,*} \< \beta \> / (\beta^2)$ with $[\beta, x] = \beta(x)$
for $x \in H^{*,*}$, where $[\beta, x] = \beta x - (-1)^{|x|} x \beta$
denotes the graded commutator.

(b)
For $\ell=2$,
$$
A(1) = \frac{\ds H^{*,*} \< \beta, P^1 \>}{\ds (\beta^2, P^1 P^1 = \tau \beta P^1 \beta, (\beta P^1)^2 = (P^1 \beta)^2)}
$$
with $[\beta, x] = \beta(x)$ and $[P^1, x] = P^1(x)$ for $x \in H^{*,*}$.
In the figure below, each bullet represents a copy of $H^{*,*}$, the
operations $\beta = Sq^1$ and $P^1 = Sq^2$ map one and two columns to
the right, respectively, and the dashed arrow indicates that $P^1 P^1 =
Sq^2 Sq^2$ is $\tau$ times the generator $\beta P^1 \beta = Sq^3 Sq^1$.
$$
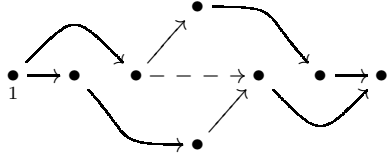
\xymatrix@-1pc{
& & & \bullet \ar@(r,ul)[rrd] \\
\underset{1}\bullet \ar[r] \ar@(ur,ul)[rr]
& \bullet \ar@(dr,l)[rrd]
& \bullet \ar[ur] \ar@{-->}[rr]
& & \bullet \ar@(dr,dl)[rr]
& \bullet \ar[r]
& \bullet \\
& & & \bullet \ar[ur]
}
$$
\end{example}

The following property is sometimes taken as the definition of~$A(n)$.

\begin{lemma} \label{lem:alggen-for-An}
For $n\ge0$ the operations $\beta, P^1, P^\ell, \dots, P^{\ell^{n-1}}$,
together with the elements of $H^{*,*}$, generate $A(n)$ as a
$\bZ/\ell$-algebra.
\end{lemma}

\begin{proof}
For $\ell$ odd, the Adem relations \cite{Voe03a}*{Thm.~10.3} show
that the subalgebra of $A(n)$ generated by $\beta, P^1, P^\ell,
\dots, P^{\ell^{n-1}}$ is isomorphic to the classical finite subalgebra
$A(n)^{\cl}$ of the classical Steenrod algebra.  By~\cite{Mil58}*{Prop.~2}
it has $\bZ/\ell$-module basis equal to the $H^{*,*}$-module basis for
$A(n)$ of Lemma~\ref{lem:left-H-basis-An-cohom}.

For $\ell=2$, the $\tau$- and $\rho$-coefficients in the Adem
relations~\cite{Rio12}*{Thm.~4.5.1} (correcting~\cite{Voe03a}*{Thm.~10.2})
mean that Milnor's product formula~\cite{Mil58}*{Thm.~4b} requires
adjustment in the motivic setting.  For $i\ge0$ let $Q_i$ be the Milnor
basis element dual to~$\tau_i$, and for $i\ge1$ and $j\ge0$ let $P_i^j$
be the Milnor basis element dual to~$\xi_i^{\ell^j}$.  In particular,
$Q_0 = \beta$ and $P_1^j = P^{\ell^j}$.  The arrays
$$
\xymatrix@=0pt{
\xi_1^{\ell^{n-1}} \\
\xi_1^{\ell^{n-2}} & \xi_2^{\ell^{n-2}} \\
\vdots & \vdots & \ddots \\
\xi_1 & \xi_2 & \dots & \xi_n \\
\tau_0 & \tau_1 & \dots & \tau_{n-1} & \tau_n \\
}
\qquad\text{and}\qquad
\xymatrix@=0pt{
P_1^{n-1} \\
P_1^{n-2} & P_2^{n-2} \\
\vdots & \vdots & \ddots \\
P_1^0 & P_2^0 & \dots & P_n^0 \\
Q_0 & Q_1 & \dots & Q_{n-1} & Q_n \\
}
$$
may be helpful, cf.~\cite{Mar83}*{p.~232}.  Let $n\ge1$ and suppose, by
induction, that the lemma holds for $A(n-1)$.  We show that the inclusions
\begin{align*}
A(n-1) \< P_1^{n-1} \> &\subset A(n-1) \< P_1^{n-1}, P_2^{n-2} \>
	\subset \dots \\
	&\subset A(n-1) \< P_1^{n-1}, \dots, P_n^0, Q_n \>
	\subset A(n)
\end{align*}
are all equalities.  Here we write $A(n-1) \< P_1^{n-1}, \dots, P_k^{n-k}
\>$ to denote the subalgebra of $A(n)$ generated by $A(n-1)$ and the
$P_t^{n-t}$ with $1 \le t \le k$, and similarly in the case with~$Q_n$.
This will complete the inductive step, since $A(n-1) \< P_1^{n-1} \>$
is generated by $\beta, P^1, \dots, P^{\ell^{n-2}}, P^{\ell^{n-1}}$
and the elements of $H^{*,*}$.  Consider $2 \le k \le n$.  We claim that
\begin{equation} \label{eq:commutator}
[P_{k-1}^{n+1-k}, P_1^{n-k}] = P_k^{n-k} \mod A(n-1) \,.
\end{equation}
The left hand commutator is an $H^{*,*}$-linear combination
of Milnor basis elements $\rho(E, R)$ in $A(n)$, as in
Lemma~\ref{lem:left-H-basis-An-cohom}.  The $H^{*,*}$-coefficient of
$\rho(E, R)$ is the sum of the $H_{*,*}$-coefficients of
\begin{equation} \label{eq:commutator-duals}
\xi_{k-1}^{\ell^{n+1-k}} \otimes \xi_1^{\ell^{n-k}}
\qquad\text{and}\qquad
\xi_1^{\ell^{n-k}} \otimes \xi_{k-1}^{\ell^{n+1-k}}
\end{equation}
in $\psi(\tau^E \xi^R)$, where we can ignore signs since $\ell=2$.
The basis element~$P_k^{n-k}$ appears with coefficient~$1$, due to
the term $\xi_{k-1}^{\ell^{n+1-k}} \otimes \xi_1^{\ell^{n-k}}$ in
$\psi(\xi_k^{\ell^{n-k}})$.

For other $\rho(E, R)$ not in $A(n-1)$, degree considerations show that
exactly one of $\xi_1^{\ell^{n-1}}, \dots, \xi_{k-1}^{\ell^{n+1-k}}$ must
divide~$\tau^E \xi^R$.  When $1 \le t \le k-2$, no term of the coproduct
$$
\psi(\xi_t^{\ell^{n-t}})
= \sum_{i+j=t} \xi_i^{\ell^{n-i}} \otimes \xi_j^{\ell^{n-t}}
$$
divides either one of the tensor products in~\eqref{eq:commutator-duals}.
Hence the $\rho(E, R)$ with these $\xi_t^{\ell^{n-t}}$
dividing $\tau^E \xi^R$ do not contribute to the commutator
in~\eqref{eq:commutator}.  In the one remaining case, $t = k-1$,
the coproduct $\psi(\xi_{k-1}^{\ell^{n+1-k}})$ contains two terms
dividing those in~\eqref{eq:commutator-duals}, namely $1 \otimes
\xi_{k-1}^{\ell^{n+1-k}}$ and $\xi_{k-1}^{\ell^{n+1-k}} \otimes 1$.
The complementary factors $\xi_1^{\ell^{n-k}} \otimes 1$ and $1 \otimes
\xi_1^{\ell^{n-k}}$ only appear in
$$
\psi(\xi_1^{\ell^{n-k}}) = 1 \otimes \xi_1^{\ell^{n-k}}
	+ \xi_1^{\ell^{n-k}} \otimes 1 \,,
$$
so the last possible contribution to~\eqref{eq:commutator} is
$\rho(E, R)$ dual to $\xi_1^{\ell^{n-k}} \cdot \xi_{k-1}^{\ell^{n+1-k}}$,
with
$H^{*,*}$-coefficient the sum of the $H_{*,*}$-coefficients in
$$
\psi(\xi_1^{\ell^{n-k}} \cdot \xi_{k-1}^{\ell^{n+1-k}})
= \psi(\xi_1^{\ell^{n-k}}) \cdot \psi(\xi_{k-1}^{\ell^{n+1-k}}) \,.
$$
Since each of $\xi_{k-1}^{\ell^{n+1-k}} \otimes \xi_1^{\ell^{n-k}}$
and $\xi_1^{\ell^{n-k}} \otimes \xi_{k-1}^{\ell^{n+1-k}}$ occurs twice
in this product, this last contribution is~$0 \mod \ell$.   This
establishes claim~\eqref{eq:commutator}.  The analogous formula
\begin{equation}
[P_n^0, Q_0] = Q_n
\end{equation}
holds strictly in $A(n)$, and was already proved
in~\cite{Voe03a}*{Prop.~13.6}.  It follows by induction on~$k$
that
$$
A(n-1) \< P_1^{n-1} \> = A(n-1) \< P_1^{n-1}, \dots, P_n^0 \>
	= A(n-1) \< P_1^{n-1}, \dots, P_n^0 , Q_n \> \,.
$$
Finally, the identity
$$
A(n-1) \< P_1^{n-1}, \dots, P_n^0 , Q_n \> = A(n)
$$
follows by classical filtration-by-excess considerations,
as in~\cite{Mar83}*{Prop.~15.8}, where the excess of $\rho(E, R)$
is defined to be $\sum_s e_s + 2 \sum_s r_s$.
\end{proof}

\begin{lemma} \label{lem:left-An-basis-A}
The operations $\rho(E,R)$ for $(E,R)$ as in Definition~\ref{def:Xn}
form a basis for $\sA$ as a free left $A(n)$-module.
\end{lemma}

\begin{proof}
This follows by dualization from Lemma~\ref{lem:left-An-comod-iso-A}.
\end{proof}

\begin{lemma} \label{lem:An-bidegs}
Let $E_n = (1, \dots, 1)$ and $R_n = (\ell^n-1, \dots, \ell-1)$, so
that $t_n = \deg(\tau^{E_n} \xi^{R_n})$ is the highest topological
degree of a monomial in $A(n)_{*,*}$.  Then $A(n)^{p,q} = 0$ unless $0
\le p \le q + d + t_n$.  Hence the subset
$$
\{(e,r) \mid A(n)^{p-e-(2\ell-2)r,q-(\ell-1)r} \ne 0\}
	\subset \{0,1\} \times \bZ
$$
is finite, for each given cohomological bidegree~$(p,q)$.
\end{lemma}

\begin{proof}
The $H_{*,*}$-module generators of $A(n)_{*,*}$ lie in homological
bidegrees $(t,u)$ with $0 \le t \le t_n$ and $u\ge0$.  Hence the
$H^{*,*}$-module generators of $A(n)$ lie in cohomological bidegrees
$(p,q)$ with $0 \le p \le t_n$ and $q\ge0$.  Since $H^{*,*}$ is
concentrated in bidegrees with $0 \le p \le q+d$, it follows that $A(n)$
is concentrated in the infinite triangular region where $0 \le p \le q +
d + t_n$.  Each line of slope $1/2$ in the $(p,q)$-plane intersects this
triangular region in a bounded interval, which implies the finiteness
assertion.
\end{proof}

\begin{definition}
For $n\ge0$ let the $A(n)$-$A(n-1)$-bimodules $C(n) = C(n)_{*,*}^\vee
\subset \sA$ and $B(n) = B(n)_{*,*}^\vee$ be the duals of the
$A(n)_{*,*}$-$A(n-1)_{*,*}$-bicomodules $C(n)_{*,*}$ and $B(n)_{*,*}$,
respectively.  Let the symbol
$$
\rho(E,R) \in B(n)
$$
be dual to $\tau^E \xi^R$ in the monomial left $H_{*,*}$-module basis
for $B(n)_{*,*}$.  The dual of the localization monomorphism $C(n)_{*,*}
\to B(n)_{*,*}$ is a canonical $A(n)$-$A(n-1)$-bimodule epimorphism $B(n)
\to C(n)$.
\end{definition}

\begin{lemma}
(a)
The operations $\rho(E,R)$, for $(E,R)$ as in
Lemma~\ref{left-H-basis-Cn-Bn}(a), form a basis for $C(n)$ as a free
left $H^{*,*}$-module.

(b)
The symbols $\rho(E,R)$, for $(E,R)$ as in
Lemma~\ref{left-H-basis-Cn-Bn}(b), form a basis for $B(n)$ as a free
left $H^{*,*}$-module.

(c)
The canonical epimorphism $B(n) \to C(n)$ satisfies
$$
\rho(E,R) \longmapsto \begin{cases}
\rho(E,R) & \text{for $r_1\ge0$,} \\
0 & \text{for $r_1<0$.}
\end{cases}
$$
\end{lemma}

\begin{proof}
Part~(a) follows from (the proof of) Lemma~\ref{left-H-basis-Cn-Bn}(a),
since $J(n) \subset \sA_{*,*}$ is a monomial ideal.  Part~(b) likewise
follows from Lemma~\ref{left-H-basis-Cn-Bn}(b).  The restriction of
$\rho(E,R)$ to $C(n)_{*,*}$ is then dual to $\tau^E \xi^R$ if $r_1\ge0$,
and zero otherwise, proving~(c).
\end{proof}

\begin{lemma} \label{lem:right-H-basis-Cn-Bn-cohom}
(a)
The operations $\rho(E,R)$, for $(E,R)$ as in
Lemma~\ref{left-H-basis-Cn-Bn}(a), also form a basis for $C(n)$ as a
free right $H^{*,*}$-module.

(b)
The symbols $\rho(E,R)$, for $(E,R)$ as in
Lemma~\ref{left-H-basis-Cn-Bn}(b), also form a basis for $B(n)$ as a
free right $H^{*,*}$-module.
\end{lemma}

\begin{proof}
For part~(a), replace $\sA_{*,*}$ and Lemma~\ref{lem:right-H-basis-A}
in the proof of Lemma~\ref{lem:right-H-basis-A-cohom} by $C(n)_{*,*}$
and Lemma~\ref{lem:right-H-basis-Cn-Bn}(a).

For part~(b), instead replace these by $B(n)_{*,*}$ and
Lemma~\ref{lem:right-H-basis-Cn-Bn}(b), and allow the filtration index~$t$
in the proof of Lemma~\ref{lem:right-H-basis-A-cohom} to run over all
integers, noting that in any given bidegree $F_t B(n) = 0$ for all
sufficiently negative~$t$.
\end{proof}

\begin{example} \label{ex:left-right-H-bases-C0-B0-cohom}
(a)
The Steenrod operations
$$
\{ \beta^e P^r \mid e \in \{0,1\}, r\ge0 \}
$$
form a basis for $C(0) \subset \sA$ as a left $H^{*,*}$-module, and as a
right $H^{*,*}$-module.  When $\ell=2$, these are the Steenrod operations
$Sq^k$ for $k\ge0$.

(b)
The symbols
$$
\{ \beta^e P^r \mid e \in \{0,1\}, r \in \bZ \} \,,
$$
with $\beta^e P^r$ dual to $\tau_0^e \xi_1^r$, form a basis for $B(0)$ as
a left $H^{*,*}$-module, and as a right $H^{*,*}$-module.  When $\ell=2$,
these are the symbols $Sq^k$ for $k \in \bZ$.  The homomorphism $B(0)
\to C(0)$ maps $\beta^e P^r$ to the corresponding Steenrod operation
for $r\ge0$, and to zero for $r < 0$.
\end{example}

\begin{lemma}
For each $n\ge0$ there is a commutative diagram of
$A(n)$-$A(n-1)$-bimodules
$$
\xymatrix{
A(n) \ar@{ >->}[r] \ar@{ >->}[d] & A(n+1) \ar@{ >->}[d] \\
C(n) \ar@{ >->}[r] & C(n+1) \\
B(n) \ar@{ >->}[r] \ar@{->>}[u] & B(n+1) \ar@{->>}[u] \,,
}
$$
where the bimodule structures on the right hand side are obtained by
restriction from the inherent $A(n+1)$-$A(n)$-bimodule structures.
\end{lemma}

\begin{proof}
This is readily obtained by comparing diagram~\eqref{eq:ACB} to its
analogue with~$n$ replaced by~$n+1$, and dualizing.
\end{proof}

\begin{proposition} \label{prop:left-An-mod-basis-Cn-Bn}
(a)
The inclusion $H_{*,*} [\xi_1^{\ell^n}]^\vee \subset C(n)$ extends to
an isomorphism
$$
A(n) \otimes_{H^{*,*}} H_{*,*} [\xi_1^{\ell^n}]^\vee
	\overset{\cong}\longto C(n)
$$
of left $A(n)$-modules.  Hence the Steenrod operations
$$
\{P^r \mid \text{$r\ge0$ with $\ell^n \mid r$} \}
$$
form a basis for $C(n)$ as a free left $A(n)$-module.

(b)
The inclusion $H_{*,*} [\xi_1^{\pm \ell^n}]^\vee \subset B(n)$ extends to
an isomorphism
$$
A(n) \otimes_{H^{*,*}} H_{*,*} [\xi_1^{\pm \ell^n}]^\vee
	\overset{\cong}\longto B(n)
$$
of left $A(n)$-modules.  Hence the symbols
$$
\{P^r \mid \text{$r \in \bZ$ with $\ell^n \mid r$} \}
$$
form a basis for $B(n)$ as a free left $A(n)$-module.  The cohomological
bidegree of $P^r$ is $((2\ell-2)r, (\ell-1)r)$.
\end{proposition}

\begin{proof}
The homomorphisms
$$
\theta \: A(n)_{*,*}^\vee \otimes_{H_{*,*}} H_{*,*} [\xi_1^{\ell^n}]^\vee
\longto (A(n)_{*,*} \otimes_{H_{*,*}} H_{*,*} [\xi_1^{\ell^n}])^\vee
$$
and
$$
\theta \: A(n)_{*,*}^\vee \otimes_{H_{*,*}} H_{*,*} [\xi_1^{\pm \ell^n}]^\vee
\longto (A(n)_{*,*} \otimes_{H_{*,*}} H_{*,*} [\xi_1^{\pm \ell^n}])^\vee
$$
are isomorphisms.  This follows from Lemmas~\ref{lem:Boa82Lem33}(c)
and~\ref{lem:An-bidegs}, since in each case the source of $\theta$ is
a direct sum of shifted copies of $A(n)$, the target of $\theta$ is the
corresponding product, and in each bidegree~$(p,q)$ only finitely many
of the factors in the product are nonzero.

The claims in~(a) and~(b) then follow by dualization from
Proposition~\ref{prop:left-An-comod-isos-Cn-Bn}.
\end{proof}

\begin{proposition} \label{prop-right-An-1-mod-basis-Bn}
The inclusion $B(0) \subset B(n)$ extends to an isomorphism
$$
B(0) \otimes_{H^{*,*}} A(n-1) \overset{\cong}\longto B(n)
$$
of right $A(n-1)$-modules.  Hence the symbols
$$
\{\beta^e P^r \mid e \in \{0,1\}, r \in \bZ \}
$$
form a basis for $B(n)$ as a free right $A(n-1)$-module.  The
cohomological bidegree of $\beta^e P^r$ is $(e + (2\ell-2)r, (\ell-1)r)$.
\end{proposition}

\begin{proof}
The homomorphism
$$
\theta \: B(0)_{*,*}^\vee \otimes_{H_{*,*}} A(n-1)_{*,*}^\vee
\longto (B(0)_{*,*} \otimes_{H_{*,*}} A(n-1)_{*,*})^\vee
$$
is an isomorphism, by Lemmas~\ref{lem:Boa82Lem33}(c)
and~\ref{lem:An-bidegs}.  Thus the claim follows by dualization
from Proposition~\ref{prop-right-An-1-comod-iso-Bn} and
Example~\ref{ex:left-right-H-bases-C0-B0-cohom}(b).
\end{proof}

\section{The small motivic Singer construction}
\label{sec:small-Singer}

In this section and the next, we generalize the classical Singer
construction $R_+(M)$ of \cite{Sin81} and \cite{LS82} to the motivic
context, following the strategy of \cite{AGM85}.  We shall write
$R_S(M)$ for the (small) construction associated to the symmetric
group $S_\ell$, which is denoted $R_+(M)$ in~\cite{LS82} and $T''(M)$
in~\cite{AGM85}, and whose desuspension $\Sigma^{-1} R_S(M)$ is denoted
$R_+(M)$ in~\cite{Sin80} and~\cite{Sin81} and $T'(M)$ in~\cite{AGM85}.
We shall write $R_\mu(M)$ for the (large) construction associated to the
cyclic group $C_\ell$ and the algebraic group $\mu_\ell$ of $\ell$-th
roots of unity, which is denoted $T(M)$ in~\cite{AGM85} and $R_+(M)$
in~\cite{LNR12}.  For $\ell=2$ the two constructions agree.

\begin{lemma} \label{lem:Bn-over-An-1-M}
Let $n\ge0$.

(a)
For each left $A(n-1)$-module $M$, the tensor product
$B(n) \otimes_{A(n-1)} M$ is a left $A(n)$-module.
The inclusion $B(0) \subset B(n)$ induces an isomorphism
$$
B(0) \otimes_{H^{*,*}} M
	\overset{\cong}\longto B(n) \otimes_{A(n-1)} M \,.
$$

(b)
If $M$ is a left $A(n)$-module, then the inclusion $B(n) \subset B(n+1)$
induces an isomorphism
$$
B(n) \otimes_{A(n-1)} M
	\overset{\cong}\longto B(n+1) \otimes_{A(n)} M
$$
of left $A(n)$-modules.

(c)
If $M$ is a left $\sA$-module, then the composition
$B(n) \to C(n) \subset \sA$ induces a left $A(n)$-module
homomorphism
$$
\epsilon_n \: B(n) \otimes_{A(n-1)} M \longto \sA \otimes_{A(n-1)} M
	\longto M \,,
$$
and these are compatible for varying~$n$.
\end{lemma}

\begin{proof}
(a)
This is clear from the $A(n)$-$A(n-1)$-bimodule structure of $B(n)$
and Proposition~\ref{prop-right-An-1-mod-basis-Bn}.

(b)
The morphism exists because $B(n) \subset B(n+1)$ is an
$A(n)$-$A(n-1)$-bimodule homomorphism, with respect to the
restricted bimodule structure on the target.  It is an isomorphism
by comparison with the isomorphisms of part~(a) for $n$ and $n+1$.

(c)
This follows because the inclusions $C(n) \subset C(n+1) \subset \sA$
are $A(n)$-$A(n-1)$-bimodule homomorphisms.  In each case the morphism
$\sA \otimes_{A(n-1)} M \longto M$ is induced by the left module action
$\sA \otimes M \to M$.
\end{proof}

\begin{definition}
Let $M$ be any left $\sA$-module.

(a)
Let the \emph{small motivic Singer construction}
$$
R_S(M) = \colim_n \, (B(n) \otimes_{A(n-1)} M)
$$
be the colimit of the sequence of isomorphisms
$$
B(0) \otimes_{H^{*,*}} M
	\overset{\cong}\longto \dots
	\overset{\cong}\longto B(n) \otimes_{A(n-1)} M
	\overset{\cong}\longto B(n+1) \otimes_{A(n)} M
	\overset{\cong}\longto \dots \,,
$$
equipped with the unique left $\sA$-module structure for which the
canonical map $B(n) \otimes_{A(n-1)} M \to R_S(M)$ is an isomorphism of
$A(n)$-modules, for each $n\ge0$.

(b)
Let the \emph{small evaluation homomorphism}
$$
\epsilon \: R_S(M) \longto M
$$
be the left $\sA$-module homomorphism such that its restriction to
$B(n) \otimes_{A(n-1)} M$ is equal to the $A(n)$-module homomorphism
$\epsilon_n$ of Lemma~\ref{lem:Bn-over-An-1-M}(c), for each $n\ge0$.
\end{definition}

Evidently, $R_S$ is an exact and colimit-preserving endofunctor
of left $\sA$-modules, and $\epsilon \: R_S \to \id$ is a natural
transformation.

\begin{lemma}
As a left $A(0)$-module, the small motivic Singer construction is
given by the tensor product
\begin{align*}
R_S(M) &\cong B(0) \otimes_{H^{*,*}} M \\
&= H^{*,*}\{\beta^e P^r \mid e \in \{0,1\}, r \in \bZ\}
	\otimes_{H^{*,*}} M \,,
\end{align*}
with the $A(0)$-action from $B(0)$.
Each element of $R_S(M)$ is thus a finite sum of terms $\beta^e
P^r \otimes m$, with $e \in \{0,1\}$, $r \in \bZ$ and $m \in M$, where
$\beta(P^r \otimes m) = \beta P^r \otimes m$ and $\beta (\beta P^r
\otimes m) = 0$.  The small evaluation homomorphism is given by
$$
\epsilon(\beta^e P^r \otimes m) =
\begin{cases}
\beta^e P^r(m) & \text{for $r\ge0$,} \\
0 & \text{for $r<0$.}
\end{cases}
$$
\end{lemma}

\begin{proof}
Clear.
\end{proof}

The following formulas generalize the one of Singer~\cite{Sin80}*{(2.1)}
for $\ell=2$ and a rewriting of the those of Li--Singer~\cite{LS82}*{\S3}
for $\ell$ odd.  By $\tau^{j\bmod2}$ we mean $\tau^0 = 1$ for $j$ even
and~$\tau^1 = \tau$ for $j$ odd.

\begin{proposition} \label{prop:Sqa-Pa-action-on-RSigmaM}
For $\ell=2$ and $a\ge0$ even the action of $Sq^a$ on $R_S(M)$ is given by
$$
Sq^a(Sq^b \otimes m) = \sum_{j=0}^{[a/2]}
	\binom{b-1-j}{a-2j} \tau^{j\bmod2} \cdot Sq^{a+b-j} \otimes Sq^j(m)
$$
for $b \in \bZ$ even, and
\begin{align*}
Sq^a(Sq^b \otimes m) &= \sum_{j=0}^{[a/2]}
	\binom{b-1-j}{a-2j} Sq^{a+b-j} \otimes Sq^j(m) \\
	&\qquad + \sum_{\substack{j=1 \\ \text{odd}}}^{[a/2]}
	\binom{b-1-j}{a-2j} \rho \cdot Sq^{a+b-j-1} \otimes Sq^j(m)
\end{align*}
for $b \in \bZ$ odd.

For $\ell$ odd and $a\ge0$ the action of $P^a$ on $R_S(M)$ is given by
$$
P^a(P^b \otimes m) = \sum_{j=0}^{[a/\ell]} (-1)^{a+j}
    \binom{(\ell-1)(b-j)-1}{a-\ell j} P^{a+b-j} \otimes P^j(m) \\
$$
and
\begin{align*}
P^a(\beta P^b \otimes m) &= \sum_{j=0}^{[a/\ell]} (-1)^{a+j}
    \binom{(\ell-1)(b-j)}{a-\ell j} \beta P^{a+b-j} \otimes P^j(m) \\
    &\qquad + \sum_{j=0}^{[(a-1)/\ell]} (-1)^{a+j-1}
    \binom{(\ell-1)(b-j)-1}{a-\ell j-1} P^{a+b-j} \otimes \beta P^j(m)
\end{align*}
for all $b \in \bZ$.
\end{proposition}

\begin{proof}
For $a=0$ the formulas confirm that $Sq^0$ and $P^0$ are the
identity operations.

For $\ell=2$ and $a>0$ even, choose $n$ so that $Sq^a \in A(n)$.
Then $Sq^j \in A(n-1)$ for all $0 \le j \le [a/2]$, and $Sq^i Sq^j
\otimes m = Sq^i \otimes Sq^j(m)$ in $B(n) \otimes_{A(n-1)} M$.
When $a < 2b$ the formulas for $Sq^a(Sq^b \otimes m)$ then follow
from the Adem relations~\cite{Voe03a}*{Thm.~10.2} for $Sq^a Sq^b$,
as corrected in~\cite{Rio12}*{Thm.~4.5.1}.

Similarly, for $\ell$ odd and $a>0$, choose $n$ so that $P^a \in A(n)$.
Then $P^j, \beta P^j \in A(n-1)$ for all $0 \le j \le [a/\ell]$, and
$P^i P^j \otimes m = P^i \otimes P^j(m)$ and $P^i \beta P^j \otimes m =
P^i \otimes \beta P^j(m)$ in $B(n) \otimes_{A(n-1)} M$.  When $a < \ell
b$ the formulas for $P^a(P^b \otimes m)$ and $P^a(\beta P^b \otimes m)$
then follow from the Adem relations~\cite{Voe03a}*{Thm.~10.3} for $P^a
P^b$ and $P^a \beta P^b$.  (The last Adem relation is valid for
$0 < a \le \ell b$, cf.~\cite{Rio12}*{Thm.~4.5.2}.)

For the rest of the argument, $\ell$ can be even or odd.  By
Definition~\ref{def:Bn-An-An-1-bicomod}, the left $A(n)_{*,*}$-coaction
on $B(n)_{*,*}$ commutes with multiplication by $\xi_1^{\ell^n}$, so the
left $A(n)$-action on $B(n)$ commutes with the operation $\beta^e P^r
\mapsto \beta^e P^{r+\ell^n}$.  All mod~$\ell$ binomial coefficients in
sight also repeat $\ell^n$-periodically in~$b$.  Hence the formulas for
$a \ge \ell b$ follow from those for $a < \ell b$.
\end{proof}

\begin{corollary} \label{cor:A-action-on-RSigmaH}
For $\ell=2$ and $a\ge0$ even the action of $Sq^a$ on $R_S(H^{*,*})
\cong B(0)$ is given by
$$
Sq^a(Sq^b) = \binom{b-1}{a} Sq^{a+b}
$$
for $b \in \bZ$.

For $\ell$ odd and $a\ge0$ the action of $P^a$ on $R_S(H^{*,*})
\cong B(0)$ is given by
$$
P^a(P^b) = (-1)^a \binom{(\ell-1)b-1}{a} P^{a+b}
$$
and
$$
P^a(\beta P^b) = (-1)^a \binom{(\ell-1)b}{a} \beta P^{a+b}
$$
for $b \in \bZ$.
\end{corollary}

\begin{proof}
This is the special case $M = H^{*,*}$ of
Proposition~\ref{prop:Sqa-Pa-action-on-RSigmaM}, where we identify
$R_S(H^{*,*}) \cong B(0) \otimes_{H^{*,*}} H^{*,*} = B(0)$ and note
that $Sq^j(1) = 0$ and $P^j(1) = 0$ in $H^{*,*}$ for all $j>0$.
When $\ell=2$, the formulas for $P^a(P^b)$ and $P^a(\beta P^b)$ agree
with the given formulas for $Sq^{2a}(Sq^{2b})$ and $Sq^{2a}(Sq^{2b+1})$,
since $\binom{b-1}{a} \equiv \binom{2b-1}{2a}$ and $\binom{b}{a} \equiv
\binom{2b}{2a} \mod 2$.
\end{proof}

\begin{notation} \label{not:H-of-Bmu-Bsigma}
Let $B\mu_\ell$ and $BS_\ell$ be the geometric classifying spaces of
the linear algebraic groups~$\mu_\ell$ and $S_\ell$, respectively.
In particular, $B\mu_\ell \simeq \hocolim_n L^{2n-1}$ as
discussed in Section~\ref{sec:tower-Thom}.
Recall from~\cite{Voe03a}*{Thm.~6.10, Thm.~6.16} that
$$
H^{*,*}(B\mu_\ell) = H^{*,*} [u, v] / (u^2 = \tau v + \rho u)
$$
with $\beta(u) = v$, and
$$
H^{*,*}(BS_\ell) = H^{*,*} [c, d] / (c^2 = \tau d + \rho c)
$$
with $\beta(c) = d$, as graded commutative $\sA$-module
$H^{*,*}$-algebras.  The cohomological bidegrees of $u$, $v$, $c$
and~$d$ are $(1,1)$, $(2,1)$, $(2\ell-3,\ell-1)$ and~$(2\ell-2, \ell-1)$,
respectively.  The coefficients $\tau$ and $\rho$ are interpreted as $0$ when
$\ell$ is odd.  Any choice of a primitive $\ell$-th root of unity~$\zeta$
defines a map $p_\zeta \: B\mu_\ell \to BS_\ell$ inducing
\begin{align*}
p_\zeta^* \: c &\longmapsto - u v^{\ell-2} \\
p_\zeta^* \: d &\longmapsto - v^{\ell-1} \,.
\end{align*}
We suppress $p_\zeta^*$ from the notation, viewing $H^{*,*}(BS_\ell)$
as an $\sA$-module subalgebra of $H^{*,*}(B\mu_\ell)$.  The natural left
$\sA$-module structure on $H^{*,*}(B\mu_\ell)$ is determined by the cases
\begin{align*}
\beta^e P^r(u) &= \begin{cases}
u & \text{for $(e,r) = (0,0)$,} \\
v & \text{for $(e,r) = (1,0)$,} \\
0 & \text{otherwise,}
\end{cases} \\
\beta^e P^r(v) &= \begin{cases}
v & \text{for $(e,r) = (0,0)$,} \\
v^\ell & \text{for $(e,r) = (0,1)$,} \\
0 & \text{otherwise,}
\end{cases}
\end{align*}
and the Cartan formula~\cite{Voe03a}*{Prop.~9.7},
leading to the expressions
\begin{align*}
P^r(u v^k) &= \binom{k}{r} u v^{(\ell-1)r+k} \\
\beta P^r(u v^k) &= \binom{k}{r} v^{(\ell-1)r+1+k} \\
P^r(v^k) &= \binom{k}{r} v^{(\ell-1)r+k} \\
\beta P^r(v^k) &= 0 \,.
\end{align*}
The restricted $\sA$-module action on $H^{*,*}(BS_\ell)$ is
given by
\begin{align*}
P^r(c d^k) &= (-1)^r \binom{(\ell-1)(k+1)-1}{r} c d^{r+k} \\
\beta P^r(c d^k) &= (-1)^r \binom{(\ell-1)(k+1)-1}{r} d^{r+1+k} \\
P^r(d^k) &= (-1)^r \binom{{(\ell-1)k}}{r} d^{r+k} \\
\beta P^r(d^k) &= 0 \,,
\end{align*}
for $r\ge0$ and $k\ge0$, cf.~\cite{Rio12}*{Prop.~4.4.6}.
\end{notation}

In particular, $\beta(v^{\ell^n}) = 0$ and $P^r(v^{\ell^n}) = 0$
for all $0 < r < \ell^n$, so multiplication by~$v^{\ell^n}$
acts left $A(n)$-linearly on $H^{*,*}(B\mu_\ell)$,
cf.~Lemma~\ref{lem:alggen-for-An}.  Likewise, multiplication by
$d^{\ell^n}$ acts left $A(n)$-linearly on $H^{*,*}(BS_\ell)$.
Hence the following two localizations inherit compatible left
$A(n)$-module structures for all $n\ge0$.  These combine to well-defined
left $\sA$-module structures, such that the localization homomorphisms
are maps of $\sA$-module $H^{*,*}$-algebras.

\begin{definition} \label{def:H-Bmu-BSigma-loc}
Let
\begin{align*}
H^{*,*}(B\mu_\ell)_{\loc} = H^{*,*}(B\mu_\ell) [1/v]
	&= H^{*,*} [u, v^{\pm1}] / (u^2 = \tau v + \rho u) \\
	&= H^{*,*} \{u^i v^k \mid i \in \{0,1\}, k \in \bZ\}
\end{align*}
and
\begin{align*}
H^{*,*}(BS_\ell)_{\loc} = H^{*,*}(BS_\ell) [1/d]
	&= H^{*,*} [c, d^{\pm1}] / (c^2 = \tau d + \rho c) \\
	&= H^{*,*} \{c^i d^k \mid i \in \{0,1\}, k \in \bZ\}
\end{align*}
denote the localizations away from~$v$ and~$d$, respectively.
\end{definition}

\begin{theorem} \label{thm:RSigmaH-SigmaHBSigmalloc}
Let $\Sigma = \Sigma^{1,0}$.
There is a left $\sA$-module isomorphism
$$
R_S(H^{*,*}) \overset{\cong}\longto
	\Sigma H^{*,*}(BS_\ell)_{\loc}
$$
defined by
$$
P^k \longmapsto \Sigma c d^{k-1}
\qquad\text{and}\qquad
\beta P^k \longmapsto - \Sigma d^k
$$
for $k \in \bZ$.
The composite
$\Sigma H^{*,*}(BS_\ell)_{\loc} \cong R_S(H^{*,*})
	\overset{\epsilon}\longto H^{*,*}$
is the left $\sA$-linear homomorphism given by
$$
\Sigma cd^{-1} \longmapsto 1
\qquad\text{and}\qquad
\Sigma c^i d^k \longmapsto 0
$$
for $(i,k) \ne (1,-1)$, where $i \in \{0,1\}$ and $k \in \bZ$.
\end{theorem}

\begin{proof}
By Corollary~\ref{cor:A-action-on-RSigmaH},
Notation~\ref{not:H-of-Bmu-Bsigma} and
Definition~\ref{def:H-Bmu-BSigma-loc} the indicated $H^{*,*}$-module
isomorphism maps $P^r(P^k)$ and $P^r(\beta P^k)$ to $P^r(\Sigma c
d^{k-1})$ and $P^r(-\Sigma d^k)$, respectively, for all $r\ge0$ and $k
\in \bZ$.  Moreover, $\beta(\Sigma c d^{k-1}) = - \Sigma d^k$ and $\beta(-
\Sigma d^k) = 0$.  Hence the isomorphism is $\sA$-linear.
The calculation of the composite follows by noting that
$\beta^e P^r(1) = 0$ in $H^{*,*}$ unless $(e,r) = (0,0)$.
\end{proof}

\section{The large motivic Singer construction}
\label{sec:large-Singer}

Our next aim, following~\cite{AGM85}*{\S5}, is to construct the large
Singer construction $R_\mu(M)$ as an extension of $R_S(M)$, with
$R_\mu(H^{*,*}) \cong \Sigma H^{*,*}(B\mu_\ell)_{\loc}$.  We first
note that $H^{*,*}(BS_\ell)_{\loc} \subset H^{*,*}(B\mu_\ell)_{\loc}$
is a pair of graded Frobenius algebras.  These duality structures
provide a conceptual origin for the explicit formulas that appear
in~\cite{AGM85}*{Lem.~5.1}.

\begin{definition} \label{def:residues}
Let the \emph{residue homomorphisms}
\begin{align*}
\res \: \Sigma H^{*,*}(B\mu_\ell)_{\loc} &\longto H^{*,*} \\
\res \: \Sigma H^{*,*}(BS_\ell)_{\loc} &\longto H^{*,*}
\end{align*}
be the left $H^{*,*}$-linear Frobenius forms defined
for $i \in \{0,1\}$ and $k \in \bZ$ by
$$
\res(\Sigma u^i v^k) = \res(\Sigma c^i d^k) = \begin{cases}
1 & \text{for $(i,k) = (1,-1)$,} \\
0 & \text{otherwise.}
\end{cases}
$$
The associated Frobenius pairings
\begin{align*}
\Sigma H^{*,*}(B\mu_\ell)_{\loc}
	\otimes_{H^{*,*}} H^{*,*}(B\mu_\ell)_{\loc}
	&\longto H^{*,*} \\
\Sigma H^{*,*}(BS_\ell)_{\loc}
	\otimes_{H^{*,*}} H^{*,*}(BS_\ell)_{\loc}
	&\longto H^{*,*}
\end{align*}
map $\Sigma x \otimes y$ to $\res(\Sigma x y)$, and the adjoint
$H^{*,*}$-linear homomorphisms
\begin{align*}
\Sigma H^{*,*}(B\mu_\ell)_{\loc}
	&\overset{\cong}\longto (H^{-*,-*}(B\mu_\ell)_{\loc})^\vee \\
\Sigma H^{*,*}(BS_\ell)_{\loc}
	&\overset{\cong}\longto (H^{-*,-*}(BS_\ell)_{\loc})^\vee
\end{align*}
are the isomorphisms given by
\begin{align*}
\Sigma v^k &\longmapsto (u v^{-k-1})^\vee \\
\Sigma u v^{k-1} &\longmapsto (v^{-k})^\vee + \rho \cdot (u v^{-k})^\vee
\end{align*}
and
\begin{align*}
\Sigma d^k &\longmapsto (c d^{-k-1})^\vee \\
\Sigma c d^{k-1} &\longmapsto (d^{-k})^\vee + \rho \cdot (c d^{-k})^\vee
\end{align*}
for $k \in \bZ$.
\end{definition}

\begin{lemma}
The Frobenius forms, the associated Frobenius pairings, and the adjoint
isomorphisms, are all left $\sA$-linear.
\end{lemma}

\begin{proof}
The residue homomorphism in the case of $B\mu_\ell$ is $\sA$-linear,
because for $r>0$ we have $P^r(uv^k) = 0$ whenever $(\ell-1)r+k =
-1$, since
$$
\binom{(\ell-1)(-r)-1}{r} = (-1)^r \binom{\ell r}{r} \equiv 0
	\mod \ell \,.
$$
The case of $BS_\ell$ follows from this, or from the second part of
Theorem~\ref{thm:RSigmaH-SigmaHBSigmalloc}.  The $\sA$-linearity of the
remaining homomorphisms follows formally.
\end{proof}

Recall the cotensor product~$\cotensor$ of comodules,
e.g.~from~\cite{EM66}*{\S2}.

\begin{definition}
Let
$$
R^S(H_{*,*}) = \lim_n \, (B(n)_{*,*} \cotensor_{A(n-1)_{*,*}} H_{*,*})
	\cong B(0)_{*,*}
$$
be the (achieved) limit of the right $A(n-1)_{*,*}$-comodule primitives
in~$B(n)_{*,*}$.  It is a left $A(n)_{*,*}$-comodule algebra for
each~$n\ge0$, and these coactions combine to a completed left
$\sA_{*,*}$-comodule algebra structure.  We write
$$
R^S(H_{*,*}) = H^{*,*}[\tilde\tau, \tilde\xi^{\pm1}] /
	(\tilde\tau^2 = \tau \tilde\xi + \rho \tilde\tau \tilde\xi) \,,
$$
with $\tilde\tau$ and $\tilde\xi$ mapping to $\tau_0$ and $\xi_1$
in $B(0)_{*,*}$, respectively.  Note that $R_S(H^{*,*}) \cong
R^S(H_{*,*})^\vee$, with $\beta^e P^r$ dual to $\tilde\tau^e \tilde\xi^r$
in the monomial basis.
\end{definition}

\begin{lemma} \label{lem:AGM85Lem5.1}
The composite left $\sA$-module isomorphism
$$
R_S(H^{*,*}) \overset{\cong}\longto \Sigma H^{*,*}(BS_\ell)_{\loc}
	\overset{\cong}\longto (H^{-*,-*}(BS_\ell)_{\loc})^\vee
$$
is the dual of the $H^{*,*}$-algebra isomorphism
$$
\Phi \: H^{-*,-*}(BS_\ell)_{\loc} \overset{\cong}\longto R^S(H_{*,*})
$$
given by
$$
c \longmapsto - \tilde\tau \tilde\xi^{-1} + \rho \cdot 1
\qquad\text{and}\qquad
d \longmapsto \tilde\xi^{-1} \,.
$$
\end{lemma}

\begin{proof}
The composite isomorphism maps $P^k$ to $(d^{-k})^\vee + \rho \cdot
(cd^{-k})^\vee$ and maps $\beta P^k$ to $- (cd^{-k-1})^\vee$, hence is
dual to the $H^{*,*}$-linear homomorphism mapping $d^{-k}$ to $(P^k)^\vee
= \tilde\xi^k$ and mapping $cd^{-k}$ to $- (\beta P^{k-1})^\vee +  \rho
\cdot (P^k)^\vee = - \tilde\tau \tilde\xi^{k-1} + \rho \cdot \tilde\xi^k$.
This is indeed an algebra isomorphism.
\end{proof}

For a left $\sA_{*,*}$-comodule $M_{*,*}$, the
$A(n)_{*,*}$-$A(n-1)_{*,*}$-bicomodule algebra product~$\phi$ on
$B(n)_{*,*}$ induces a pairing
\begin{multline*}
(B(n)_{*,*} \cotensor_{A(n-1)_{*,*}} H_{*,*}) \otimes_{H_{*,*}}
(B(n)_{*,*} \cotensor_{A(n-1)_{*,*}} M_{*,*}) \\
\longto
(B(n)_{*,*} \cotensor_{A(n-1)_{*,*}} M_{*,*})
\end{multline*}
of left $A(n)_{*,*}$-comodules for each $n\ge0$, making
$$
R^S(M_{*,*}) = \lim_n \, (B(n)_{*,*} \cotensor_{A(n-1)_{*,*}} M_{*,*})
$$
an $R^S(H_{*,*})$-module in completed left $\sA_{*,*}$-comodules.
Viewing $R^S(M_{*,*})$ as an $H^{-*,-*}(BS_\ell)_{\loc}$-module
via the algebra isomorphism~$\Phi$, we can form the induced
$H^{-*,-*}(B\mu_\ell)_{\loc}$-module
$$
R^\mu(M_{*,*}) = H^{-*,-*}(B\mu_\ell)_{\loc}
	\underset{H^{-*,-*}(BS_\ell)_{\loc}}{\otimes} R^S(M_{*,*})
\,.
$$
As a left $A(n)_{*,*}$-comodule, it is isomorphic to a finite direct sum
$$
R^\mu(M_{*,*}) \cong H^{*,*}\{1, v^{\ell^n}, \dots, v^{\ell^n(\ell-2)}\}
	\otimes_{H^{*,*}} R^S(M_{*,*}) \,,
$$
where each power of $v^{\ell^n}$ is $A(n)_{*,*}$-comodule primitive.

Dually, for a left $\sA$-module~$M$ the completed $A(n)$-$A(n-1)$-bimodule
coproduct
$$
B(n) = B(n)_{*,*}^\vee \overset{\phi^\vee}\longto (B(n)_{*,*} \otimes_{H_{*,*}} B(n)_{*,*})^\vee = B(n) \hatotimes_{H^{*,*}} B(n)
$$
induces a ``copairing''
$$
B(n) \otimes_{A(n-1)} M \longto (B(n) \otimes_{A(n-1)} H^{*,*})
	\hatotimes_{H^{*,*}} (B(n) \otimes_{A(n-1)} M)
$$
of left $A(n)$-modules for each~$n\ge0$, making the small Singer
construction
$$
R_S(M) = \colim_n \, (B(n) \otimes_{A(n-1)} M)
$$
a completed $R_S(H^{*,*})$-comodule in left $\sA$-modules.
Here $R_S(H^{*,*})$ has the completed $H^{*,*}$-coalgebra
structure dual to the $H_{*,*}$-algebra structure on $R^S(H_{*,*})$
that appears in Lemma~\ref{lem:AGM85Lem5.1}.  It corresponds via the
isomorphism in Theorem~\ref{thm:RSigmaH-SigmaHBSigmalloc} to a completed
$H^{*,*}$-coalgebra structure on $\Sigma H^{*,*}(BS_\ell)_{\loc}$.
Moreover, the algebra inclusion $H^{-*,-*}(BS_\ell)_{\loc} \subset
H^{-*,-*}(B\mu_\ell)_{\loc}$ in left $\sA$-modules corresponds under
duality and the Frobenius isomorphisms from Definition~\ref{def:residues}
to a completed $H^{*,*}$-coalgebra epimorphism
$$
\pi \: \Sigma H^{*,*}(B\mu_\ell)_{\loc}
	\onto \Sigma H^{*,*}(BS_\ell)_{\loc}
$$
in left $\sA$-modules, given by
$$
\Sigma v^{(\ell-1)k} \longmapsto (-1)^k \Sigma d^k
\qquad\text{and}\qquad
\Sigma u v^{(\ell-1)k-1} \longmapsto (-1)^k \Sigma c d^{k-1} \,,
$$
while the remaining $H^{*,*}$-module generators $\Sigma u^i v^k$ with
$i \in \{0,1\}$ and $k \in \bZ$ map to zero.
This discussion motivates the following definition.

\begin{definition}
Let $M$ be any left $\sA$-module.

(a) Let the \emph{large motivic Singer construction}
$$
R_\mu(M) = \Sigma H^{*,*}(B\mu_\ell)_{\loc}
	\underset{\Sigma H^{*,*}(BS_\ell)_{\loc}}{\cotensor} R_S(M)
$$
be the left $\sA$-module coinduced from~$R_S(M)$ along the
completed $H^{*,*}$-coalgebra epimorphism~$\pi$.
As a left $A(n)$-module it is isomorphic to the finite direct sum
$$
R_\mu(M) \cong H^{*,*}\{1, v^{\ell^n}, \dots, v^{\ell^n(\ell-2)}\}
	\otimes_{H^{*,*}} R_S(M)
$$
where $A(n)$ acts trivially, i.e., via $\eta_{L,n}^\vee \: A(n) \to
H^{*,*}$, on each power of~$v^{\ell^n}$.

(b) Let the \emph{large evaluation homomorphism}
$$
\epsilon \: R_\mu(M) \longto M
$$
be the composite $\epsilon (\pi \cotensor 1) \: R_\mu(M) \to R_S(M) \to M$.
\end{definition}

\begin{corollary} \label{cor:RmuH-SigmaHBmulloc}
There is a left $\sA$-module isomorphism
$$
R_\mu(H^{*,*}) \cong \Sigma H^{*,*}(B\mu_\ell)_{\loc} \,.
$$
The composite $\Sigma H^{*,*}(B\mu_\ell)_{\loc} \cong R_\mu(H^{*,*})
\overset{\epsilon}\longto H^{*,*}$ equals the residue homomorphism
for $B\mu_\ell$.
\end{corollary}

\begin{proof}
This follows directly from Theorem~\ref{thm:RSigmaH-SigmaHBSigmalloc}.
\end{proof}

\begin{lemma}
As a left $A(0)$-module, the large motivic Singer construction
is given by the tensor product
$$
R_\mu(M) \cong H^{*,*}\{ \Sigma u^i v^k \mid i \in \{0,1\}, k \in \bZ\}
	\otimes_{H^{*,*}} M
$$
with the $A(0)$-action from $\Sigma H^{*,*}(B\mu_\ell)_{\loc}$.
Each element of $R_\mu(M)$ is thus a finite sum of terms $\Sigma
u^i v^k \otimes m$, with $i \in \{0,1\}$, $k \in \bZ$ and $m \in M$,
where $\beta(\Sigma u v^k \otimes m) = - \Sigma v^{k+1} \otimes m$
and $\beta(\Sigma v^k \otimes m) = 0$.
\end{lemma}

\begin{proof}
Clear.
\end{proof}

The following formulas generalize the classical one of
Singer~\cite{Sin81}*{(3.2)} for $\ell=2$, and of Lun{\o}e--Nielsen and
the second author~\cite{LNR12}*{Def.~3.1} for $\ell$ odd.  The latter two
formulas were surely known to the authors of~\cite{AGM85}.

\begin{proposition}
For $\ell=2$ and $r\ge0$ the action of $Sq^{2r}$ on $R_\mu(M) \cong
R_S(M)$ is given by
\begin{multline*}
Sq^{2r}(\Sigma uv^k \otimes m) = \sum_{j=0}^{[r/2]}
    \binom{k-j}{r-2j}
	\Sigma u v^{r+k-j} \otimes Sq^{2j}(m) \\
+ \sum_{j=0}^{[(r-1)/2]}
    \binom{k-j}{r-2j-1}
	\tau \cdot \Sigma v^{r+k-j} \otimes Sq^{2j+1}(m)
\end{multline*}
and
\begin{multline*}
Sq^{2r}(\Sigma v^k \otimes m) = \sum_{j=0}^{[r/2]}
    \binom{k-j}{r-2j}
	\Sigma v^{r+k-j} \otimes Sq^{2j}(m) \\
+ \sum_{j=0}^{[(r-1)/2]}
    \binom{k-j-1}{r-2j-1}
	\Sigma (u + \rho) v^{r+k-j-1} \otimes Sq^{2j+1}(m) \,.
\end{multline*}
Here $\Sigma (u + \rho) v^{r+k-j-1} = \Sigma u v^{r+k-j-1}
+ \rho \cdot \Sigma v^{r+k-j-1}$.

For $\ell$ odd and $r\ge0$ the action of $P^r$ on $R_\mu(M)$ is given by
$$
P^r(\Sigma u v^{k-1} \otimes m) = \sum_{j=0}^{[r/\ell]}
    \binom{k - (\ell-1)j-1}{r-\ell j}
	\Sigma u v^{k+(\ell-1)(r-j)-1} \otimes P^j(m)
$$
and
\begin{align*}
P^r(\Sigma v^k \otimes m) &= \sum_{j=0}^{[r/\ell]}
    \binom{k-(\ell-1)j}{r-\ell j}
	\Sigma v^{k+(\ell-1)(r-j)} \otimes P^j(m) \\
    &\qquad + \sum_{j=0}^{[(r-1)/\ell]}
    \binom{k-(\ell-1)j-1}{r-\ell j-1}
	\Sigma u v^{k+(\ell-1)(r-j)-1} \otimes \beta P^j(m) \,.
\end{align*}
\end{proposition}

\begin{proof}
For $\ell=2$, the formulas are obtained from
Proposition~\ref{prop:Sqa-Pa-action-on-RSigmaM} by replacing
$Sq^{2k}$ and $Sq^{2k+1}$ by $\Sigma cd^{k-1} = \Sigma
uv^{k-1}$ and $-\Sigma d^k = \Sigma v^k$, respectively, as in
Theorem~\ref{thm:RSigmaH-SigmaHBSigmalloc}.  The summations over $0
\le j \le [a/2]$ split into two cases, according to the parity of~$j$,
and the resulting terms can be collected as shown.

For $\ell$ odd, we first rewrite $R_S(M)$ as
$$
\Sigma H^{*,*}(BS_\ell)_{\loc} \otimes_{H^{*,*}} M
	= H^{*,*} \{ \Sigma c^i d^k \mid i \in \{0,1\}, k \in \bZ \}
	\otimes_{H^{*,*}} M \,,
$$
replacing $P^k$ and $\beta P^k$ by $\Sigma cd^{k-1}$ and $-\Sigma d^k$,
respectively.  For $r\ge0$ the action of $P^r$ on $R_S(M)$ is then given by
$$
P^r(\Sigma (cd^{-1}) (-d)^b \otimes m) = \sum_{j=0}^{[r/\ell]}
    \binom{(\ell-1)(b-j)-1}{r-\ell j} \Sigma (cd^{-1}) (-d)^{r+b-j} \otimes P^j(m) \\
$$
and
\begin{align*}
P^r(\Sigma (-d)^b \otimes m) &= \sum_{j=0}^{[r/\ell]}
    \binom{(\ell-1)(b-j)}{r-\ell j} \Sigma (-d)^{r+b-j} \otimes P^j(m) \\
    &\qquad + \sum_{j=0}^{[(r-1)/\ell]}
    \binom{(\ell-1)(b-j)-1}{r-\ell j-1} \Sigma (cd^{-1}) (-d)^{r+b-j} \otimes \beta P^j(m)
\end{align*}
for $b \in \bZ$.
Substituting $cd^{-1} = uv^{-1}$, $-d = v^{\ell-1}$ and $k = (\ell-1)b$
we obtain the claimed formulas, in the cases where $k$ is a multiple of
$\ell-1$.  The general cases follow, since for $n$ so large that $0 < r
< \ell^n$ the action of $P^r$ on $R_\mu(M)$ commutes with multiplication
by $v^{\ell^n}$, and $\ell^n$ is relatively prime to $\ell-1$.
All mod~$\ell$ binomial coefficients in sight are $\ell^n$-periodic as
functions of~$k \in \bZ$.
\end{proof}

\section{The evaluations are $\Ext$-equivalences}
\label{sec:evaluations}

We can now adapt \cite{AGM85}*{Lem.~2.2} to the motivic setting.  We
sidestep their use of $\bF_p \otimes_{\sA} (-)$ and $\Tor$-equivalences,
since $H^{*,*}$ is not naturally a right $\sA$-module, and pass directly
to $\Hom_{\sA}(-, H^{*,*})$ and $\Ext$-equivalences.

\begin{lemma} \label{lem:R+-Afree-flat}
Let $M$ be a free left $\sA$-module.
The Singer constructions $R_S(M)$ and $R_\mu(M)$ are free as left
$A(n)$-modules, for each~$n$, and flat as left $\sA$-modules.
\end{lemma}

\begin{proof}
Since $M$ is left $\sA$-free, it is left $A(n-1)$-free by
Lemma~\ref{lem:left-An-basis-A}.  Hence the left $A(n)$-module
$R_S(M) = B(n) \otimes_{A(n-1)} M$ is a direct sum of (suitably
suspended) copies of $B(n)$, each of which is left $A(n)$-free by
Proposition~\ref{prop:left-An-mod-basis-Cn-Bn}(b).  Therefore $R_S(M)$
is left $A(n)$-free for each $n$, so that
$$
\Tor^{\sA}_s(K, R_S(M))
	\cong \colim_n \Tor^{A(n)}_s(K, R_S(M)) = 0
$$
for each right $\sA$-module $K$ and $s\ge1$.  Equivalently, $R_S(M)$
is left $\sA$-flat.

Moreover, $R_\mu(M)$ is a direct sum, as a left $A(n)$-module, of
copies of $R_S(M)$.  Hence it is also left $A(n)$-free for each $n$,
and therefore left $\sA$-flat, by the same argument as before.
\end{proof}

\begin{proposition} \label{prop:RSigma-RC-Afree}
Let $M$ be a free left $\sA$-module.
The evaluation homomorphisms induce isomorphisms
\begin{align*}
\Hom(\epsilon, \id) \: \Hom_{\sA}(M, H^{*,*})
	&\overset{\cong}\longto \Hom_{\sA}(R_S(M), H^{*,*}) \\
\Hom(\epsilon, \id) \: \Hom_{\sA}(M, H^{*,*})
	&\overset{\cong}\longto \Hom_{\sA}(R_\mu(M), H^{*,*}) \,.
\end{align*}
\end{proposition}

\begin{proof}
It suffices to consider the case $M = \sA = \colim_n A(n-1)$.  Then
\begin{align*}
\Hom_{A(n)}(R_S(A(n-1)), H^{*,*})
	&= \Hom_{A(n)}(B(n) \otimes_{A(n-1)} A(n-1), H^{*,*}) \\
	&= \Hom_{A(n)}(B(n), H^{*,*}) \\
	&\cong \Hom_{H^{*,*}}(H^{*,*}\{ P^r \mid
		\text{$r \in \bZ$ with $\ell^n \mid r$}\}, H^{*,*}) \\
	&\cong \prod_{\ell^n \mid r} H^{*,*}\{\xi_1^r\}
\end{align*}
by Proposition~\ref{prop:left-An-mod-basis-Cn-Bn}(b), where we identify
$\xi_1^r$ with the dual of the left $A(n)$-module generator~$P^r \in
B(n)$.  The bidegrees of the classes $\xi_1^r$ with $\ell^n \mid r$
are integer multiples of $((2\ell-2)\ell^n, (\ell-1)\ell^n)$.  In any
fixed bidegree, only the factor generated by $\xi_1^0$ can make a nonzero
contribution to this product when $n$ is sufficiently large.  Hence
\begin{align*}
\Hom_{\sA}(R_S(\sA), H^{*,*})
	&\cong \lim_n \Hom_{A(n)}(R_S(A(n-1)), H^{*,*}) \\
	&\cong \lim_n \prod_{\ell^n \mid r} H^{*,*}\{\xi_1^r\}
		\cong H^{*,*} \{\xi_1^0\} \,.
\end{align*}
Since $\epsilon$ maps $P^0 \otimes 1$ to $P^0(1) = 1$, it follows that
$$
H^{*,*} \cong \Hom_{\sA}(\sA, H^{*,*})
	\overset{\Hom(\epsilon,1)}\longto
	\Hom_{\sA}(R_S(\sA), H^{*,*}) \cong H^{*,*} \{\xi_1^0\}
$$
is an isomorphism.

Similarly,
\begin{multline*}
\Hom_{A(n)}(R_\mu(A(n-1)), H^{*,*}) \\
\cong H^{*,*}\{1, v^{\ell^n}, \dots, v^{\ell^n(\ell-2)}\}
	\otimes_{H^{*,*}}
	\prod_{\ell^n \mid r} H^{*,*}\{\xi_1^r\} \,.
\end{multline*}
The bidegrees of the classes $v^{\ell^n j} \otimes \xi_1^r$ with $0 \le
j \le \ell-2$ and $\ell^n \mid r$ are integer multiples of $(2 \ell^n,
\ell^n)$.  Again, in any fixed bidegree, only the factor generated by
$v^0 \otimes \xi_1^0$ can make a nonzero contribution for $n$ sufficiently
large.  Hence
\begin{align*}
\Hom_{\sA}(R_\mu(\sA), H^{*,*})
	&\cong \lim_n \Hom_{A(n)}(R_\mu(A(n-1)), H^{*,*}) \\
	&\cong H^{*,*} \{v^0 \otimes \xi_1^0\} \,.
\end{align*}
Since $\epsilon(\pi \cotensor 1)$ maps $v^0 \otimes P^0$ to
$P^0(1) = 1$, it follows that
$$
H^{*,*} \cong \Hom_{\sA}(\sA, H^{*,*})
        \overset{\Hom(\epsilon,1)}\longto
        \Hom_{\sA}(R_\mu(\sA), H^{*,*}) \cong H^{*,*} \{v^0 \otimes \xi_1^0\}
$$
is an isomorphism.
\end{proof}

The $\Ext$-groups for modules over~$\sA$ or $A(n)$ are trigraded.  In the
case of an $\sA$-module~$M$ we write $\Ext_{\sA}^{s,t,u}(M, H^{*,*})$
for the group in tridegree~$(s,t,u)$, where $s$ is the cohomological
degree and $(t, u)$ is the internal bidegree.

\begin{definition}
An $\sA$-module homomorphism $\theta \: L \to M$ will be said to be an
\emph{$\Ext$-equivalence} if the induced homomorphism
$$
\theta^* \: \Ext_{\sA}^{*,*,*}(M, H^{*,*})
	\longto \Ext_{\sA}^{*,*,*}(L, H^{*,*})
$$
is an isomorphism.
\end{definition}

We can now generalize part of~\cite{AGM85}*{Prop.~1.2, Thm.~1.3}.

\begin{theorem} \label{thm:epsilon-Ext-eqs}
Let $M$ be any left $\sA$-module.  The (small and large) evaluation
homomorphisms
$$
\epsilon \: R_S(M) \longto M
\qquad\text{and}\qquad
\epsilon \: R_\mu(M) \longto M
$$
are $\Ext$-equivalences.
\end{theorem}

\begin{proof}
Let
$$
\dots \longto F_1 \longto F_0 \longto M \to 0
$$
be a free $\sA$-module resolution of~$M$.  Then
$$
\dots \longto R_S(F_1) \longto R_S(F_0) \longto R_S(M) \to 0
$$
is a flat $\sA$-module resolution, and a free $A(n)$-module
resolution for each~$n$, by Lemma~\ref{lem:R+-Afree-flat}.
By Proposition~\ref{prop:RSigma-RC-Afree} the evaluation homomorphism
induces an isomorphism
$$
\Hom(\epsilon,1) \: \Hom_{\sA}(F_s, H^{*,*})
	\overset{\cong}\longto \Hom_{\sA}(R_S(F_s), H^{*,*})
$$
for each $s\ge0$.
Passing to cohomology, it also induces isomorphisms
$$
\epsilon' \: \Ext^s_{\sA}(M, H^{*,*})
	\overset{\cong}\longto H^s(\Hom_{\sA}(R_S(F_*), H^{*,*}))
$$
for all $s\ge0$.  Let
$$
\dots \longto E_1 \longto E_0 \longto R_S(M) \to 0
$$
be a free $\sA$-module resolution of~$R_S(M)$, and choose an $\sA$-module
chain map $\zeta \: E_* \to R_S(F_*)$ over $R_S(M)$.  There is then an
induced map~$\zeta^*$ of $\lim$-$\limone$ short exact sequences, from
\begin{multline*}
0 \to \limone_n H^{s-1}(\Hom_{A(n)}(R_S(F_*), H^{*,*}))
	\longto H^s(\Hom_{\sA}(R_S(F_*), H^{*,*})) \\
	\longto \lim_n H^s(\Hom_{A(n)}(R_S(F_*), H^{*,*})) \to 0
\end{multline*}
to
\begin{multline*}
0 \to \limone_n \Ext^{s-1}_{A(n)}(R_S(M), H^{*,*})
	\longto \Ext^s_{\sA}(R_S(M), H^{*,*}) \\
	\longto \lim_n \Ext^s_{A(n)}(R_S(M), H^{*,*}) \to 0 \,,
\end{multline*}
cf.~\cite{Mar83}*{Prop.~11.9, Prop.~13.4}.  When viewed as an
$A(n)$-module chain map, $\zeta$ becomes a chain homotopy equivalence,
hence induces isomorphisms
$$
\zeta^* \: H^s(\Hom_{A(n)}(R_S(F_*), H^{*,*}))
	\overset{\cong}\longto \Ext^s_{A(n)}(R_S(M), H^{*,*})
$$
for all~$n$ and~$s$.  Applying $\lim_n$ and $\limone_n$, we deduce
that
$$
\zeta^* \: H^s(\Hom_{\sA}(R_S(F_*), H^{*,*}))
	\overset{\cong}\longto \Ext^s_{\sA}(R_S(M), H^{*,*})
$$
is an isomorphism.  Hence the composite $\epsilon^* = \zeta^* \epsilon'$
is also an isomorphism, as claimed.

The proof for $R_\mu$ in place of $R_S$ is identical.
\end{proof}

\begin{corollary} \label{cor:res-Ext-eqs}
The residue homomorphisms
$$
\res \: \Sigma H^{*,*}(BS_\ell)_{\loc} \longto H^{*,*}
\qquad\text{and}\qquad
\res \: \Sigma H^{*,*}(B\mu_\ell)_{\loc} \longto H^{*,*}
$$
are $\Ext$-equivalences.
\end{corollary}

\begin{proof}
In view of Theorem~\ref{thm:RSigmaH-SigmaHBSigmalloc} and
Corollary~\ref{cor:RmuH-SigmaHBmulloc}, this is the case $M = H^{*,*}$
of Theorem~\ref{thm:epsilon-Ext-eqs}.
\end{proof}

\section{Generalized Eilenberg--MacLane spectra}
\label{sec:generalized-EM}

Since $H$ is cellular, the monomial basis $\{\tau^E \xi^R\}_{(E,R)}$
for $\sA_{*,*}$ as a right (or left) $H_{*,*}$-module determines an
equivalence $H \wedge H \simeq \bigvee_{(E,R)} \Sigma^{\|E,R\|} H$
of right (or left) $H$-module spectra.  Here $(E, R)$ ranges over the
sequences in Lemma~\ref{lem:left-H-basis-A}, and $\|E,R\| = \| \tau^E
\xi^R \|$.  It follows that the natural homomorphism
$$
\sA_{*,*} \otimes_{H_{*,*}} H_{*,*}(X)
	\overset{\cong}\longto \pi_{*,*}((H \wedge H) \wedge_H (H \wedge X))
	= \pi_{*,*}(H \wedge H \wedge X)
$$
is an isomorphism for any motivic spectrum~$X$, and that $1 \wedge \eta
\wedge 1 \: H \wedge X = H \wedge \bS \wedge X \to H \wedge H \wedge X$
induces a natural left $\sA_{*,*}$-coaction on $H_{*,*}(X)$.

If $M \simeq H \wedge X$ for some motivic spectrum~$X$, then the fork
$$
\xymatrix{
\pi_{*,*} M \ar[r]^-{\eta} & \pi_{*,*}(H \wedge M)
	\ar@<1ex>[r]^-{\eta\wedge1} \ar@<-1ex>[r]_-{1\wedge\eta}
	& \pi_{*,*}(H \wedge H \wedge M)
}
$$
is split by $\mu \: \pi_{*,*}(H \wedge M) \cong \pi_{*,*}(H \wedge
H \wedge X) \to \pi_{*,*}(H \wedge X) \cong \pi_{*,*} M$ and $\mu \:
\pi_{*,*}(H \wedge H \wedge M) \to \pi_{*,*}(H \wedge M)$, hence exhibits
$\pi_{*,*} M$ as a split equalizer~\cite{ML98}*{\S VI.6}.  Under the
identifications $\pi_{*,*}(H \wedge M) = H_{*,*}(M)$ and $\pi_{*,*}(H \wedge
H \wedge M) \cong \sA_{*,*} \otimes_{H_{*,*}} H_{*,*}(M)$, this provides
an isomorphism
\begin{equation} \label{eq:pi-comodHom}
\pi_{*,*} M \cong H_{*,*} \cotensor_{\sA_{*,*}} H_{*,*}(M)
	\cong \Hom_{\sA_{*,*}}(H_{*,*}, H_{*,*}(M))
\end{equation}
of $\pi_{*,*} M$ with the left $\sA_{*,*}$-comodule primitives
in $H_{*,*}(M)$.

\begin{definition} \label{def:motivic-GEM}
By a \emph{motivic GEM} (short for motivic generalized Eilenberg--MacLane
spectrum) we shall mean a left $H$-module spectrum
$$
M \simeq \bigvee_{\alpha \in J} \Sigma^{p_\alpha, q_\alpha} H
$$
that is equivalent to a wedge sum of bigraded suspensions of $H$.
\end{definition}

These are precisely the $H$-cellular module spectra~$M$, in the sense
of \cite{DI05}*{\S7.9}, with the property that $\pi_{*,*}(M)$ is free
as a left~$H_{*,*}$-module.  This generalizes the split Tate objects of
\cite{Voe10}*{\S2.4}, in that we allow arbitrary bigraded suspensions.

If $M$ is a motivic GEM, we can write $M \simeq H \wedge X$ with $X =
\bigvee_{\alpha \in J} S^{p_\alpha, q_\alpha}$.  Then
$$
H_{*,*}(X) \cong \bigoplus_{\alpha \in J}
	\Sigma^{p_\alpha, q_\alpha} H_{*,*}
$$
as left $\sA_{*,*}$-comodules, and the natural homomorphism
$$
\pi_{*,*} F(X, H) = \pi_{*,*} F_H(H \wedge X, H)
	\overset{\cong}\longto
	\Hom_{H_{*,*}}(H_{*,*}(X), H_{*,*})
	= H_{*,*}(X)^\vee
$$
is an isomorphism, so that
$$
H^{*,*}(X) \cong \prod_{\alpha \in J}
	\Sigma^{p_\alpha, q_\alpha} H^{*,*}
$$
as left $\sA$-modules.

\begin{definition} \label{def:bifinite-type}
With these notations we say that $H \wedge X$ (or $H_{*,*}(X)$, or
$H^{*,*}(X)$) has \emph{bifinite type} if for each bidegree $(p,q)$
there are only finitely many $\alpha \in J$ for which
$$
p_\alpha \le p
\qquad\text{or}\qquad
p_\alpha - q_\alpha \le p-q \,.
$$
\end{definition}

This condition is more restrictive than the notions of motivically
finite type from~\cite{DI10}*{Def.~2.11, Def.~2.12}
and of finite type from~\cite{HKO11}*{\S2}.  It ensures that
both inclusions
\begin{align*}
\bigoplus_{\alpha \in J} \Sigma^{p_\alpha, q_\alpha} H_{*,*}
&\overset{\cong}\longto
\prod_{\alpha \in J} \Sigma^{p_\alpha, q_\alpha} H_{*,*} \\
\bigoplus_{\alpha \in J} \Sigma^{p_\alpha, q_\alpha} H^{*,*}
&\overset{\cong}\longto
\prod_{\alpha \in J} \Sigma^{p_\alpha, q_\alpha} H^{*,*}
\end{align*}
are isomorphisms, so that the canonical homomorphism
$$
H_{*,*}(X) \overset{\cong}\longto \Hom_{H^{*,*}}(H^{*,*}(X), H^{*,*})
	= H^{*,*}(X)^\vee
$$
is an isomorphism.
Moreover, if $H \wedge X$ has bifinite type,
then so does $H \wedge H \wedge X \simeq \bigvee_{(E,R)} \Sigma^{\|E,R\|}
H \wedge X$.  Hence $f \mapsto f^\vee$ defines an
isomorphism
\begin{equation} \label{eq:comodHom-modHom}
\Hom_{\sA_{*,*}}(H_{*,*}, H_{*,*}(X))
	\overset{\cong}\longto \Hom_{\sA}(H^{*,*}(X), H^{*,*})
\end{equation}
from the $\sA_{*,*}$-comodule homomorphisms $H_{*,*} \to H_{*,*}(X)$ to
the $\sA$-module homomorphisms $H^{*,*}(X) \to H^{*,*}$.

\section{A delayed limit Adams spectral sequence}
\label{sec:delayed-Adams}

Let
\begin{equation} \label{eq:Y-f-tower}
\dots \to Y(m+1) \overset{f_{m+1}}\longto Y(m)
	\overset{f_m}\longto Y(m-1) \to \dots
\end{equation}
be any tower of motivic spectra.  Its homotopy limit
$Y = \holim_m Y(m)$ sits in a homotopy cofiber sequence
$$
\Sigma^{-1} \prod_m Y(m)
	\overset{i}\longto Y
	\overset{j}\longto \prod_m Y(m)
	\overset{k}\longto \prod_m Y(m) \,,
$$
cf.~\cite{Mil62}*{Lem.~2}, where $k$ is the difference between the
identity map and the product of the maps~$f_m$.  Let
$$
\xymatrix{
\dots \ar[r] & \bS_2 \ar[rr]^-{\alpha}
	&& \bS_1 \ar[rr]^-{\alpha} \ar[dl]^-{\beta}
	&& \bS_0 \ar[dl]^-{\beta} \\
&& H \wedge \bS_1 \ar@{-->}[ul]^-{\gamma}
	&& H \wedge \bS_0 \ar@{-->}[ul]^-{\gamma}
}
$$
be the canonical mod~$\ell$ Adams resolution of the motivic sphere
spectrum~$\bS = \bS_0$, inductively defined by the homotopy cofiber
sequences
$$
\bS_{s+1} \overset{\alpha}\longto \bS_s
	\overset{\beta}\longto H \wedge \bS_s
	\overset{\gamma}\longto \Sigma \bS_{s+1}
$$
where $\beta = \eta \wedge 1$, cf.~\cite{Ada74a}*{p.~318}.  (Dashed arrows
indicate morphisms of degree~$-1$.)  Form the smash products
\begin{align*}
Y_s(m) &= \bS_s \wedge Y(m) \\
K_s(m) &= H \wedge \bS_s \wedge Y(m)
\end{align*}
so as to obtain a tower of canonical Adams resolutions
\begin{equation} \label{eq:tower-Adams-res}
\xymatrix@-1.5pc{
& \vdots \ar[d] && \vdots \ar[d] && \vdots \ar[d] \\
\dots \ar[r] & Y_2(m{+}1) \ar[rr] \ar[dd] && Y_2(m) \ar[rr] \ar[dd] && Y_2(m{-}1) \ar[dd] \ar[r] & \dots \\
& \dots \ar[r] & K_1(m{+}1) \ar[rr]|-\hole \ar@{-->}[ul] && K_1(m) \ar[rr]|-\hole \ar@{-->}[ul] && K_1(m{-}1) \ar[r] \ar@{-->}[ul] & \dots \\
\dots \ar[r] & Y_1(m{+}1) \ar[rr] \ar[ur] \ar[dd] && Y_1(m) \ar[rr] \ar[ur] \ar[dd] && Y_1(m{-}1) \ar[ur] \ar[dd] \ar[r] & \dots \\
& \dots \ar[r] & K_0(m{+}1) \ar[rr]|-\hole \ar@{-->}[ul] && K_0(m) \ar[rr]|-\hole \ar@{-->}[ul] && K_0(m{-}1) \ar[r] \ar@{-->}[ul] & \dots \\
\dots \ar[r] & Y_0(m{+}1) \ar[rr] \ar[ur] && Y_0(m) \ar[rr] \ar[ur] && Y_0(m{-}1) \ar[ur] \ar[r] & \dots \,.
}
\end{equation}
Let
\begin{align*}
Y_s &= \holim_m Y_s(m) \\
K_s &= \holim_m K_s(m)
\end{align*}
be the homotopy limits of the terms in these Adams resolutions.
These fit in a commutative diagram
\begin{equation} \label{eq:ijalphabeta-bitower}
\xymatrix@-1pc{
\vdots \ar[d] && \vdots \ar[d] && \vdots \ar[d] \\
\Sigma^{-1} \prod_m Y_2(m) \ar[rr]^-{i} \ar[dd]_-{\alpha} && Y_2 \ar[rr]^-{j} \ar[dd]_(0.3){\alpha} && \prod_m Y_2(m) \ar[dd]_(0.3){\alpha} \\
 & \Sigma^{-1} \prod_m K_1(m) \ar[rr]^(0.5){i}|-\hole \ar@{-->}[ul]_-{\gamma} && K_1 \ar[rr]^(0.25){j}|-\hole \ar@{-->}[ul]_-{\gamma} && \prod_m K_1(m) \ar@{-->}[ul]_-{\gamma} \\
\Sigma^{-1} \prod_m Y_1(m) \ar[ur]_-{\beta} \ar[rr]^-{i} \ar[dd]_-{\alpha} && Y_1 \ar[ur]_-{\beta} \ar[rr]^-{j} \ar[dd]_(0.3){\alpha} && \prod_m Y_1(m) \ar[ur]_-{\beta} \ar[dd]_(0.3){\alpha} \\
 & \Sigma^{-1} \prod_m K_0(m) \ar[rr]^(0.5){i}|-\hole \ar@{-->}[ul]_-{\gamma} && K_0 \ar[rr]^(0.25){j}|-\hole \ar@{-->}[ul]_-{\gamma} && \prod_m K_0(m) \ar@{-->}[ul]_-{\gamma} \\
\Sigma^{-1} \prod_m Y_0(m) \ar[ur]_-{\beta} \ar[rr]^-{i} && Y_0 \ar[ur]_-{\beta} \ar[rr]^-{j} && \prod_m Y_0(m) \ar[ur]_-{\beta}
}
\end{equation}
with horizontal homotopy cofiber sequences extending to
\begin{align*}
\Sigma^{-1} \prod_m Y_s(m) &\overset{i}\longto Y_s
	\overset{j}\longto \prod_m Y_s(m)
	\overset{k}\longto \prod_m Y_s(m) \\
\Sigma^{-1} \prod_m K_s(m) &\overset{i}\longto K_s
	\overset{j}\longto \prod_m K_s(m)
	\overset{k}\longto \prod_m K_s(m) \,.
\end{align*}
The subdiagram
\begin{equation} \label{eq:limitAdams}
\xymatrix{
\dots \ar[r] & Y_2 \ar[rr]^-{\alpha}
        && Y_1 \ar[rr]^-{\alpha} \ar[dl]^-{\beta}
        && Y_0 \ar[dl]^-{\beta} \\
&& K_1 \ar@{-->}[ul]^-{\gamma} && K_0 \ar@{-->}[ul]^-{\gamma}
}
\end{equation}
is not generally an Adams resolution.  Nonetheless, one may consider its
associated homotopy spectral sequence, with $E_1$-term
$$
{}^{\lim} E_1^{s,*,*} = \pi_{*,*}(K_s)
$$
and abutment $\pi_{*,*}(Y_0) = \pi_{*,*}(Y)$.  Under
finiteness hypotheses which ensure that all limits in sight are exact,
the $E_2$-term of this \emph{limit Adams spectral sequence} was described
in~\cite{CMP87}*{Prop.~7.1} and~\cite{LNR12}*{Prop.~2.2}.  In our motivic
context, these finiteness hypotheses are only realistic if $H^{*,*}$
is finite in each bidegree, which excludes some very interesting base
schemes~$S$, such as $\Spec$ of a global field.
(For example, Euclid knew that $H^{1,1}(\Spec \bQ) = \bQ^\times /
(\bQ^\times)^\ell \cong \bZ/(2,\ell) \oplus \bigoplus_{\text{$p$ prime}}
\bZ/\ell$ is infinitely generated.)

To avoid this restrictive hypothesis, we shall instead show
that there is a modified Adams spectral sequence, in the style
of~\cite{Mil72}*{Lem.~5.3.1} and~\cite{Bru82}, with the same abutment
as before, whose $E_2$-term is recognizable under more flexible
finiteness conditions.  This kind of modification is referred to
in~\cite{BR21}*{\S12.6} as a delayed Adams spectral sequence, to
distinguish it from another kind (the hastened one) of modified Adams
spectral sequence in current usage, cf.~\cite{BHHM08}*{\S3}.

To construct the delayed limit Adams spectral sequence we may assume
that the maps $i$ and $\alpha$ in~\eqref{eq:ijalphabeta-bitower} are
all cofibrations, let $W_0 = Y_0$, and form the pushouts
$$
W_s  = Y_s \, \cup \, \Sigma^{-1} \prod_m Y_{s-1}(m)
$$
along $\Sigma^{-1} \prod_m Y_s(m)$, for all $s\ge1$.  There are
then homotopy cofiber sequences
$$
W_1 \longto W_0 \overset{\beta j}\longto L_0 = \prod_m K_0(m)
	\longto \Sigma W_1
$$
and
$$
W_{s+1} \longto W_s
	\overset{\beta j \cup \beta}\longto
	L_s = \prod_m K_s(m) \, \vee \, \Sigma^{-1} \prod_m K_{s-1}(m)
	\longto \Sigma W_{s+1}
$$
for all $s\ge1$, defining the spectra~$L_s$.  This produces a delayed
resolution
\begin{equation} \label{eq:delayedAdamsres}
\xymatrix{
\dots \ar[r] & W_2 \ar[rr]^-{\alpha}
        && W_1 \ar[rr]^-{\alpha} \ar[dl]^-{\beta}
        && W_0 \ar[dl]^-{\beta} \\
&& L_1 \ar@{-->}[ul]^-{\gamma} && L_0 \ar@{-->}[ul]^-{\gamma}
}
\end{equation}
of $W_0 = Y_0 \simeq Y$.  The inclusions $Y_s \subset W_s$ induce a map
of diagrams from~\eqref{eq:limitAdams} to~\eqref{eq:delayedAdamsres}.

\begin{definition} \label{def:del-lim-adams-sp-seq}
The \emph{delayed limit Adams spectral sequence} of the
tower~\eqref{eq:Y-f-tower} is the homotopy spectral sequence associated
to the resolution~\eqref{eq:delayedAdamsres}, with $E_1$-term
$$
{}^{\del} E_1^{s,*,*} = \pi_{*,*}(L_s)
$$
and $d_1$-differential induced by the composite
$$
\beta \gamma \: L_s \longto \Sigma W_{s+1} \longto \Sigma L_{s+1} \,.
$$
\end{definition}

We now make the assumption that each $H \wedge Y(m)$ is a motivic GEM of
bifinite type.  For example, this is the case if each $Y(m)$ is cellular
with $H_{*,*}(Y(m))$ free of bifinite type as a left $H_{*,*}$-module.
It follows by induction on~$s$ that each $K_s(m)$ is a motivic GEM
of bifinite type.  Hence the isomorphisms~\eqref{eq:pi-comodHom}
and~\eqref{eq:comodHom-modHom} identify the $(E_1, d_1)$-term
$$
\dots \longfrom \pi_{*,*} K_2(m)
	\longfrom \pi_{*,*} K_1(m)
	\longfrom \pi_{*,*} K_0(m)
	\leftarrow 0
$$
of the Adams spectral sequence for $Y(m)$ with $\Hom_{\sA}(-, H^{*,*})$
applied to the free $\sA$-module resolution

$$
\dots \longto H^{*,*}(K_2(m))
	\overset{\partial}\longto H^{*,*}(K_1(m))
	\overset{\partial}\longto H^{*,*}(K_0(m)) \to 0
$$
of $H^{*,*}(Y(m))$.  In view of~\eqref{eq:tower-Adams-res}, these
resolutions are compatible for varying~$m$.  Passing to colimits over~$m$,
we obtain a flat $\sA$-module resolution
$$
\dots \longto H_c^{*,*}(K_2)
	\overset{\partial}\longto H_c^{*,*}(K_1)
	\overset{\partial}\longto H_c^{*,*}(K_0) \to 0
$$
of $H_c^{*,*}(Y)$, where we write
\begin{align*}
H_c^{*,*}(Y) &= \colim_m H^{*,*}(Y(m)) \\
H_c^{*,*}(K_s) &= \colim_m H^{*,*}(K_s(m))
\end{align*}
for the ``continuous'' cohomology groups of the towers $\{Y(m)\}_m$
and $\{K_s(m)\}_m$, respectively.  The $\sA$-module $H_c^{*,*}(K_s)$
might not be free, but remains flat, since such modules are preserved
under filtered colimits.  These colimits can also be written as cokernels,
as in the following diagram with exact rows and columns.
$$
\xymatrix{
& \vdots \ar[d] & \vdots \ar[d] & \vdots \ar[d] \\
0 \ar[r] & \bigoplus_m H^{*,*}(K_1(m)) \ar[r]^-{k^*} \ar[d]^-{\partial}
	& \bigoplus_m H^{*,*}(K_1(m)) \ar[r]^-{j^*} \ar[d]^-{\partial}
	& H_c^{*,*}(K_1) \ar[r] \ar[d]^-{\partial} & 0 \\
0 \ar[r] & \bigoplus_m H^{*,*}(K_0(m)) \ar[r]^-{k^*} \ar[d]^-{\epsilon}
	& \bigoplus_m H^{*,*}(K_0(m)) \ar[r]^-{j^*} \ar[d]^-{\epsilon}
	& H_c^{*,*}(K_0) \ar[r] \ar[d]^-{\epsilon} & 0 \\
0 \ar[r] & \bigoplus_m H^{*,*}(Y(m)) \ar[r]^-{k^*} \ar[d]
	& \bigoplus_m H^{*,*}(Y(m)) \ar[r]^-{j^*} \ar[d]
	& H_c^{*,*}(Y) \ar[r] \ar[d] & 0 \\
& 0 & 0 & 0
}
$$
Omitting the bottom row and the right hand column, we have a
bicomplex of free $\sA$-modules, whose total complex $(F_*, \partial)$
is a free resolution of $H_c^{*,*}(Y)$.  Here
$$
F_0 = \bigoplus_m H^{*,*}(K_0(m))
$$
and
$$
F_s = \bigoplus_m H^{*,*}(K_s(m))
	\, \oplus \, \bigoplus_m H^{*,*}(\Sigma^{-1} K_{s-1}(m))
$$
for $s\ge1$.
Hence we can recognize
\begin{align*}
{}^{\del} E_1^{0,*,*} &=\pi_{*,*} L_0 = \prod_m \pi_{*,*} K_0(m)
	\cong \prod_m \Hom_{\sA}(H^{*,*}(K_0(m)), H^{*,*}) \\
&\cong \Hom_{\sA}(\bigoplus_m H^{*,*}(K_0(m)), H^{*,*})
	= \Hom_{\sA}(F_0, H^{*,*})
\end{align*}
and
\begin{align*}
{}^{\del} E_1^{s,*,*} &= \pi_{*,*} L_s \cong \prod_m \pi_{*,*} K_s(m)
	\oplus \prod_m \pi_{*,*} \Sigma^{-1} K_{s-1}(m) \\
&\cong \prod_m \Hom_{\sA}(H^{*,*}(K_s(m)), H^{*,*})
	\oplus \prod_m \Hom_{\sA}(H^{*,*}(\Sigma^{-1} K_{s-1}(m)), H^{*,*}) \\
&\cong \Hom_{\sA}(\bigoplus_m H^{*,*}(K_s(m))
	\, \oplus \, \bigoplus_m H^{*,*}(\Sigma^{-1} K_{s-1}(m)), H^{*,*}) \\
&= \Hom_{\sA}(F_s, H^{*,*})
\end{align*}
for $s\ge1$.  Moreover, $d_1 \: {}^{\del} E_1^{s,*,*} \to {}^{\del}
E_1^{s+1,*,*}$ is induced by the boundary operator $\partial \: F_{s+1}
\to F_s$ in the total complex, as can be verified by tracing through the
definitions.  Since $(F_*, \partial)$ is a free $\sA$-module resolution
of~$H_c^{*,*}(Y)$, we obtain the desired isomorphism
$$
{}^{\del} E_2^{s,*,*} \cong \Ext_{\sA}^{s,*,*}(H_c^{*,*}(Y), H^{*,*})
$$
for each $s\ge0$.

\begin{proposition} \label{prop:dellimE2}
Let $\dots \to Y(m+1) \to Y(m) \to \dots$ be a tower of
motivic spectra, with each $H \wedge Y(m)$ a motivic GEM
of bifinite type.  Let $Y = \holim_m Y(m)$ and
$H_c^{*,*}(Y) = \colim_m H^{*,*}(Y(m))$.
The delayed limit Adams spectral sequence
$$
{}^{\del} E_1^{s,*,*} \Longrightarrow \pi_{*,*}(Y)
$$
has $E_2$-term
$$
{}^{\del} E_2^{s,*,*} = \Ext_{\sA}^{s,*,*}(H_c^{*,*}(Y), H^{*,*}) \,,
$$
with $\Ext$ calculated in the category of $\sA$-modules.
\end{proposition}

\begin{proof}
This summarizes the discussion so far in this section.
\end{proof}

Let $Y_\infty(m) = \holim_s Y_s(m)$.  The Adams spectral sequence
for~$Y(m)$ is conditionally convergent~\cite{Boa99}*{Def.~5.10}
to $\pi_{*,*} Y(m)$ if and only if $\pi_{*,*} Y_\infty(m) = 0$.
This fails in many interesting examples, such as when $\pi_{*,*} Y(m)$
contains torsion of order prime to~$\ell$, but often becomes true after
$(\ell,\eta)$-adic completion in the following sense.

\begin{definition} \label{def:ell-eta-completion}
Let $\ell \in \pi_{0,0}(\bS)$ be $\ell$ times the class of the identity
map $\bS \to \bS$, and let $\eta \in \pi_{1,1}(\bS)$ be the class of the
Hopf fibration $S^{3,2} \simeq \bA^2 - \{0\} \to \bP^1 \simeq S^{2,1}$.
For any motivic spectrum~$X$ let
\begin{align*}
X^\wedge_\ell &= \holim_n X/\ell^n \\
X^\wedge_{\ell,\eta} &= \holim_n \, (X^\wedge_\ell)/\eta^n
	\simeq \holim_n X/(\ell^n, \eta^n)
\end{align*}
be the \emph{$\ell$- and $(\ell,\eta)$-adic completions} of $X$,
respectively, as in~\cite{HKO11}*{p.~574} and~\cite{Man21}*{Def.~3.2.9}.
There are canonical completion maps $X \to X^\wedge_\ell \to
X^\wedge_{\ell,\eta}$.
\end{definition}

We note that $X^\wedge_\ell \simeq X^\wedge_{\ell,\eta}$ when
$\eta$ acts nilpotently, e.g., for $\ell$ odd.

\begin{lemma} \label{lem:HMod-elletacomplete}
For each $H$-module spectrum~$M$ the completion map $M \to
M^\wedge_{\ell,\eta}$ is a $\pi_{*,*}$-isomorphism.
\end{lemma}

\begin{proof}
Multiplication by $\ell$ and by $\eta$ act trivially on $H_{*,*}$,
hence also on $\pi_{*,*} M$.  The tower of short exact sequences
$$
\xymatrix@R-0.5pc{
& \vdots \ar[d]_{\id} & \vdots \ar[d] & \vdots \ar[d]_{0} \\
0 \ar[r] & \pi_{*,*}(M) \ar[r] \ar[d]_-{\id}
	& \pi_{*,*}(M/\ell^{n+1}) \ar[r] \ar[d]
	& \pi_{*,*}(\Sigma M) \ar[r] \ar[d]_-{0} & 0 \\
0 \ar[r] & \pi_{*,*}(M) \ar[r] \ar[d]_-{\id}
	& \pi_{*,*}(M/\ell^n) \ar[r] \ar[d]
	& \pi_{*,*}(\Sigma M) \ar[r] \ar[d]_-{0} & 0 \\
& \vdots & \vdots & \vdots
}
$$
and the $\lim$-$\limone$ sequence imply that $M \to M^\wedge_\ell$ is
a $\pi_{*,*}$-isomorphism.  Likewise, the tower of short exact sequences
$$
\xymatrix@R-0.5pc{
& \vdots \ar[d]_{\id} & \vdots \ar[d] & \vdots \ar[d]_{0} \\
0 \ar[r] & \pi_{*,*}(M^\wedge_\ell) \ar[r] \ar[d]_-{\id}
	& \pi_{*,*}(M^\wedge_\ell/\eta^{n+1}) \ar[r] \ar[d]
	& \pi_{*,*}(\Sigma^{n+2,n+1} M^\wedge_\ell) \ar[r] \ar[d]_-{0} & 0 \\
0 \ar[r] & \pi_{*,*}(M^\wedge_\ell) \ar[r] \ar[d]_-{\id}
	& \pi_{*,*}(M^\wedge_\ell/\eta^n) \ar[r] \ar[d]
	& \pi_{*,*}(\Sigma^{n+1,n} M^\wedge_\ell) \ar[r] \ar[d]_-{0} & 0 \\
& \vdots & \vdots & \vdots
}
$$
and the $\lim$-$\limone$ sequence imply that $M^\wedge_\ell \to
M^\wedge_{\ell,\eta}$ is a $\pi_{*,*}$-isomorphism.
\end{proof}

Applying $(\ell,\eta)$-adic completion to the tower
of Adams resolutions~\eqref{eq:tower-Adams-res} yields
another diagram of the same shape.
At its lower edge, the tower of spectra
$$
\dots \to Y(m+1)^\wedge_{\ell,\eta}
	\longto Y(m)^\wedge_{\ell,\eta}
	\longto Y(m-1)^\wedge_{\ell,\eta} \to \dots
$$
has homotopy limit~$Y^\wedge_{\ell,\eta}$.  For each~$m$ the Adams
resolution of~$Y(m)$ maps to the diagram
$$
\xymatrix{
\dots \ar[r] & Y_2(m)^\wedge_{\ell,\eta} \ar[rr]^-{\alpha}
        && Y_1(m)^\wedge_{\ell,\eta} \ar[rr]^-{\alpha} \ar[dl]^-{\beta}
        && Y_0(m)^\wedge_{\ell,\eta} \ar[dl]^-{\beta} \\
&& K_1(m)^\wedge_{\ell,\eta} \ar@{-->}[ul]^-{\gamma}
	&& K_0(m)^\wedge_{\ell,\eta} \ar@{-->}[ul]^-{\gamma}
}
$$
of homotopy cofiber sequences, and by Lemma~\ref{lem:HMod-elletacomplete}
the induced map of homotopy spectral sequences is an isomorphism from
the $E_1$-term and onward.  The new spectral sequence has abutment
$\pi_{*,*}(Y(m)^\wedge_{\ell,\eta})$, and is conditionally convergent
to this target if and only if $\pi_{*,*}(Y_\infty(m)^\wedge_{\ell,\eta})
= 0$.

We now make the additional assumption, for each~$m$, that the
Adams spectral sequence for~$Y(m)$ converges conditionally to
$\pi_{*,*}(Y(m)^\wedge_{\ell,\eta})$.  The following theorem was
proved by Hu--Kriz--Ormsby~\cite{HKO11}*{Thm.~1} in the case of a
cellular spectrum~$X$ of finite type over~$\Spec k$ for~$k$ a field
of characteristic~$0$.  It was generalized to bounded below spectra~$X$
over~$S$, in our generality, by Mantovani~\cite{Man21}.

The homotopy $t$-structure on~$SH(S)$ is defined as
in~\cite{Hoy15}*{\S2.1}, and a motivic spectrum is bounded below
if it lies in $SH(S)_{\ge-m}$ for some finite~$m$.  (When $S =
\Spec k$ for a field~$k$, Morel's stable $\bA^1$-connectivity
theorem~\cite{Mor05}*{Thm.~3} shows that $X$ lies in $SH(S)_{\ge-m}$
if and only if the homotopy sheaves $\underline{\pi}_{t,u}(X)$ vanish
whenever $t-u < -m$.)

\begin{theorem} \label{thm:mantovani}
Suppose that~$X$ in $SH(S)$ is bounded below in the homotopy
$t$-structure.  Then the mod~$\ell$ Adams spectral sequence for~$X$
is conditionally convergent to $\pi_{*,*}(X^\wedge_{\ell,\eta})$.
\end{theorem}

\begin{proof}
As reviewed in~\cite{Man21}*{\S5}, this is an application
of~\cite{Man21}*{Thm.~1.0.2, Thm.~1.0.4} in the case $E = H$, which
satisfies Mantovani's hypotheses because of~\cite{Hoy15}*{Thm.~3.8,
Thm.~7.12} and~\cite{Spi18}*{Thm.~10.3}.
\end{proof}

\begin{proposition} \label{prop:del-lim-cond-conv}
Let $Y = \holim_m Y(m)$ be the homotopy limit of a tower of
motivic spectra.  Suppose, for each~$m$, that the mod~$\ell$
Adams spectral sequence for $Y(m)$ converges conditionally to
$\pi_{*,*}(Y(m)^\wedge_{\ell,\eta})$.  Then the limit and delayed limit
Adams spectral sequences
\begin{align*}
{}^{\lim} E_1^{s,*,*} = \pi_{*,*}(K_s)
	&\Longrightarrow \pi_{*,*}(Y^\wedge_{\ell,\eta}) \\
{}^{\del} E_1^{s,*,*} = \pi_{*,*}(L_s)
	&\Longrightarrow \pi_{*,*}(Y^\wedge_{\ell,\eta})
\end{align*}
are both conditionally convergent to the bigraded homotopy groups of
the $(\ell,\eta)$-adic completion of~$Y$.
\end{proposition}

\begin{proof}
Let $Y_\infty = \holim_s Y_s$ and $W_\infty = \holim_s W_s$.
The inclusions
$$
W_{s+1} \subset Y_s \subset W_s \subset Y_{s-1}
$$
imply that $Y_\infty \simeq W_\infty$.  Granting that
$\pi_{*,*}(Y_\infty(m)^\wedge_{\ell,\eta}) = 0$ for
all~$m$, the short exact $\lim$-$\limone$ sequence
shows that $\pi_{*,*}((Y_\infty)^\wedge_{\ell,\eta}) \cong
\pi_{*,*}((W_\infty)^\wedge_{\ell,\eta}) = 0$, so that both the limit
Adams spectral sequence and the delayed limit Adams spectral sequence
are conditionally convergent to $\pi_{*,*}(Y^\wedge_{\ell,\eta})$.
\end{proof}

Let $g \: X \to Y = \holim_m Y(m)$ be a map of motivic spectra.
The resulting compatible maps from the canonical Adams resolution of~$X$
to the canonical Adams resolutions of the~$Y(m)$ induce a
map from the former to the diagram~\eqref{eq:limitAdams}, which
can be naturally continued to map to the delayed limit Adams
resolution~\eqref{eq:delayedAdamsres}.  Applying $(\ell,\eta)$-adic
completion, and passing to the associated homotopy spectral
sequences, we obtain morphisms of spectral sequences
$$
g \: E_1^{*,*,*}(X) \longto {}^{\lim} E_1^{*,*,*}(Y)
	\longto {}^{\del} E_1^{*,*,*}(Y)
$$
with abutment
$$
g \: \pi_{*,*}(X^\wedge_{\ell,\eta}) \longto \pi_{*,*}(Y^\wedge_{\ell,\eta})
	\overset{=}\longto \pi_{*,*}(Y^\wedge_{\ell,\eta}) \,.
$$
We can now appeal to a special case of Boardman's comparison
theorem~\cite{Boa99}*{Thm.~7.2} for conditionally convergent spectral
sequences.  This version of the comparison theorem is particularly
convenient, in view of the failure of strong convergence for the motivic
Adams spectral sequence for the sphere spectrum over a number field,
demonstrated by Kylling--Wilson in~\cite{KW19}*{Cor.~7.8}.

\begin{proposition} \label{prop:comparison}
Let $g \: X \to Y = \holim_m Y(m)$, with $H \wedge X$ and each
$H \wedge Y(m)$ a motivic GEM of bifinite type.  Suppose that
the mod~$\ell$ Adams spectral sequences for $X$ and the $Y(m)$
are conditionally convergent to $\pi_{*,*}(X^\wedge_{\ell,\eta})$
and $\pi_{*,*}(Y(m)^\wedge_{\ell,\eta})$, respectively.  If 
the $\sA$-module homomorphism
$$
g^* \: H_c^{*,*}(Y) \longto H^{*,*}(X)
$$
is an $\Ext$-isomorphism, so that
\begin{multline*}
g \: E_2^{s,*,*}(X) = \Ext_{\sA}^{s,*,*}(H^{*,*}(X), H^{*,*}) \\
	\overset{\cong}\longto
	{}^{\del} E_2^{s,*,*}(Y) = \Ext_{\sA}^{s,*,*}(H_c^{*,*}(Y), H^{*,*})
\end{multline*}
is an isomorphism, then $g$ induces an isomorphism
$$
g \: \pi_{*,*}(X^\wedge_{\ell,\eta})
	\overset{\cong}\longto \pi_{*,*}(Y^\wedge_{\ell,\eta}) \,.
$$
\end{proposition}

\begin{proof}
The identification of the ($E_1$- and) $E_2$-term for $X$ follows as
usual from~\eqref{eq:pi-comodHom} and~\eqref{eq:comodHom-modHom},
and the delayed limit $E_2$-term for~$Y$ is given by
Proposition~\ref{prop:dellimE2}.  We now apply \cite{Boa99}*{Thm.~7.2(i)}.
If~$g$ induces an isomorphism of $E_2$-terms, then it certainly also
induces isomorphisms of $E_\infty$- and $RE_\infty$-terms.  Hence $g$
induces the stated isomorphism of (filtered) abutments.
\end{proof}

\section{A tower of Thom spectra}
\label{sec:tower-Thom}

Our next aim is to construct a diagram
$$
\xymatrix{
& & L_{-\infty}^\infty \ar[d] \\
Th(-(m+1)\gamma_n^*) \ar[r]^-{i} \ar[d]_-{j}
	& Th(-(m+1)\gamma_{n+1}^*) \ar[r] \ar[d]_-{j}
	& L_{-2m-2}^\infty \ar[d]_-{j} \\
Th(-m\gamma_n^*) \ar[r]^-{i}
	& Th(-m\gamma_{n+1}^*) \ar[r]
	& L_{-2m}^\infty
}
$$
of motivic Thom spectra, with $H_c^{*,*}(L_{-\infty}^\infty)$ realizing
$H^{*,*}(B\mu_\ell)_{\loc}$.  The left hand square will commute up to
a generalized sign, and replacing $i$ by $i^4$ will give a strictly
commuting diagram in the stable homotopy category~$SH(S)$.

For an algebraic vector bundle $\alpha \downarrow X$ over a smooth scheme~$X$
(over our base scheme~$S$) we let $E(\alpha)$ denote its total space,
let $E_0(\alpha) = E(\alpha) \setminus z(X)$ be the complement of its zero
section, and let the Thom space $Th(\alpha) = E(\alpha) / E_0(\alpha)$ be the
motivic quotient space, formed as in~\cite{Voe03a}*{\S4}.

For $n\ge1$ let $\mu_\ell$ act diagonally on $\bA^n$, let $\bA^n_0 =
\bA^n \setminus \{0\}$, and let $L^{2n-1} = \bA^n_0/\mu_\ell$ denote
the $n$-th algebraic lens space, which is smooth and quasi-projective.
Write $[x_1, \dots, x_n]$ in $L^{2n-1}$ for the image of $(x_1,
\dots, x_n)$ in $\bA^n_0$.  The inclusion $\mu_\ell \subset \bG_m$
induces a projection $L^{2n-1} \to \bP^{n-1}$.  Let $\gamma_n =
\gamma^1_n \downarrow L^{2n-1}$ be the pullback of the tautological
line bundle~$\sO(-1)$ over~$\bP^{n-1}$, with total space given by the
balanced product
$$
E(\gamma_n) = \bA^n_0 \times_{\mu_\ell} \bA^1
$$
and bundle projection mapping $[x_1, \dots, x_n; y]$
to $[x_1, \dots, x_n]$.
Let $\epsilon_n = \epsilon^n_n \downarrow L^{2n-1}$
be the trivial rank~$n$ bundle over~$L^{2n-1}$, with total space
$E(\epsilon_n) = L^{2n-1} \times \bA^n$
and bundle projection to the first factor.  There is a canonical
embedding $\gamma_n \to \epsilon_n$, given in coordinates by
$$
[x_1, \dots, x_n; y]
	\longmapsto ([x_1, \dots, x_n], x_1 y, \dots, x_n y) \,.
$$
Let $\zeta_n = \zeta^{n-1}_n \downarrow L^{2n-1}$ be its cokernel, so that
there is a short exact sequence
$$
0 \to \gamma_n \longto \epsilon_n \longto \zeta_n \to 0
$$
of algebraic vector bundles over~$L^{2n-1}$.  Let $\gamma_n^*$,
$\epsilon_n^*$ and $\zeta_n^*$ denote the dual bundles, fitting in
a short exact sequence
\begin{equation} \label{eq:SES-of-dual-bundles}
0 \to \zeta_n^* \overset{j}\longto \epsilon_n^*
	\longto \gamma_n^* \to 0 \,.
\end{equation}
Here the total space of $\gamma_n^*$ is given by the orbit space
$E(\gamma_n^*) = (\bA^n_0 \times \bA^1)/\mu_\ell$, where $\mu_\ell$
acts diagonally.

More generally, the total space of the $k$-fold direct
sum $k \gamma_n^* = \gamma_n^* \oplus \dots \oplus \gamma_n^*$,
where $k\ge0$, is
$$
E(k \gamma_n^*) = (\bA^n_0 \times \bA^k)/\mu_\ell \,,
$$
which comes with a canonical map to $\bA^{n+k}_0/\mu_\ell = L^{2n+2k-1}$.
By~\cite{MV99}*{Lem.~3.1.6}, the inclusion
$$
Th(k\gamma_n^*) = \frac{E(k\gamma_n^*)}{E_0(k\gamma_n^*)} =
	\frac{(\bA^n_0 \times \bA^k)/\mu_\ell}
		{(\bA^n_0 \times \bA^k_0)/\mu_\ell}
	\longto \frac{(\bA^n \times \bA^k)_0/\mu_\ell}
		{(\bA^n \times \bA^k_0)/\mu_\ell}
$$
is an equivalence of Nisnevich sheaves, and by~\cite{MV99}*{Ex.~3.2.2}
the projection $(\bA^n \times \bA^k_0)/\mu_\ell \to \bA^k_0/\mu_\ell
= L^{2k-1}$ is an $\bA^1$-homotopy equivalence, so there is a homotopy
cofiber sequence
\begin{equation} \label{eq:stuntedlens}
L^{2k-1} \overset{i^n}\longto L^{2n+2k-1} \longto Th(k\gamma_n^*)
\end{equation}
of motivic spaces.  Following
James~\cite{Jam59}*{p.~117}, Atiyah~\cite{Ati61}*{Prop.~4.3}
and~Kambe--Matsunaga--Toda~\cite{KMT66}*{Thm.~1} we may therefore write
$Th(k\gamma_n^*) = L_{2k}^{2n+2k-1}$ for the Thom space of $k\gamma_n^*$
over $L^{2n-1}$, and refer to it as a motivic stunted lens space.

Following Mahowald and Adams~\cite{Ada74b}*{p.~4}, we are, however, more
interested in the cases where $k = -m$ is negative, corresponding to
Thom spectra $Th(-m \gamma_n^*) = L_{-2m}^{2n-2m-1}$ of virtual bundles
$-m\gamma_n^*$.  In view of~\eqref{eq:SES-of-dual-bundles}, $-m\gamma_n^*
\equiv m\zeta_n^* - m\epsilon_n^*$ as virtual bundles over~$L^{2n-1}$,
where $m \epsilon_n^*$ is trivial of rank~$mn$, which leads to the
following definition.

\begin{definition}
For $m\ge0$ let
$$
L_{-2m}^{2n-2m-1} = Th(-m\gamma_n^*) = \Sigma^{-2mn,-mn} Th(m\zeta_n^*)
$$
denote a (finite, motivic) stunted lens spectrum.
\end{definition} 

Consider the inclusion $i \: L^{2n-1} \to L^{2n+1}$.  There are natural
isomorphisms
$i^*(\gamma_{n+1}) \cong \gamma_n$,
$i^*(\epsilon_{n+1}) \cong \epsilon_n \oplus \epsilon^1_n$ and
$i^*(\zeta_{n+1}) \cong \zeta_n \oplus \epsilon^1_n$,
where $\epsilon^1_n$ is trivial of rank~$1$.
Dually, $i^*(\gamma_{n+1}^*) \cong \gamma_n^*$,
$i^*(\epsilon^*_{n+1}) \cong \epsilon^*_n \oplus (\epsilon^1_n)^*$ and
$i^*(\zeta^*_{n+1}) \cong \zeta^*_n \oplus (\epsilon^1_n)^*$.

Consider also the inclusion $j \: \zeta_n^* \to \epsilon_n^*$
of bundles over~$L^{2n-1}$.

\begin{definition} \label{def:i-j-maps-for-lens-spectra}
Let
$$
i \: L_{-2m}^{2n-2m-1} = Th(-m\gamma_n^*)
	\longto Th(-m\gamma_{n+1}^*) = L_{-2m}^{2n-2m+1}
$$
and
$$
j \: L_{-2m-2}^{2n-2m-3} = Th(-(m+1)\gamma_n^*)
	\longto Th(-m\gamma_n^*) = L_{-2m}^{2n-2n-1}
$$
be the maps obtained by applying
$\Sigma^{-2m(n+1),-m(n+1)}$
and
$\Sigma^{-2(m+1)n,-(m+1)n}$ to
$$
\Sigma^{2m,m} Th(m\zeta_n^*)
	\overset{\sh}\cong Th(m(\zeta_n^* \oplus (\epsilon^1_n)^*))
	\overset{m i}\longto Th(m\zeta_{n+1}^*)
$$
and
$$
Th((m+1)\zeta_n^*) \overset{m \id \oplus j}\longto
	Th(m\zeta_n^* \oplus \epsilon_n^*) = \Sigma^{2n,n} Th(m\zeta_n^*)
\,,
$$
respectively.  Here $\sh$ denotes the isomorphism of spectra induced
by the shuffle $m\zeta_n^* \oplus m(\epsilon^1_n)^* \cong
m(\zeta_n^* \oplus (\epsilon^1_n)^*)$.
\end{definition}

\begin{definition}
Let $-\epsilon \in \pi_{0,0}(\bS)$ be the class of the symmetry
isomorphism $\gamma \: S^{2,1} \wedge S^{2,1} \cong S^{2,1} \wedge
S^{2,1}$.  It satisfies $(-\epsilon)^2 = 1$, since $\gamma^2 = \id$.
\end{definition}

\begin{lemma} \label{lem:i-j-rectangle}
The rectangle
$$
\xymatrix{
L_{-2m-2}^{2n-2m-3} = Th(-(m+1)\gamma_n^*) \ar[r]^-{i} \ar[d]_-{j}
	& Th(-(m+1)\gamma_{n+1}^*) = L_{-2m-2}^{2n-2m-1} \ar[dd]^-{j} \\
L_{-2m}^{2n-2m-1} = Th(-m\gamma_n^*) \ar[d]_-{(-\epsilon)^{mn}}^-{\simeq} \\
L_{-2m}^{2n-2m-1} = Th(-m\gamma_n^*) \ar[r]^-{i}
	& Th(-m\gamma_{n+1}^*) = L_{-2m}^{2n-2m+1}
}
$$
commutes up to homotopy.
\end{lemma}

\begin{proof}
The diagrams
$$
\xymatrix{
Th((m+1)(\zeta_n^* \oplus (\epsilon^1_n)^*))
	\ar[r]^-{(m+1)i} \ar[d]_-{m \id \oplus j \oplus \id}
	& Th((m+1)\zeta_{n+1}^*) \ar[d]^-{m \id \oplus j} \\
Th(m(\zeta_n^* \oplus (\epsilon^1_n)^*) \oplus \epsilon_n^*
	\oplus (\epsilon^1_n)^*) \ar[r]^-{m i \oplus i}
	& Th(m\zeta_{n+1}^* \oplus \epsilon_{n+1}^*)
}
$$
and
$$
\xymatrix{
Th((m+1) \zeta_n^* \oplus (m+1)(\epsilon^1_n)^*)
	\ar[r]^-{\sh} \ar[d]_-{m \id \oplus j \oplus \id}
    & Th((m+1) (\zeta_n^* \oplus (\epsilon^1_n)^*))
	\ar[dd]^-{m \id \oplus j \oplus \id} \\
Th(m \zeta_n^* \oplus \epsilon_n^* \oplus (m+1)(\epsilon^1_n)^*)
	\ar[d]_-{\id \oplus \chi_{m,n} \oplus \id}^-{\cong} & \\
Th(m \zeta_n^* \oplus m (\epsilon^1_n)^*
    \oplus \epsilon_n^* \oplus (\epsilon^1_n)^*)
	\ar[r]^-{\sh \oplus \id}
    & Th(m (\zeta_n^* \oplus (\epsilon^1_n)^*) \oplus
	\epsilon_n^* \oplus (\epsilon^1_n)^*)
}
$$
commute strictly, where $\chi_{m,n}$ is induced by the symmetry isomorphism
$$
\epsilon_n^* \oplus m(\epsilon^1_n)^*
	\cong m(\epsilon^1_n)^* \oplus \epsilon_n^* \,,
$$
hence is homotopic to multiplication by $(-\epsilon)^{mn}$.  Applying
$$
\Sigma^{-2(m+1)(n+1), -(m+1)(n+1)}
$$
yields the stated homotopy commutative rectangle.
\end{proof}

\begin{corollary} \label{cor:i4-j-square}
The square
$$
\xymatrix{
L_{-2m-2}^{2n-2m-3} = Th(-(m+1)\gamma_n^*) \ar[r]^-{i^4} \ar[d]_-{j}
	& Th(-(m+1)\gamma_{n+4}^*) = L_{-2m-2}^{2n-2m+5} \ar[d]^-{j} \\
L_{-2m}^{2n-2m-1} = Th(-m\gamma_n^*) \ar[r]^-{i^4}
	& Th(-m\gamma_{n+4}^*) = L_{-2m}^{2n-2m+7}
}
$$
commutes up to homotopy.
\end{corollary}

\begin{proof}
This follows from Lemma~\ref{lem:i-j-rectangle}, since $mn + m(n+1) +
m(n+2) + m(n+3)$ is always even.
\end{proof}

\begin{definition} \label{def:infinite-lens-spectra}
Let the (infinite, motivic) stunted lens spectrum
$$
L_{-2m}^\infty = \hocolim_n L_{-2m}^{2n-2m-1}
	= \hocolim_n Th(-m\gamma_n^*)
$$
be the homotopy colimit of the maps
$$
\dots \to L_{-2m}^{2n-2m-1} \overset{i}\longto L_{-2m}^{2n-2m+1}
	\to \dots \,.
$$
For a fixed choice of commuting homotopies in
Corollary~\ref{cor:i4-j-square}, let $j \: L_{-2m-2}^\infty \longto
L_{-2m}^\infty$ be the induced map.  Let
$$
L_{-\infty}^\infty = \holim_m L_{-2m}^\infty
$$
be the homotopy limit of the resulting tower
$$
\dots \to L_{-2m-2}^\infty \overset{j}\longto L_{-2m}^\infty
	\to \dots \,.
$$
\end{definition}

Recall Notation~\ref{not:H-of-Bmu-Bsigma}.

\begin{lemma} \label{lem:cohomology-lens}
$$
H^{*,*}(L^{2n-1}) = H^{*,*}[u,v]/(u^2 = \tau v + \rho u, v^n) \,,
$$
where $v$ is the mod~$\ell$ Euler class of $\gamma_n^* \downarrow
L^{2n-1}$ and $\beta(u) = v$.
\end{lemma}

\begin{proof}
This follows by the same argument as for~\cite{Voe03a}*{Thm.~6.10},
working with $L^{2n-1} \to \bP^{n-1}$ in place of $B\mu_\ell \to
\bP^\infty$.
\end{proof}

\begin{definition}
Let $U_{m\zeta_n^*} \in H^{2m(n-1),m(n-1)}(Th(m\zeta_n^*))$ be
the mod~$\ell$ Thom class of $m\zeta_n^* \downarrow L^{2n-1}$.
Let
$$
U_{-m\gamma_n^*} \in H^{-2m,-m}(L_{-2m}^{2n-2m-1})
	= H^{-2m,-m}(Th(-m\gamma_n^*))
$$
be its image under the (de-)suspension isomorphism.
We write $x \mapsto x \cdot U_{m\zeta_n^*}$ and
$x \mapsto x \cdot U_{-m\gamma_n^*}$ for the Thom isomorphisms
\begin{align*}
H^{*,*}(L^{2n-1}) &\cong H^{*+2m(n-1),*+m(n-1)}(Th(m\zeta_n^*)) \\
H^{*,*}(L^{2n-1}) &\cong H^{*-2m,*-m}(Th(-m\gamma_n^*)) \,,
\end{align*}
cf.~\cite{Voe03a}*{Prop.~4.3}.
\end{definition}

\begin{lemma}
The homomorphisms
\begin{align*}
i^* \: H^{*,*}(Th(-m\gamma_{n+1}^*))
	&\longto H^{*,*}(Th(-m\gamma_n^*)) \\
j^* \: H^{*,*}(Th(-m\gamma_n^*))
	&\longto H^{*,*}(Th(-(m+1)\gamma_n^*))
\end{align*}
are given by
\begin{align*}
i^*(x \cdot U_{-m\gamma_{n+1}^*}) &= i^*(x) \cdot U_{-m\gamma_n^*} \\
j^*(x \cdot U_{-m\gamma_n^*}) &= x v \cdot U_{-(m+1)\gamma_n^*} \,.
\end{align*}
\end{lemma}

\begin{proof}
The Thom class of $m\zeta_{n+1}^*$ maps under~$i^*$ to the Thom
class of $i^*(m\zeta_{n+1}^*) = m(\zeta_n^* \oplus (\epsilon^1_n)^*)$,
which corresponds under the suspension isomorphism to the Thom class
of $m\zeta_n^*$.  This proves the first formula, where $x \in
H^{*,*}(L^{2n+1})$.

The Thom class of $m\zeta_n^*$ corresponds under the suspension isomorphism
to the Thom class of $m\zeta_n^* \oplus \epsilon_n^*$.  By the Jouanolou
trick~\cite{Voe03a}*{Lem.~4.7} it maps under~$j^*$ to the Euler class~$v$
of~$\gamma_n^*$ times the Thom class of $(m+1)\zeta_n^*$.  This proves
the second formula, where $x \in H^{*,*}(L^{2n-1})$.
\end{proof}

\begin{proposition} \label{prop:HcLinfty-vs-HBmuloc}
The structure maps $L_{-2m}^{2n-2m-1} \to L_{-2m}^\infty$ and
$L_{-\infty}^\infty \to L_{-2m}^\infty$ induce $\sA$-module isomorphisms
$$
H^{*,*}(L_{-2m}^\infty) \cong H^{*,*}(B\mu_\ell) \{ U_{-m\gamma^*} \}
	\cong H^{*,*}(B\mu_\ell) \{ v^{-m} \}
$$
and
$$
H_c^{*,*}(L_{-\infty}^\infty) = \colim_m H^{*,*}(L_{-2m}^\infty)
	\cong H^{*,*}(B\mu_\ell)_{\loc} \,.
$$
\end{proposition}

\begin{proof}
Since each $i^* \: H^{*,*}(Th(-m\gamma_{n+1}^*)) \to
H^{*,*}(Th(-m\gamma_n^*))$ is surjective, the $\lim$-$\limone$ sequence
gives an isomorphism
$$
H^{*,*}(L_{-2m}^\infty) \cong \lim_n H^{*,*}(Th(-m\gamma_n^*))
	\cong \lim_n H^{*,*}(L^{2n-1}) \{U_{-m\gamma_n^*}\} \,.
$$
Letting $U_{-m\gamma^*}$ correspond to the compatible
sequence $(U_{-m\gamma_n^*})_n$ gives the first isomorphism.
The second isomorphism sends~$U_{-m\gamma^*}$ to $v^{-m}$.
The induced homomorphism
$j^* \: H^{*,*}(B\mu_\ell) \{ U_{-m\gamma^*} \}
\to H^{*,*}(B\mu_\ell) \{ U_{-(m+1)\gamma^*} \}$
maps $U_{-m\gamma^*}$ to $v \cdot U_{-(m+1)\gamma^*}$, hence
corresponds to the homomorphism
$$
H^{*,*}(B\mu_\ell) \{ v^{-m} \}
	\longto H^{*,*}(B\mu_\ell) \{ v^{-m-1} \}
$$
sending $v^{-m}$ to $v \cdot v^{-m-1}$.  It follows that
$\colim_m H^{*,*}(L_{-2m}^\infty)$, i.e., the continuous cohomology
$H_c^{*,*}(L_{-\infty}^\infty)$, is isomorphic to the localization
$H^{*,*}(B\mu_\ell) [1/v] = H^{*,*}(B\mu_\ell)_{\loc}$.

It remains to justify that these isomorphisms are compatible with the
Steenrod operations.  The short exact sequence
$$
0 \to H^{*,*}(Th(m\gamma_n^*)) \longto H^{*,*}(L^{2n+2m-1})
	\overset{i^{n*}}\longto H^{*,*}(L^{2m-1}) \to 0
$$
induced from~\eqref{eq:stuntedlens} shows that the Steenrod action on
$U_{m\gamma_n^*}$ matches that on~$v^m$ in $H^{*,*}(L^{2n-1}) \{v^m\}$.
By another application of the Jouanolou trick, and the Cartan formula,
it follows that the Steenrod action on $U_{m\zeta_n^*}$ is compatible
with that on $\Sigma^{2mn,mn} v^{-m}$.  Hence, by stability, the
Steenrod action on $U_{-m\gamma_n^*}$ matches that on $v^{-m}$ in
$H^{*,*}(L^{2n-1}) \{v^{-m}\}$.  Passing to the limit over~$n$ and the
colimit over~$m$ completes the argument.
\end{proof}

The next two lemmas confirm the assumptions required (for recognition of
the $E_2$-term and conditional convergence) of the delayed limit Adams
spectral sequence for $\Sigma L_{-\infty}^\infty$.

\begin{lemma} \label{lem:Thom-cellular}
The spectra $L_{-2m}^{2n-2m-1}$ are cellular of finite type.
The spectra $L_{-2m}^\infty$ are cellular and bounded below.
\end{lemma}

\begin{proof}
The Zariski cover of $L^{2n-1}$ by the affines $(\bA^{i-1} \times
\bA^1_0 \times \bA^{n-i})/\mu_\ell$, with $1 \le i \le n$, is
completely stably cellular in the sense of~\cite{DI05}*{Def.~3.7}.
It trivializes~$\gamma_n$, hence also~$m \zeta_n^*$.  It follows as
in~\cite{DI05}*{Cor.~3.10} that $L_{-2m}^{2n-2m-1} = \Sigma^{-2mn,-mn}
Th(m \zeta_n^*)$ is cellular.  Inspection of the argument shows that
it admits a cell structure with finitely many cells, all in bidegrees
$(p_\alpha, q_\alpha)$ satisfying $p_\alpha - q_\alpha \ge -m$.

Since this bound is uniform, it follows by passage to the homotopy
colimit that $L_{-2m}^\infty$ is also cellular, with cells in the same
range of bidegrees.  Hence $L_{-2m}^\infty$ lies in $SH(S)_{\ge-m}$
of the homotopy $t$-structure.
\end{proof}

Recall Definitions~\ref{def:motivic-GEM} and~\ref{def:bifinite-type}.

\begin{lemma} \label{lem:lens-GEMs-bifinite}
The $H$-module spectra $H \wedge L_{-2m}^{2n-2m-1}$ and $H \wedge
L_{-2m}^\infty$ are motivic GEMs of bifinite type.
\end{lemma}

\begin{proof}
These spectra are $H$-cellular by Lemma~\ref{lem:Thom-cellular}.
The homology version of Lemma~\ref{lem:cohomology-lens} shows that
$H_{*,*}(L^{2n-1})$ is finitely generated and free over $H_{*,*}$
on generators in bidegrees $(i+2k,i+k)$ for $i \in \{0,1\}$ and $0
\le k < n$.  It then follows from the Thom isomorphism in motivic
homology that $H_{*,*}(L_{-2m}^{2n-2m-1})$ is finitely generated and
free on similar generators for $i \in \{0,1\}$ and $-m \le k < n-m$,
and that $H_{*,*}(L_{-2m}^\infty)$ is free on one generator in each
bidegree~$(i+2k,i+k)$ for $i \in \{0,1\}$ and $k \ge -m$.  In particular,
$H_{*,*}(L_{-2m}^\infty)$ is of bifinite type.
\end{proof}

Finally, we construct maps
\begin{equation} \label{eq:PT-maps}
\xymatrix{
\bS \ar[r]^-{c} & \Sigma^{2,1} \bP_{-\infty}^{-1} \ar[r]^-{d} \ar[d]
	& \Sigma L_{-\infty}^{-1} \ar[d] \\
	& \Sigma^{2,1} \bP_{-\infty}^\infty
	& \Sigma L_{-\infty}^\infty \,,
}
\end{equation}
whose composite induces the (large) residue homomorphism from
Definition~\ref{def:residues} in motivic cohomology.

To define $\bP_{-\infty}^\infty$, let the (finite, motivic) stunted
projective spectrum $\bP_{-m}^{n-m-1} = Th(-m \gamma_n^* \downarrow
\bP^{n-1})$ be the Thom spectrum of the negative of $m \gamma_n^* =
\sO(1)^m$ over $\bP^{n-1}$.  We have maps $i \: \bP_{-m}^{n-m-1} \to
\bP_{-m}^{n-m}$ and $j \: \bP_{-m-1}^{n-m-2} \to \bP_{-m}^{n-m-1}$ as in
Definition~\ref{def:i-j-maps-for-lens-spectra}, and let $\bP_{-m}^\infty
= \hocolim_n \bP_{-m}^{n-m-1}$ and $\bP_{-\infty}^\infty = \holim_m
\bP_{-m}^\infty$ as in Definition~\ref{def:infinite-lens-spectra}.
We obtain $H_c^{*,*}(\bP_{-\infty}^\infty) \cong H^{*,*}[v^{\pm1}]$,
by the same arguments as for lens spectra.

\begin{proposition} \label{prop:map-c}
There is a map $c \: \bS \to \Sigma^{2,1} \bP_{-\infty}^{-1}$ inducing
$$
\Sigma^{2,1} v^{-1} \longmapsto 1
$$
in cohomology.  The $H^{*,*}$-module generators $\Sigma^{2,1} v^k$
for $k \le -2$ map to zero.
\end{proposition}

\begin{proof}
We use that $\bP^{n-1}$ is a smooth projective variety, with
stable normal bundle
$$
\nu = \epsilon^1_n - n \gamma_n^* = \sO - \sO(1)^n \,.
$$
By the construction leading to algebraic Atiyah duality,
see~\cite{Voe03b}*{Prop.~2.7}, \cite{Hu05}*{Cl.~2}
and~\cite{Hoy17}*{\S5.3}, there is a Pontryagin--Thom collapse map
$$
c_n \: \bS \longto Th(\nu \downarrow \bP^{n-1})
	= \Sigma^{2,1} \bP_{-n}^{-1}
$$
inducing the homomorphism $\Sigma^{2,1} v^{-1} \mapsto 1$ in cohomology.
The generators $\Sigma^{2,1} v^k$ for $-n \le k \le -2$ map to zero
for bidegree reasons.  When combined with the Thom diagonal and an
adjunction, this leads to the Atiyah duality equivalence $Th(\nu
\downarrow \bP^{n-1}) \simeq D(\bP^{n-1}_+) = F(\bP^{n-1}_+, \bS)$,
under which~$c_n$ is functionally dual to the collapse map $\bP^{n-1}_+
\to S^0$.  In particular, these maps are compatible up to homotopy for
varying~$n$, and combine to define a map~$c \: \bS \to \Sigma^{2,1}
\bP_{-\infty}^{-1} \simeq D(\bP^\infty_+)$, as required.
\end{proof}

Since $L^{2n-1}$ is only quasi-projective, we need a different
method to obtain the second map of~\eqref{eq:PT-maps}.

\begin{proposition} \label{prop:map-d}
There is a map $d \: \Sigma^{2,1} \bP_{-\infty}^{-1} \to \Sigma
L_{-\infty}^{-1}$ inducing
$$
\Sigma u v^{-1} \longmapsto \Sigma^{2,1} v^{-1}
$$
in cohomology, modulo $H^{*,*}$-multiples of $\Sigma^{2,1} v^k$
for $k \le -2$.
\end{proposition}

\begin{proof}
For algebraic vector bundles $\alpha, \beta \downarrow X$ over the same
smooth scheme~$X$, where $\beta$ may be virtual, there is a homotopy
cofiber sequence
$$
Th(p^* \beta \downarrow E_0(\alpha))
	\longto Th(\beta \downarrow X)
	\longto Th(\alpha \oplus \beta \downarrow X)
$$
of motivic spectra.  Here $p \: E_0(\alpha) \to X$ denotes the
projection, and $p^* \beta$ is the pullback of $\beta$ along~$p$.
We apply this with $X = \bP^{n-1}$, $\alpha = \gamma_n^{\otimes \ell}
= \sO(-\ell)$ and $\alpha \oplus \beta = \nu$, the stable normal
bundle of~$\bP^{n-1}$.  We identify $E_0(\gamma_n^{\otimes \ell})
\cong L^{2n-1}$, as in~\cite{Voe03a}*{Lem.~6.3}.  Moreover, $p^*
\gamma_n^{\otimes \ell} \cong \epsilon^1_n$ over $L^{2n-1}$, i.e.,
$\sO(-\ell)$ and $\sO$ pull back to the same bundle, which implies
that $p^*\beta \cong -n \gamma_n^*$ as a stable bundle
over $L^{2n-1}$.  This leads to the homotopy cofiber sequence
$$
L_{-2n}^{-1}
	\longto Th(\beta \downarrow \bP^{n-1})
	\longto \Sigma^{2,1} \bP_{-n}^{-1}
	\overset{d_n}\longto \Sigma L_{-2n}^{-1} \,.
$$
The long exact sequence in cohomology shows that the connecting map $d_n$
induces a homomorphism mapping $\Sigma u v^{-1}$ to $\Sigma^{2,1} v^{-1}$,
modulo $H^{*,*}$-multiples of $\Sigma^{2,1} v^k$ for $-n \le k \le -2$.
Again, these maps are compatible up to homotopy for varying~$n$, and
combine to define the required map~$d$.
\end{proof}

\begin{proposition} \label{prop:map-e}
There is a map $e \: \bS \to \Sigma L_{-\infty}^\infty$ of motivic
spectra, inducing the residue homomorphism
$$
e^* = \res \: \Sigma H^{*,*}(B\mu_\ell)_{\loc}
	\longto H^{*,*}
$$
in cohomology.
\end{proposition}

\begin{proof}
We take $e$ to be $dc$ followed by the inclusion $\Sigma L_{-\infty}^{-1}
\to \Sigma L_{-\infty}^\infty$ (of the homotopy fiber of $\Sigma
L_{-\infty}^\infty \to \Sigma L_0^\infty$).  To check that $e$ induces
the residue homomorphism, we use Corollary~\ref{cor:res-Ext-eqs}
in cohomological degree~$0$, giving an isomorphism
$$
\res \: \Hom_{\sA}(H^{*,*}, H^{*,*})
	\overset{\cong}\longto
	\Hom_{\sA}(\Sigma H^{*,*}(B\mu_\ell)_{\loc}, H^{*,*}) \,.
$$
In other words, any $\sA$-module homomorphism
$\Sigma H^{*,*}(B\mu_\ell)_{\loc} \to H^{*,*}$ is
characterized by its value on $\Sigma u v^{-1}$.
Since $e^*$ and $\res$ agree on this element, they are equal.
\end{proof}

\section{The motivic Lin and Gunawardena theorems}
\label{sec:motivic-lin-gunawardena}

We can now prove a motivic refinement of the classical theorems of
Lin~\cite{Lin80} (for $\ell=2$) and Gunawardena~\cite{Gun81} (for $\ell$
an odd prime).

Recall that $\mu_\ell$ denotes the algebraic group of $\ell$-th roots of
unity, $L^{2n-1} = (\bA^n \setminus \{0\})/\mu_\ell$ is an algebraic lens
space, $\gamma_n^* \downarrow L^{2n-1}$ is the dual of the tautological
line bundle, $L_{-2m}^{2n-2m-1} = Th(-m \gamma_n^*)$ is a stunted
lens spectrum, and $L_{-2m}^\infty = \hocolim_n L_{-2m}^{2n-2m-1}$ and
$L_{-\infty}^\infty = \holim_m L_{-2m}^\infty$ are infinite lens spectra.
The continuous mod~$\ell$ cohomology $H_c^{*,*}(L_{-\infty}^\infty) =
\colim_m H^{*,*}(L_{-2m}^\infty)$ is isomorphic to the localization
$H^{*,*}(B\mu_\ell)_{\loc} = H^{*,*}[u, v^{\pm1}]/(u^2 = \tau v +
\rho u)$.  The map $e \: \bS \to \Sigma L_{-\infty}^\infty$ induces the
$\Ext$-equivalence $\res \: \Sigma H^{*,*}(B\mu_\ell)_{\loc} \to H^{*,*}$.

\begin{theorem} \label{thm:mot-lin-gun}
Let $S$ be a Noetherian scheme of finite dimension~$d$, essentially
smooth over a field or Dedekind domain, and let $\ell$ be a prime that
is invertible on~$S$.  The $(\ell, \eta)$-completed map
$$
e^\wedge_{\ell,\eta} \: \bS^\wedge_{\ell,\eta}
	\longto (\Sigma L_{-\infty}^\infty)^\wedge_{\ell,\eta}
$$
is a $\pi_{*,*}$-isomorphism.  If $S = \Spec k$ for $k$ a field, then
$e^\wedge_{\ell,\eta}$ is a motivic equivalence.
\end{theorem}

\begin{proof}
We apply Proposition~\ref{prop:comparison} with $X = \bS$, $Y = \Sigma
L_{-\infty}^\infty$, $Y(m) = \Sigma L_{-2m}^\infty$ and $g = e$.
$$
\bS \overset{e}\longto \Sigma L_{-\infty}^\infty \to \dots \to
	\Sigma L_{-2m}^\infty \to \dots \,.
$$
The $H$-modules $H \wedge \bS$ and $H \wedge \Sigma L_{-2m}^\infty$ are
motivic GEMs of bifinite type by Lemma~\ref{lem:lens-GEMs-bifinite}.
Moreover, $\bS$ and each $\Sigma L_{-2m}^\infty$ is
bounded below in the homotopy $t$-structure on $SH(S)$ by
Lemma~\ref{lem:Thom-cellular}.  By Theorem~\ref{thm:mantovani} the
Adams spectral sequences for~$\bS$ and the $\Sigma L_{-2m}^\infty$
are conditionally convergent to the $(\ell,\eta)$-adic completions.
By Proposition~\ref{prop:del-lim-cond-conv} the delayed limit Adams
spectral sequence for $\Sigma L_{-\infty}^\infty$ is also conditionally
convergent to the $(\ell,\eta)$-adic completion.  The $\sA$-module
homomorphism
$$
e^* \: H_c^{*,*}(\Sigma L_{-\infty}^\infty)
	\longto H^{*,*}(\bS)
$$
agrees with
$$
\res \: \Sigma H^{*,*}(B\mu_\ell)_{\loc} \longto H^{*,*}
$$
via the isomorphism of Proposition~\ref{prop:HcLinfty-vs-HBmuloc},
by Proposition~\ref{prop:map-e}.  Finally, $\res$ is an
$\Ext$-equivalence by Corollary~\ref{cor:res-Ext-eqs}.
Hence the induced map of spectral sequences
\begin{multline*}
e \: E_2^{s,*,*}(\bS)
	= \Ext_{\sA}^{s,*,*}(H^{*,*}, H^{*,*}) \\
\overset{\cong}\longto
	{}^{\del} E_2^{s,*,*}(\Sigma L_{-\infty}^\infty)
	= \Ext_{\sA}^{s,*,*}(\Sigma H^{*,*}(B\mu_\ell)_{\loc}, H^{*,*})
\end{multline*}
is an isomorphism, from the $E_2$-term and onward, and the
map of abutments
$$
e \: \pi_{*,*}(\bS^\wedge_{\ell,\eta})
	\overset{\cong}\longto
\pi_{*,*}((\Sigma L_{-\infty}^\infty)^\wedge_{\ell,\eta})
$$
is an isomorphism of (filtered) bigraded abelian groups.

Let $Ce^\wedge_{\ell,\eta}$ denote the homotopy cofiber
of~$e^\wedge_{\ell,\eta}$.  If $S = \Spec k$ for a field~$k$, then the
homotopy sheaves $\underline{\pi}_{*,*}(Ce^\wedge_{\ell,\eta})$ are
pure in the sense of~\cite{Mor05}*{Def.~6.4.9}, hence unramified in
the sense of~\cite{Mor12}*{Def.~2.1}, by~\cite{Mor05}*{Lem.~6.4.11}
and~\cite{Mor12}*{Thm.~1.9}.  For any (irreducible) smooth
$k$-scheme~$U$ with function field~$k(U)$, the vanishing of
$\pi_{*,*}(Ce^\wedge_{\eta,\ell})$ over $\Spec k(U)$ implies
the vanishing of $\pi_{*,*}(Ce^\wedge_{\eta,\ell})$ over~$U$,
so that $\underline{\pi}_{*,*}(Ce^\wedge_{\eta,\ell}) = 0$ and
$e^\wedge_{\eta,\ell}$ is a motivic equivalence.  This application of
Morel's theorems also appears in~\cite{Hem22}*{Prop.~4}.
\end{proof}

\begin{remark}
As an alternative to our fairly explicit construction of~$e \: \bS \to
\Sigma L_{-\infty}^\infty$ in Proposition~\ref{prop:map-e}, one might
appeal to the weight~$0$ part of the delayed limit Adams spectral sequence
$$
{}^{\del} E_2^{s,t,0}
= \Ext_{\sA}^{s,t,0}(\Sigma H^{*,*}(B\mu_\ell)_{\loc}, H^{*,*})
	\Longrightarrow_s
	\pi_{t-s,0}((\Sigma L_{-\infty}^\infty)^\wedge_{\ell,\eta})
$$
for $\Sigma L_{-\infty}^\infty$ to show the existence of a homotopy class
$e' \in \pi_{0,0}((\Sigma L_{-\infty}^\infty)^\wedge_{\ell,\eta})$
detected in Adams filtration $s=0$ by $\res \: \Sigma
H^{*,*}(B\mu_\ell)_{\loc} \to H^{*,*}$ in
$$
{}^{\del} E_\infty^{0,0,0} \subset {}^{\del} E_2^{0,0,0}
	= \Hom_{\sA}(\Sigma H^{*,*}(B\mu_\ell)_{\loc}, H^{*,*}) \,.
$$
By Corollary~\ref{cor:res-Ext-eqs}
$$
{}^{\del} E_2^{s,t,0} \cong
	\Ext_{\sA}^{s,t,0}(H^{*,*}, H^{*,*}) \,,
$$
and $\res$ corresponds to $\id \: H^{*,*} \to H^{*,*}$ in
$\Hom_{\sA}(H^{*,*}, H^{*,*})$.

If $d < \ell-1$, which is always the case for $S = \Spec k$, the canonical
(or normalized cobar) $\sA$-module resolution of $H^{*,*}$ shows that
${}^{\del} E_2^{s,t,0} \cong \Ext_{\sA}^{s,t,0}(H^{*,*}, H^{*,*}) = 0$
whenever $t-s = -1$.  Hence $d_r^{0,0,0} = 0$ for all~$r\ge2$, so that
${}^{\del} E_\infty^{0,0,0} = {}^{\del} E_2^{0,0,0}$ and $\res$ is an
infinite cycle.

To ensure that $\res$ detects a homotopy class, we also need
strong convergence in its bidegree.  By~\cite{Boa99}*{Thm.~7.3}, see
also~\cite{HR19}*{Thm.~3.9}, it suffices to know that $RE_\infty^{s,t,0}
= 0$ whenever $t-s = +1$.  By~\cite{KW19}*{Cor.~6.1} this condition
is satisfied for $S = \Spec k$, subject to the additional hypothesis
for $\ell=2$ that $k$ has finite virtual cohomological dimension.
The class $e'$ is then represented by a map $e' \: \bS \to (\Sigma
L_{-\infty}^\infty)^\wedge_{\ell,\eta}$, which can be used in place
of~$e$.
\end{remark}

\end{document}